\documentclass{amsart}

\usepackage{amsmath,amscd}
\usepackage{amsthm}
\usepackage{amssymb}
\usepackage{setspace}
\usepackage{fancyhdr}
\usepackage{stmaryrd}
\usepackage{bbm}
\input xy
\xyoption{all}
\usepackage[unicode]{hyperref}
\usepackage{color}

\usepackage[normalem]{ulem}
\newtheorem{thm}{{Theorem}}[section]
\newtheorem{cor}[thm]{{Corollary}}
\newtheorem{lemma}[thm]{{Lemma}}
\newtheorem{prop}[thm]{Proposition}

\theoremstyle{definition}
\newtheorem{defn}[thm]{Definition}
\newtheorem{rk}[thm]{{Remark}}

\renewcommand{\phi}{\varphi}

\newcommand{\N}{\mathbb N}

\newcommand{\Z}{\mathbb Z}

\newcommand{\NN}{\mathbb N}
\newcommand{\QQ}{\mathbb{Q}}
\newcommand{\ZZ}{\mathbb{Z}}
\newcommand{\VV}{\mathbb{V}}
\newcommand{\DD}{\mathbb{D}}
\newcommand{\BB}{\mathbb{B}}

\newcommand{\Zp}{\mathbb{Z}_p}
\newcommand{\Qp}{\mathbb{Q}_p}
\newcommand{\Cp}{\mathbb{C}_p}
\newcommand{\Fp}{\mathbb{F}_p}

\newcommand{\FF}{\mathbb{F}}

\newcommand{\LL}{\mathbb{L}}

\newcommand{\TT}{\mathbb{T}}

\newcommand{\cF}{\mathcal{F}}
\newcommand{\cO}{\mathcal{O}}
\newcommand{\cG}{\mathcal{G}}
\newcommand{\cM}{\mathcal{M}}
\newcommand{\cX}{\mathcal{X}}
\newcommand{\cP}{\mathcal{P}}

\newcommand{\cY}{\mathcal{Y}}

\newcommand{\cT}{\mathcal{T}}
\newcommand{\cE}{\mathcal{E}}

\newcommand{\cU}{\mathcal{U}}
\newcommand{\cV}{\mathcal{V}}

\newcommand{\ka}{\kappa}
\newcommand{\ra}{\rightarrow}
\newcommand{\lra}{\longrightarrow}

\newcommand{\hra}{\hookrightarrow}

\newcommand{\ur}{\mathrm{ur}}

\DeclareMathOperator{\Spec}{Spec}
\DeclareMathOperator{\Sp}{Sp}

\DeclareMathOperator{\Spa}{Spa}

\DeclareMathOperator{\Spf}{Spf}

\DeclareMathOperator{\Gal}{Gal}

\DeclareMathOperator{\Fil}{Fil}

\newcommand{\bk}{\overline{k}}

\newcommand{\xto}[1][]{\xrightarrow{#1}}
\newcommand{\simto}{%\stackrel{\sim}{\to}}%isomorphic to
\xto[\sim]} %isomorphic to;

\newcommand{\BBdrp}{\mathbb{B}_{\mathrm{dR}}^+}
\newcommand{\BBcrp}{\mathbb{B}_{\mathrm{cris}}^+}

\newcommand{\BBcr}{\mathbb{B}_{\mathrm{cris}}}
\newcommand{\AAcr}{\mathbb{A}_{\mathrm{cris}}}
\newcommand{\AAinf}{\mathbb{A}_{\mathrm{inf}}}
\newcommand{\BBinf}{\mathbb{B}_{\mathrm{inf}}}

\newcommand{\BBdr}{\mathbb{B}_{\mathrm{dR}}}

\newcommand{\proet}{\mathrm{pro\acute{e}t}}

\newcommand{\et}{\mathrm{\acute{e}t}}

\newcommand{\ho}{\widehat{\otimes}}

\newcommand{\dR}{\mathrm{dR}}
\newcommand{\cris}{\mathrm{cris}}
\newcommand{\ad}{\mathrm{ad}}
\newcommand{\an}{\mathrm{an}}

\numberwithin{equation}{subsection}

\begin{document}

\author{Fucheng Tan, Jilong Tong}
\title{Crystalline comparison isomorphisms in $p$-adic Hodge theory: the absolutely unramified case}

\date{}

\maketitle

\markright{Crystalline comparison isomorphism in $p$-adic Hodge theory}

\thispagestyle{plain}

\begin{abstract}

We construct the crystalline comparison isomorphisms for proper smooth formal schemes over an absolutely unramified base. Such isomorphisms hold for \'etale cohomology with nontrivial coefficients, as well as  in the relative setting, i.e. for proper smooth morphisms of smooth formal schemes. The proof is formulated in terms of the pro-\'etale topos introduced by Scholze, and uses his primitive comparison theorem for the structure sheaf on the pro-\'etale site. Moreover, we need to prove the Poincar\'e lemma for crystalline period sheaves, for which we adapt the idea of Andreatta and Iovita. Another ingredient for the proof is the geometric acyclicity of crystalline period sheaves, whose computation is due to Andreatta and Brinon.
\end{abstract}

\tableofcontents

\section*{Notation}

\begin{itemize}

\item Let $p$ be a prime number.

\item Let $k$ be a $p$-adic field, i.e., a discretely valued complete nonarchimedean extension of $\Qp$, whose residue field $\kappa$ is a perfect field of characteristic $p$.  (We often assume $k$ to be  absolutely unramified in this paper.)

\item Let  $\bk$ be a fixed algebraic closure of $k$.  Set $\Cp:=\widehat{\bk}$ the $p$-adic completion of $\bk$. The $p$-adic valuation $v$ on $\Cp$ is normalized so that $v(p)=1$. Write the absolute Galois group $\Gal(\bk/k)$ as $G_k$.

\item For a (commutative unitary)  ring $A$, let $A\langle T_1,\ldots, T_d\rangle$ be the PD-envelope of the polynomial ring $A[T_1,\ldots, T_d]$ with respect to the ideal $(T_1,\ldots, T_d)\subset A[T_1,\ldots, T_d]$ (with the requirement that the PD-structure be compatible with the one on the ideal $(p)$) and then let $A\{\langle T_1,\ldots, T_d\rangle\}$ be its $p$-adic completion.

\item We use the symbol $\simeq$ to denote canonical isomorphisms, and sometimes quasi-isomorphisms.  The symbol $\approx$ is frequently used for almost isomorphisms with respect to some almost-setting that will be fixed later.

 \end{itemize}

\section{Introduction}

Let $k$ be a discretely valued complete nonarchimedean field over $\Qp$, which is absolutely unramified.

Consider a rigid analytic variety over $k$, or more generally  an adic space $X$ over $\Spa(k,\cO_k)$ which admits a proper smooth  formal model $\cX$ over $\Spf{\cO_k}$, whose special fiber is denoted by $\cX_0$. Let $\LL$ be a lisse $\Z_p$-sheaf on $X_{\et}$.
On the one hand, we have the $p$-adic \'etale cohomology $H^i(X_{\bk}, \LL)$ which is a finitely generated $\Zp$-module carrying a continuous $G_k=\Gal(\bk/k)$-action. On the other hand,  one may consider the crystalline cohomology $H_{\cris}^i(\cX_0/\cO_k, \cE)$ with the coefficient $\cE$ being a filtered (convergent) $F$-isocrystal on $\cX_0/\cO_k$. At least in the case that $X$ comes from a scheme and the coefficients $\LL$ and $\cE$ are trivial, it was Grothendieck's problem of the \emph{mysterious functor} to find a comparison between the two cohomology theories. This problem was later formulated as the \emph{crystalline conjecture} by Fontaine \cite{Fon82}.

In the past decades, the crystalline conjecture was proved in various generalities,  by  Fontaine-Messing, Kato, Tsuji, Niziol,  Faltings, Andreatta-Iovita,  Beilinson and Bhatt. Among them, the first proof for the whole conjecture was given by Faltings \cite{Fal}. Along this line, Andreatta-Iovita introduced   the Poincar\'e lemma for the crystalline period sheaf $\BBcr$ on the Faltings site, a sheaf-theoretic   generalization of  Fontaine's period ring $B_{\cris}	$. Both the approach of Fontaine-Messing and that of Faltings-Andreatta-Iovita use an intermediate topology, namely the syntomic topology and the Faltings topology, respectively. The approach of Faltings-Andreatta-Iovita, however, has the advantage that it works for nontrivial coefficients $\LL$ and $\cE$.

More recently, Scholze \cite{Sch} introduced the pro-\'etale site $X_{\proet}$, which allows him to construct the de Rham comparison isomorphism for any proper smooth adic space over  a discretely valued complete nonarchimedean field over  $\Qp$, with coefficients being lisse $\Z_p$-sheaves on $X_{\proet}$. (The notion of lisse $\Zp$-sheaf on $X_{\et}$ and that   on $X_{\proet}$ are equivalent.) Moreover, his approach is direct and flexible enough to attack the relative version of the de Rham comparison isomorphism, i.e. the comparison for a proper smooth morphism between two smooth adic spaces.

It seems  that to deal with nontrivial coefficients in a comparison isomorphism, one is forced to work over analytic bases. For the generality and some technical advantages provided by the pro-\'etale topology, we adapt Scholze's approach to give a proof of the crystalline conjecture for    proper smooth formal schemes over $\Spf{\cO_k}$, with nontrivial coefficients, in both  absolute and relative settings. Meanwhile, we point out that the method adopted in our proof is rather different from that in \cite{BMS}, in which the authors develop a new cohomology that allows them  to prove a strong integral comparison theorem (for trivial coefficients).

Let us explain our construction of crystalline comparison isomorphism (in the absolutely unramified case) in more details.
First of all, Scholze is able to prove the finiteness of the \'etale cohomology of  a proper smooth  adic space over $\mathbb C_p=\widehat{\bk}$ with coefficient $\LL'$ being an $\FF_p$-local system. Consequently, he shows the following ``primitive comparison", an almost  (with respect to the maximal ideal of $\cO_{\mathbb C_p}$) isomorphism
\[
H^i(X_{\mathbb C_p,\et}, \LL')\otimes_{\Fp} \cO_{\mathbb C_p}/p\stackrel{\approx}{\lra} H^i(X_{\mathbb C_p,\et}, \LL'\otimes_{\Fp} \cO_{X}^+/p).
\]
With some more efforts, one can produce the primitive comparison isomorphism in the crystalline case:
\begin{thm}[see Theorem \ref{inout}] For $\LL$ a lisse $\Zp$-sheaf on $X_{\et}$, we have a functorial isomorphism of $B_{\cris}$-modules
\begin{equation}H^i(X_{\bk,\et}, \LL)\otimes_{\Zp} B_{\cris} \simto H^i(X_{\bk,\proet}, \LL\otimes \mathbb{B}_{\cris}).\end{equation}
compatible with $G_k$-action, filtration, and Frobenius.
\end{thm}
It seems to us that such a result alone may have interesting arithmetic applications, since it works for any lisse $\Zp$-sheaves, without the crystalline condition needed for comparison theorems.

 Following Faltings,  we say a lisse $\ZZ_p$-sheaf $\LL$  on  the pro-\'etale site $X_{\proet}$  is \emph{crystalline} if there exists a filtered $F$-isocrystal $\cE$ on $\cX_0/\cO_k$ together with an isomorphism of $\cO\BBcr$-modules
\begin{equation}\label{associated1}
\cE\otimes_{\cO_X^{\ur}} \cO\BBcr\simeq \LL\otimes_{\ZZ_p}\cO\BBcr,
\end{equation}
which is compatible with connection, filtration and Frobenius. Here, $\cO_X^{\ur}$ is the pullback to $X_{\proet}$  of $\cO_{\cX_{\et}}$ and $\cO\BBcr$ is the crystalline period sheaf of $\cO_X^{\ur}$-module with connection $\nabla$ such that $\cO\BBcr^{\nabla=0}=\BBcr$.  When this holds, we say the lisse $\ZZ_p$-sheaf $\LL$ and the filtered $F$-isocrystal $\cE$ are \emph{associated}.

We illustrate the  construction of  the crystalline comparison isomorphism briefly.
Firstly, we prove  a Poincar\'e lemma for the crystalline period sheaf $\mathbb{B}_{\cris}$ on  $X_{\proet}$. It follows from the Poincar\'e lemma (Proposition \ref{poincare}) that the natural morphism from $\BBcr$  to the de Rham complex $DR(\cO\BBcr)$ of $\cO\BBcr$ is a quasi-isomorphism, which is compatible with  filtration and Frobenius.
When $\LL$ and $\cE$ are associated, the natural morphism
\[
\LL\otimes_{\ZZ_p}DR(\cO\BBcr) \rightarrow DR(\cE)\otimes \cO\BBcr
\]
is an isomorphism compatible with Frobenius  and filtration. Therefore we find  a quasi-isomorphism
\begin{equation*}
\LL\otimes_{\ZZ_p}\BBcr\simeq   DR(\cE)\otimes\cO\BBcr.
\end{equation*}
From this we deduce
\begin{equation*}
R\Gamma(X_{\bk,\proet}, \LL\otimes_{\ZZ_p} \BBcr)\simto R\Gamma(X_{\bk,\proet}, DR(\cE)\otimes \cO\mathbb B_{\cris} ).
\end{equation*}
Via the natural  morphism of topoi $\overline{w}: X_{\bk, \proet}^{\sim}\ra \mathcal X_{\et}^{\sim}$, one has
\begin{equation*}
R\Gamma(X_{\bk, \proet},DR(\cE)\otimes\cO\BBcr)\simeq R\Gamma(\cX_{\et},DR(\cE)\ho_{\cO_k} B_{\cris}))
\end{equation*}
for which  we have used the fact that the natural morphism \[\cO_{\cX}\widehat{\otimes}_{\cO_k}B_{\cris}\to R\overline w_{\ast}\cO\BBcr\] is an isomorphism (compatible with extra structures), which is  a result of Andreatta-Brinon.

Combining the isomorphisms above, we obtain the desired crystalline comparison isomorphism.
\begin{thm}[see Theorem \ref{thm.comp}]\label{abs}  Let $\LL$ be a lisse $\ZZ_p$-sheaf on $X$ and $\mathcal E$ be a filtered $F$-isocrystal on $\cX_0/\cO_k$ which  are associated as in (\ref{associated1}). Then there is a natural isomorphism of $B_{\cris}$-modules
\[
H^i(X_{\bk,\et}, \LL)\otimes B_{\cris} \simto H_{\cris}^i(\cX_0/\cO_k,\mathcal E)\otimes_{k} B_{\cris}
\]which is compatible with $G_k$-action, filtration and Frobenius.
\end{thm}

After obtaining a refined version of the acyclicity of crystalline period sheaf $\cO\BBcr$ in \S\ \ref{acy}, we achieve the crystalline comparison in the relative setting, which reduces to Theorem \ref{abs} when $\cY=\Spf \cO_k$:

\begin{thm}[see Theorem \ref{thm.relativecomp}]
Let  $f\colon \cX\to \cY$ be a proper smooth morphism of smooth formal schemes over $\Spf \cO_k$, with $f_k\colon X\to Y$ the generic fiber  and  $f_{\cris}$  the morphism between the crystalline topoi. Let $\LL, \cE$ be as in Theorem \ref{abs}. Suppose that $R^if_{k*}\LL$ is a lisse  $\ZZ_p$-sheaf on $Y$. Then it is crystalline and is associated to the filtered $F$-isocrystal $R^if_{\cris \ast}\cE$.
 \end{thm}

\textbf{Acknowledgments.}
The authors are deeply indebted to Andreatta, Iovita and Scholze for the works \cite{AI} and \cite{Sch}. They wish to express their gratitude to  Kiran Kedlaya for his interest in this project, and to the referee for the detailed comments. During the preparation of this work, the first named author was supported by the Research Institute for Mathematical
Sciences, a Joint Usage/Research Center located in Kyoto University, and the second named author benefited from the Support Project of High-level Teachers in Beijing Municipal Universities in the Period of 13th Five-year Plan.

\section{Crystalline period sheaves}

Let $k$ be a discretely valued nonarchimedean extension of $\Qp$, with $\kappa$ its residue field.
Let $X$ be a locally noetherian adic space over $\mathrm{Spa}(k,\cO_k)$.  For the fundamentals on the pro-\'etale site $X_{\proet}$, we refer to \cite{Sch}.

The following terminology and notation will be used frequently throughout the paper. We shall fix once for all an algebraic closure $\bk$ of $k$, and consider $X_{\bk}:=X\times_{\Spa(k, \cO_k)}\Spa(\bk, \cO_{\bk})$ as an object of $X_{\proet}$ (see the paragraph after the proof of Proposition 3.13 in \cite{Sch}). As in  \cite[Definition 4.3]{Sch}, an object $U\in X_{\proet}$ lying above $X_{\bk}$ is called an \emph{affinoid perfectoid} (lying above $X_{\bk}$) if $U$ has a pro-\' etale presentation $U=\varprojlim U_i \to X$ by affinoids $U_i=\Spa(R_i,R_i^+)$ above $X_{\bk}$ such that, with $R^+$ the $p$-adic completion of $\varinjlim R_i^+$ and $R=R^+[1/p]$,  the pair $(R,R^+)$ is a perfectoid affinoid $(\widehat{\bk}, \cO_{\widehat{\bk}})$-algebra. Write $\widehat{U}=\Spa(R,R^+)$. By \cite[Proposition 4.8, Lemma 4.6]{Sch}, the set of affinoid perfectoids lying above $X_{\bk}$ of $X_{\proet}$ forms a basis for the topology.

\subsection{Period sheaves and their acyclicities} Following \cite{Sch}, let
\[
\nu\colon X_{\proet}^{\sim}\longrightarrow X_{\et}^{\sim}
\]
be the morphism of topoi, which, on the underlying sites, sends an \'etale morphism $U\to X$ to the pro-\'etale morphism from $U$ (viewed as a constant projective system) to $X$. Consider $\cO_X^+=\nu^{-1}\cO_{X_{\et}}^+$ and $\cO_X=\nu^{-1}\cO_{X_{\et}}$,  the (uncompleted) structural sheaves on $X_{\proet}$. More concretely, for $U=\varprojlim U_i$ a qcqs (quasi-compact and quasi-separated) object of $X_{\proet}$, one has  $\cO_X(U)=\varinjlim \cO_{X}(U_i)=\varinjlim \cO_{X_{\et}}(U_i)$ (\cite[Lemma 3.16]{Sch}). Set
\[
\widehat{\cO}_{X}^+:=\varprojlim_{n} \cO_X^+/p^n,\quad \widehat{\cO}_X:=\widehat{\cO}_X^+[1/p],\quad \textrm{and}\quad \cO_X^{\flat +}:=\varprojlim_{x\mapsto x^p}\cO_X^+/p.
\]
For $U\in X_{\proet}$ an affinoid perfectoid lying above $X_{\bk}$ with $\widehat{U}=\Spa(R,R^+)$, by \cite[Lemmas 4.10, 5.10]{Sch}, we have
\begin{equation*}
\widehat{\cO}_X^+(U) =R^+, \quad \widehat{\cO}_X(U)=R, \quad \textrm{and}\quad \cO_X^{\flat+}(U) = R^{\flat+}:=\varprojlim_{x\mapsto x^p}R^+/p.
\end{equation*}
Denote
\[
R^{\flat+}\ra R^+,\quad x=(x_0, x_1, \cdots)\mapsto x^{\sharp}:=\lim_{n\ra \infty}\widehat{x}_n^{p^n},
\]
for $\widehat{x}_{n}$ any lifting from $R^{+}/p$ to $R^+$. We have the multiplicative bijection induced by projection $R^+\ra R^+/p$:
\[
\varprojlim_{x\mapsto x^p}R^+\simto R^{\flat+},
\]
whose inverse sends $x\in R^{\flat+}$ to $(x^{\sharp},(x^{1/p})^{\sharp}, \cdots)$. Put $\mathbb A_{\inf}:=W(\cO_X^{\flat +})$ and $\BBinf=\AAinf[\frac{1}{p}]$. As $R^{\flat +}$ is a perfect ring, $\AAinf(U)=W(R^{\flat +})$ has no $p$-torsion. In particular, $\AAinf$ has no $p$-torsion and it is a subsheaf of $\BBinf$.

Following Fontaine, define as in \cite[Definition 6.1]{Sch} a natural  morphism
\begin{equation}\label{theta}
\theta\colon \mathbb A_{\inf}\to \widehat{\cO}_X^+
\end{equation} which, on an affinoid perfectoid $U$ with $\widehat{U}=\Spa(R,R^+)$, is given by
\begin{equation}\label{thetaU}
\theta(U)\colon \mathbb A_{\inf}(U)=W(R^{\flat +}) \longrightarrow \widehat{\cO}_X^+(U)=R^+, \quad  \sum_{n=0}^{\infty}p^n[x_n] \mapsto \sum_{n=0}^{\infty}p^nx_{n}^{\sharp}
\end{equation}
with $x_n\in R^{\flat +}$. As $(R,R^+)$ is a perfectoid affinoid algebra, $\theta(U)$ is known to be surjective (cf. \cite[5.1.2]{Bri}). Therefore, $\theta$ is also surjective.

\begin{defn}Let $X$ be a locally noetherian adic space over $\Spa(k, \cO_k)$ as above.
\begin{enumerate}
\item Define $\mathbb A_{\cris}$ to be the $p$-adic completion of the PD-envelope $\mathbb A_{\cris}^0$ of $\mathbb A_{\inf}$ with respect to the ideal sheaf $\ker(\theta)\subset \mathbb A_{\inf}$, and define $\mathbb B_{\cris}^+:=\mathbb A_{\cris}[1/p]$.

\item For $r\in \mathbb Z_{\geq 0}$, set $\Fil^r \AAcr^0:=\ker(\theta)^{[r]}\subset \AAcr^0$ to be the $r$-th divided power ideal, and $\Fil^{-r}\AAcr^0=\AAcr^0$. The family $\{\Fil^r\AAcr^0:r\in \mathbb Z\}$ gives a descending filtration of $\AAcr^0$.

\item For $r\in  \ZZ$, define $\Fil^{r}\AAcr\subset \AAcr$ to be the image of the following morphism of sheaves (we shall see below that this map is injective):
\begin{equation}\label{eq.mor}
\varprojlim_n  (\Fil^r\AAcr^0)\slash p^n\longrightarrow \varprojlim_n \AAcr^0/p^n=\AAcr,
\end{equation}
and define $\Fil^{r}\BBcrp=\Fil^r\AAcr[1/p]$.
\end{enumerate}
\end{defn}

Let $p^{\flat}=(p_i)_{i\geq 0}$ be a fixed family of elements of $\bk$ such that $p_0=p$ and that $p_{i+1}^p=p_i$ for any $i\geq 0$. Set $\xi:=[p^{\flat}]-p$, which can be seen as a section of the restriction $\mathbb A_{\inf}|_{X_{\bk}}$ of the pro-\'etale sheaf $\mathbb A_{\inf}$ to $X_{\proet}/X_{\bk}$.
\begin{prop} \label{nonzerodivisor}
We have $\ker(\theta)|_{X_{\bk}}=(\xi)\subset \mathbb A_{\inf}|_{X_{\bk}}$. Furthermore, $\xi\in \AAinf|_{X_{\bk}}$ is not a zero-divisor.
\end{prop}

\begin{proof} As the set of affinoid perfectoids $U$ lying above $X_{\bk}$ forms a basis for the topology of $X_{\proet}/X_{\bk}$, we only need to check that, for any such  $U$, $\xi\in \AAinf(U)$ is not a zero-divisor and that the kernel of $\theta(U)\colon \AAinf(U)\to \widehat{\cO}_X^+(U)$ is generated by $\xi$. Write $\widehat{U}=\Spa(R,R^+)$. Then $\AAinf(U)=W(R^{\flat +})$ and $\widehat{\cO}_X^{+}(U)=R^+$, hence we reduce our statement to (the proof of) \cite[Lemma 6.3]{Sch}.
\end{proof}

\begin{cor}\label{cor.DespOfAcris0} \begin{enumerate}
\item We have $
\mathbb A_{\cris}^0|_{X_{\bk}}=\mathbb A_{\inf}|_{X_{\bk}}\left[\xi^n/n!: n\in \mathbb N \right]\subset \mathbb{B}_{\inf}|_{X_{\bk}}$. In particular, $\AAcr^0$ and $\AAcr$ have no $p$-torsion. Moreover, for every $r\geq 0$, $\Fil^r\AAcr^0|_{X_{\bk}}=\mathbb{A}_{\inf}|_{X_{\bk}}[\xi^{n}/n! : n\geq r]$ and $\mathrm{gr}^r \AAcr^0|_{X_{\bk}}=\widehat{\cO}_X^+\cdot (\xi^r/r!)\stackrel{\sim}{\rightarrow}\widehat{\cO}_X^+|_{X_{\bk}}$.

\item The morphism \eqref{eq.mor} is injective, hence $\varprojlim_n\Fil^r\AAcr^0/p^n\stackrel{\sim}{\longrightarrow}\Fil^r\AAcr$. Moreover, for $r\geq 0$, $\mathrm{gr}^r \AAcr|_{X_{\bk}}\stackrel{\sim}{\longrightarrow}\widehat{\cO}_X^+|_{X_{\bk}}$.
\end{enumerate}
\end{cor}

\begin{proof} The first three statements in (1) are clear from Proposition \ref{nonzerodivisor}.
In particular, for $r\geq 0$ we have the following exact sequence
\begin{equation}\label{eq.sesFilAcris0}
\xymatrix{0\ar[r] &  \Fil^{r+1}\AAcr^0|_{X_{\bk}}\ar[r]& \Fil^r \AAcr^0|_{X_{\bk}}\ar[r] & \widehat{\cO}_X^+|_{X_{\bk}}\ar[r] & 0},
\end{equation}
where the second map sends $a\xi^{r}/r!$ to $\theta(a)$. This gives the last assertion   of (1).

As $\widehat{\cO}_X^+$ has no $p$-torsion, an induction on $r$ shows that the cokernel of the inclusion $\Fil^r\AAcr^0\subset \AAcr^0 $ has no $p$-torsion. As a result, the morphism \eqref{eq.mor} is injective and $\Fil^r\AAcr$ is the $p$-adic completion of $\Fil^r\AAcr^0$. Since $\widehat{\cO}_X^+$ is $p$-adically complete, we deduce from \eqref{eq.sesFilAcris0} also the following short exact sequence after passing to $p$-adic completions:
\begin{equation}\label{eq.sesFilAcris}
\xymatrix{0\ar[r] &  \Fil^{r+1}\AAcr |_{X_{\bk}}\ar[r]& \Fil^r \AAcr|_{X_{\bk}}\ar[r] & \widehat{\cO}_X^+|_{X_{\bk}}\ar[r] & 0}
\end{equation}
giving the last part of (2).
\end{proof}

Let $\epsilon=(\epsilon^{(i)})_{i\geq 0}$ be a sequence of elements of $\bk$ such that $\epsilon^{(0)}=1$, $\epsilon^{(1)}\neq 1$ and $(\epsilon^{(i+1)})^p=\epsilon^{(i)}$ for all $i\geq 0$. Then $1-[\epsilon]$ is a well-defined element of the restriction $\mathbb A_{\inf}|_{X_{\bk}}$ to $X_{\proet}/X_{\bk}$ of $\mathbb A_{\inf}$. Moreover $1-[\epsilon]\in \ker(\theta)|_{X_{\bk}}=\Fil^1\AAcr|_{X_{\bk}}$. Let
\begin{equation}\label{t}
t:=\log([\epsilon])=-\sum_{n=1}^{\infty}\frac{(1-[\epsilon])^n}{n},
\end{equation}
which is well-defined in $\AAcr|_{X_{\bk}}$ since $\Fil^1\AAcr$ is a PD-ideal.

\begin{defn} Let $X$ be a locally noetherian adic space over $\Spa(k,\cO_k)$. Define $\BBcr=\BBcrp[1/t]$. For $r\in \mathbb Z$, set $\Fil^r\BBcr=\sum_{s\in \ZZ}t^{-s}\Fil^{r+s}\BBcrp\subset \BBcr$.
\end{defn}

\begin{rk}\label{rk.element-t} We shall see in Corollary \ref{cor.no-t-torsion} that $t$ is not a zero-divisor in $\AAcr$ and in $\BBcrp$, so $\BBcrp\subset \BBcr$.
\end{rk}

Before investigating these period sheaves in details, we first study them  over a perfectoid affinoid $(\widehat{\bk}, \cO_{\widehat{\bk}})$-algebra $(R,R^+)$. Consider
\[
\mathbb A_{\inf}(R,R^{+}):=W(R^{\flat +}),\quad \mathbb B_{\inf}(R,R^+):=\mathbb A_{\inf}(R,R^+)[1/p],
\]
and define the morphism
\begin{equation*}
\theta_{(R,R^{+})}\colon \mathbb A_{\inf}(R,R^+)\longrightarrow R^+
\end{equation*}
in the same way as in (\ref{thetaU}). It is known to be surjective as $(R,R^+)$ is perfectoid. The element $\xi$ generates $\ker(\theta_{(R,R^+)})$ and is not a zero-divisor in $\mathbb A_{\inf}(R,R^+)$. Let $\mathbb{A}_{\cris}(R,R^{+})$ be the $p$-adic completion of the PD-envelope of $\mathbb A_{\inf}(R,R^+)$ with respect to $\ker(\theta_{(R,R^+)})$. So $\mathbb A_{\cris}(R,R^+)$ is the $p$-adic completion of
\[
\mathbb A_{\cris}^0(R,R^{+}):=\mathbb A_{\inf}(R,R^{+})\left[\frac{\xi^n}{n!} : n\in \mathbb N\right]\subset \mathbb B_{\inf}(R,R^{+}).
\]
For $r$ an integer, let $\Fil^r\AAcr^0(R,R^+)\subset \AAcr^0(R,R^+)$ be the $r$-th PD-ideal, \emph{i.e.,} the ideal generated by $\xi^n/n!$ for $n\geq \mathrm{max}\{r,0\}$.  Let $\Fil^r\AAcr(R,R^+)\subset \mathbb A_{\cris}(R,R^+)$ be the closure (for the $p$-adic topology) of $\Fil^r\AAcr^0(R,R^+)$ inside $\AAcr(R,R^+)$. Finally, put $\mathbb B_{\cris}^{+}(R,R^+):=\mathbb A_{\cris}(R,R^+)[1/p]$, $\mathbb B_{\cris}(R,R^{+}):=\mathbb B_{\cris}^{+}(R,R^+)[1/t]$, and for $r\in \mathbb Z$, set
\begin{eqnarray*}
\Fil^r\BBcrp(R,R^+):=\Fil^ r\AAcr(R,R^+)[1/p] \textrm{ and} \\ \Fil^r\BBcr(R,R^+):=\sum_{s\in \ZZ}t^{-s}\Fil^{r+s}\BBcrp(R,R^+).
\end{eqnarray*}

In particular, taking $R^+=\cO_{\Cp}$ with $\Cp$ the $p$-adic completion of the fixed algebraic closure $\bk$ of $k$, we get Fontaine's rings $A_{\cris}$, $B_{\cris}^+$, $B_{\cris}$ as in \cite{Fon94}. Write $\Cp^{\flat}$ the tilt of $\Cp$, which is an algebraically closed nonarchimedean field of characteristic $p$. The maximal ideal of its ring of integers $\cO_{\Cp}^{\flat}$ is generated by $[p^{\flat}]^{1/p^N}$ for all $N\in \mathbb N$.
Let $\mathcal I\subset A_{\cris}$ be the ideal generated by
\[
\{[\epsilon]^{1/p^N}-1, [p^{\flat}]^{1/p^N}:N\in \mathbb N\}\subset A_{\cris}.
\]
By \cite[Lemme 6.3.1]{Bri}, we have $\mathcal I\subset \mathcal I^2+ p^n\cdot A_{\cris}$ for any $n\in \mathbb N_{>0}$. In particular, $\mathcal I\cdot ( A_{\cris}/p^n)=(\mathcal I\cdot (A_{\cris}/p^n))^2$. In the following, when working with algebras (or modules) over $A_{\cris}/p^n$, we consider the almost-setting with respect to the ideal $\mathcal I\cdot (A_{\cris}/p^n)\subset A_{\cris}/p^n$. When $n=1$, as $\epsilon^{1/p^N}-1\in \cO_{\Cp}^{\flat}$ is contained in the maximal ideal, $\mathcal I\cdot (A_{\cris}/p)$ is the same as the ideal generated by $\{ [p^{\flat}]^{1/p^N}:N\in \mathbb N\}$. So the almost-setting adopted here for $A_{\cris}/p$-modules is the same as the one used by Scholze (see the paragraph before Theorem 6.5 of \cite{Sch} for his convention).

\begin{lemma}\label{lem.technical-I}Let $X$ be a locally noetherian adic space over $(k,\cO_k)$. Let $\mathcal{F}$ be a $p$-adically complete sheaf of $A_{\cris}$-modules on $X_{\proet}$, flat over $\Zp$. Set $\mathcal F_n=\mathcal F/p^n$, $n\in \mathbb Z_{\geq 1}$. Assume that, for any affinoid perfectoid $U$ above $X_{\bk}$,
\begin{enumerate}
\item[(a)] there exists a $p$-adically complete $A_{\cris}$-module $F(U)$, flat over $\mathbb Z_p$, equipped with a morphism of $A_{\cris}$-modules $\alpha_U:F(U) \ra \mathcal F(U)$ such that the composed morphism
\[
\alpha_{U,1}: F(U)/p\stackrel{\alpha_U \textrm{ mod }p}{\lra} \mathcal F(U)/p\lra \cF_1(U)
\]
is an almost isomorphism; and
\item[(b)] the $A_{\cris}/p$-module $H^i(U,\mathcal F_1)$ is almost zero for any $i>0$.
\end{enumerate}
Then, for an affinoid perfectoid $U$ as above, $n\geq 1$ and $i>0$,
\begin{enumerate}
\item[(1)] the composed morphism
\[
\alpha_{U,n}: F(U)/p^n \stackrel{\alpha_U \textrm{ mod }p^n}{\lra} \mathcal F(U)/p^n \lra \mathcal F_n(U)
\]
is an almost isomorphism, and $H^i(U,\mathcal F_n)$ is almost zero;
\item[(2)] we have $\mathcal I\cdot R^1\varprojlim \cF_n(U)=0$, and that $\ker(\alpha_U)$ and $\mathrm{coker}(\alpha_U)$ are killed by $\mathcal I^2$. Furthermore, $\mathcal I^2\cdot R^i\varprojlim \mathcal F_n=0$, and $\mathcal I^{2i+1}\cdot H^i(U,\mathcal F)=0$.
\end{enumerate}
\end{lemma}

\begin{proof} Let $U$ be a affinoid perfectoid lying above $X_{\bk}$. For $n\in \mathbb Z_{\geq 1}$, let $F(U)_n=F(U)/p^n$. Since $\mathcal F$ and $F(U)$ are flat over $\Zp$, we have exact sequences
\[
0\lra \mathcal F_1\stackrel{\cdot p^n}{\lra} \mathcal F_{n+1}\stackrel{\rm can}{\lra} \mathcal F_n\lra 0,
\]
and
\[
0\lra F(U)_{1}\stackrel{\cdot p^n}{\lra} F(U)_{n+1}\stackrel{\rm can}{\lra} F(U)_n\lra 0,
\]
from which we deduce exact sequences
\[
H^i(U,\mathcal F_1)\lra H^i(U,\mathcal F_{n+1})\lra H^i(U,\mathcal F_n), \quad i\geq 0,
\]
and a commutative diagram with exact rows
\[
\xymatrix{0\ar[r] & F(U)_1\ar[r]^{\cdot p^n}\ar[d]^{\alpha_{U,1}} & F(U)_{n+1}\ar[r] \ar[d]^{\alpha_{U,n+1}}& F(U)_n\ar[r]\ar[d]^{\alpha_{U,n}} & 0 \\ 0\ar[r] & \mathcal F_1(U)\ar[r]^{\cdot p^n} & \mathcal F_{n+1}(U)\ar[r] & \mathcal F_n(U).}
\]
So, by induction on $n$, (1) follows from conditions (a) and (b) above and the fact that $\mathcal I \cdot (A_{\cris}/p^n)=\mathcal I^2\cdot (A_{\cris}/p^n)$ for any $n\in \mathbb N$ (\cite[Lemme 6.3.1]{Bri}). In particular, the collection $(\alpha_{U,n})_{n\in \mathbb N}$ gives a morphism of projective systems
\[
(F(U)_n)_{n\in \mathbb N}\lra (\mathcal F_n(U))_{n\in \mathbb N}
\]
whose kernel and cokernel are killed by $\mathcal I$. Passing to limits relative to $n$, we find $\mathcal I\cdot R^1\varprojlim \mathcal F_n(U)=0$, and that $\ker(\alpha_U)$ and $\mathrm{coker}(\alpha_U)$ are killed by $\mathcal I^2$, giving also the first part of (2).

To go further, let $\mathrm{Sh}$ (resp. $\mathrm{PreSh}$) denote the category of sheaves (resp. of presheaves) on $X_{\proet}$, and let $\mathrm{Sh}^{\mathbb N}$ (resp. $\mathrm{PreSh}^{\mathbb N}$) denote the category of projective systems of sheaves (resp. projective systems of presheaves) indexed by $\mathbb N$ on $X_{\proet}$. The projective limit functor $\varprojlim:  \mathrm{Sh}^{\mathbb N}\ra \mathrm{Sh}$ factors as
\[
\mathrm{Sh}^{\mathbb N} \stackrel{\sigma}{\lra} \mathrm{PreSh}^{\mathbb N}\stackrel{\varprojlim'}{\lra} \mathrm{PreSh} \stackrel{a}{\lra} \mathrm{Sh},
\]
where the first functor $\sigma$ is induced from the natural inclusion $\mathrm{Sh}\subset \mathrm{PreSh}$, the second is the projective limit functor of presheaves, and the third takes a presheaf to its associated sheaf. Let $\tau:=\varprojlim' \circ \sigma$. Since the functor $a$ is exact, $R\varprojlim=a\circ R\tau $. In particular, for each $i$, $R^i\varprojlim \mathcal F_n$ is the associated sheaf of the presheaf $R^i\tau (\mathcal F_{\bullet})$, where we
denote by $\mathcal F_{\bullet}$ the projective system $(\mathcal F_n)_{n\in \mathbb N}$. Let
\[
0\lra I(0)_{\bullet}\lra I(1)_{\bullet}\lra \cdots
\]
be an injective resolution of $\mathcal F_{\bullet}$ in $\mathrm{Sh}^{\mathbb N}$, with $I(j)_{\bullet}=(I(j)_n)_{n\in \mathbb N}$. For each $i$, $R^i\sigma (\mathcal F_{\bullet})$ is the $i$-th cohomology of this complex in $\mathrm{PreSh}^{\mathbb N}$. On the other hand, for each $n$, this resolution gives an injective resolution of $\mathcal F_n$ in the category of sheaves on $X_{\proet}$ (\cite[(1.1) Proposition]{Jan})
\[
0\lra I(0)_{n}\lra I(1)_{n}\lra \cdots.
\]
So, for $U$ an affinoid perfectoid lying above $X_{\bk}$, $H^i(U,\mathcal F_n)$ is the $i$-th cohomology group of the induced complex
\[
0\lra I(0)_{n}(U)\lra I(1)_{n}(U)\lra \cdots,
\]
which is annihilated by $\mathcal I$ when $i>0$ by (1). Varying $n$, we find
\begin{equation}\label{eq.IRsigma=0}
\mathcal I\cdot R^i\sigma(\mathcal F_{\bullet})(U)=\mathcal I\cdot H^i(\cU, \cF_{\bullet})=0, \quad \textrm{for } i>0.
\end{equation}
On the other hand, as infinite products exist and are exact functors in $\mathrm{PreSh}$, by \cite[(1.6) Proposition]{Jan}, we have an exact sequence of presheaves for each $i\in \mathbb Z$:
\[
0\lra R^1{\varprojlim}' R^{i-1}\sigma(\mathcal F_{\bullet}) \lra R^i\tau (\mathcal F_{\bullet})\lra {\varprojlim}' R^{i}\sigma(\mathcal F_{\bullet})\lra 0.
\]
The latter gives an exact sequence of abelian groups:
\[
0\lra \left(R^1{\varprojlim}' R^{i-1}\sigma(\mathcal F_{\bullet})\right)(U) \lra R^i\tau (\mathcal F_{\bullet})(U)\lra \left({\varprojlim}' R^{i}\sigma(\mathcal F_{\bullet})\right)(U)\lra 0.
\]
We claim that $\mathcal I^2 \cdot R^i\tau (\mathcal F_{\bullet})(U)=0$ for $i\geq 1$. Indeed,
when $i\geq 2$, our claim follows from \eqref{eq.IRsigma=0}. When $i=1$, by what we have shown in the first paragraph, $\mathcal I\cdot (R^1\varprojlim' \sigma(\mathcal F_{\bullet}))(U)=\mathcal I\cdot R^1\varprojlim (\mathcal F_n(U))=0$.
Combining \eqref{eq.IRsigma=0}, we get $\mathcal I^2\cdot R^1\tau(\mathcal F_{\bullet})(U)=0$, as claimed. Since $R^i\varprojlim\mathcal F_n$ is the associated sheaf of $R^i\tau(\mathcal F_{\bullet})$, we deduce $\mathcal I^2\cdot  R^i\varprojlim \mathcal F_n=0$ when $i>0$. This proves the second part of (2).

Now, because $\mathcal I^2\cdot R^i\varprojlim\mathcal F_n=0$ for $i>0$, for the spectral sequence below
\[
E_2^{i,j}=H^i(U, R^j\varprojlim \mathcal F_n)\Longrightarrow H^{i+j}(U,R\varprojlim \mathcal F_n),
\]
one checks that $\mathcal I^2\cdot E_{\infty}^{i,j}=0$ for $j>0$, $E_{\infty}^{i,0}=E_{i+1}^{i,0}$, and the surjection $E_2^{i,0}\ra E_{\infty}^{i,0}$ has kernel killed by $\mathcal I^{2i-2}$. It follows that the canonical map
\[
H^i(U,\mathcal F)=H^i(U,\varprojlim\mathcal F_n)\lra H^i(U,R\varprojlim \mathcal F_n)
\]
has kernel annihilated by $\mathcal I^{2i-2}$ and cokernel annihilated by $\mathcal I^{2i}$. Using the short exact sequence (see Lemma \ref{jannsen})
\[
0\lra R^1\varprojlim H^{i-1}(U,\mathcal F_n) \lra H^i(U, R\varprojlim\mathcal F_n)\lra \varprojlim H^i(U,\mathcal F_n)\lra 0,
\]
and that $R^1\varprojlim H^{i-1}(U,\mathcal F_n)$ is annihilated by $\mathcal I$,
one deduces that the morphism
\[
H^i(U,\mathcal F)=H^i(U,\varprojlim\mathcal F_n)\lra \varprojlim H^i(U,\mathcal F_n), \quad i\geq 0
\]
is an isomorphism up to $\mathcal I^{2i}$-torsion, i.e., its kernel and cokernel are killed by $\mathcal I^{2i}$. In particular, $H^i(U,\mathcal F)$ is killed by $\mathcal I^{2i+1}$ when $i>0$, as wanted.
\end{proof}

\begin{lemma}\label{vanish}
Let $X$ be a locally noetherian adic space over $(k,\cO_k)$. Let $U\in X_{\proet}$ be an affinoid perfectoid above $X_{\bk}$ with $\widehat{U}=\mathrm{Spa}(R,R^{+})$. Then there is a natural filtered morphism $\mathbb{A}_{\cris}(R, R^+) \ra \mathbb A_{\cris}(U)$ of $A_{\cris}$-algebras, inducing an almost isomorphism $\Fil^r \mathbb{A}_{\cris}(R, R^+)/p^n \ra (\Fil^r\mathbb A_{\cris}/p^n)(U)$ for any $r\geq 0$ and $n\geq 1$. Moreover, $H^i(U,\Fil^r\mathbb A_{\cris}/p^n)^a=0$ for any $i>0$. \end{lemma}

\begin{proof}  As $U$ is affinoid perfectoid, $\widehat{\cO}_X^+(U)=R^+$, $\cO_{X}^{\flat +}(U)=R^{\flat +}$ and $\theta(U)=\theta_{(R,R^+)}$. In particular, $\mathbb A_{\inf}(U)=\mathbb A_{\inf}(R,R^+)$, and the natural morphism
\[
\mathbb A_{\inf}(R,R^+)=\mathbb A_{\inf}(U) \longrightarrow \mathbb A_{\cris}(U)
\]
sends $\ker(\theta_{(R,R^+)})$ into $\Fil^1\AAcr(U)$. As $\Fil^1\AAcr(U)\subset \AAcr(U)$ has a PD-structure, the morphism above induces a map $\AAcr^0(R,R^+)\ra \AAcr(U)$, respecting the filtrations on both sides. Passing to $p$-adic completions, we obtain the required filtered morphism $\AAcr(R,R^+)\ra \AAcr(U)$ of $A_{\cris}$-algebras.

In particular, for each $r\geq 0$, we have a natural morphism
\[
\Fil^r\AAcr(R,R^+)\lra \Fil^r\AAcr(U).
\]
Composing its reduction modulo $p^n$ with $\Fil^r\AAcr(U)/p^n\ra (\Fil^r\AAcr/p^n)(U)$, we get a morphism
\begin{equation}\label{eq.composedmorphism}
\Fil^r\AAcr(R,R^+)/p^n\lra (\Fil^r\AAcr/p^n)(U)
\end{equation}
for all $n\geq 1$. We need to show that this is an almost isomorphism of $A_{\cris}/p^n$-modules, and that $H^i(U,\Fil^r\AAcr/p^n)^a=0$ for $i>0$. Using Lemma \ref{lem.technical-I} (1), one reduces to the case where $n=1$. Then, we claim that it suffices to prove this when $r=0$. Indeed, from the exact sequence \eqref{eq.sesFilAcris0} and the fact that $\widehat{\cO}_X^+$ is $p$-torsion free, we deduce a short exact sequence for each $r\geq 0$:
\begin{equation}\label{eq.sesModp}
0\lra \left(\Fil^{r+1}\AAcr|_{X_{\bk}}\right)/p \lra \left(\Fil^r\AAcr|_{X_{\bk}}\right)/p \lra \left(\cO_{X}^+|_{X_{\bk}}\right)/p\lra 0.
\end{equation}
We have an short exact sequence for $\Fil^r\AAcr(R,R^+)/p$ obtained in a similar way:
\[
0\lra \Fil^{r+1}\AAcr(R,R^+)/p \lra \Fil^r(R,R^+)/p \lra R^+/p\lra 0.
\]
As $U$ is affinoid perfectoid, by \cite[Lemma 4.10]{Sch}, the natural morphism
\[
R^+/p=\widehat{\cO}_X^+(U)/p\lra (\cO_X^+/p)(U)
\]
is an almost isomorphism: recall that the almost-setting adopted here for $A_{\cris}/p$-modules is the same as the one used by Scholze in \cite{Sch}. So we have a commutative diagram with exact rows, such that the right vertical map is an almost isomorphism:
\[
\xymatrix{0\ar[r] & (\Fil^{r+1}\AAcr/p)(U)^a \ar[r] &  (\Fil^r\AAcr/p)(U)^a \ar[r]  & (\cO_{X}^+/p)(U)^{a} &   \\ 0\ar[r] & \Fil^{r+1}\AAcr(R,R^+)^a/p \ar[r] \ar[u]&  \Fil^r(R,R^+)^a/p \ar[r] \ar[u]& (R^+/p)^a\ar[r]\ar[u]^{\approx} & 0.}
\]
In particular, the upper row of the diagram above is right exact. On the other hand, combined with \cite[Lemma 4.10]{Sch}, the long exact sequence associated with \eqref{eq.sesModp} gives an isomorphism
\[
H^i(U,\Fil^{r+1}\AAcr/p)^a\stackrel{\sim}{\lra} H^{i}(U,\Fil^r\AAcr/p)^a, \quad \forall ~ i\geq 1.
\]
Therefore, our claim follows by induction on $r\geq 0$.

So, it remains to prove the second part of our lemma when $r=0$ and $n=1$. Denote by $\alpha_1$ the map \eqref{eq.composedmorphism} in this case. Recall the following identification of $\AAcr(R,R^{+})/p$ (see \cite[Proposition 6.1.2]{Bri})
\[
\AAcr(R,R^+)/p\stackrel{\sim}{\lra} (R^{\flat+}/(p^{\flat})^p)[\delta_i:i \in \mathbb N]/(\delta_i^p:i\in \mathbb N),
\]
with $\delta_i$ being the image of $\xi^{[p^{i+1}]}$. Similarly, restricting to $X_{\proet}/X_{\bk}$, we have
\[
\AAcr/p\stackrel{\sim}{\lra} (\cO_{X}^{\flat+}/(p^{\flat})^p)[\delta_i:i\in \mathbb N]/(\delta_i^p:i\in \mathbb N).
\]
In particular, $\AAcr/p$ is a direct sum of copies of $\cO_{X}^{\flat+}/(p^{\flat})^p$ on $X_{\proet}/X_{\bk}$. Under these identifications, the morphism $\alpha_1$ is induced by
\[
R^{\flat+}/(p^{\flat})^p = \cO_X^{\flat+}(U)/(p^{\flat})^p\lra (\cO_{X}^{\flat+}/(p^{\flat})^p)(U).
\]
Since $U$ is qcqs, to conclude the proof, it suffices to show $H^i(U,\cO_{X}^{\flat+}/(p^{\flat})^p)^a=0$ for $i>0$, and that the morphism above is an almost isomorphism. Both of these two assertions follow from \cite[Lemma 4.10]{Sch}.
\end{proof}

\begin{cor}\label{cor.acyclicityBcris}Keep the notation of Lemma \ref{vanish}. In particular, $U$ is an affinoid perfectoid of $X_{\proet}$ lying above $X_{\bk}$, with $\widehat{U}=\Spa(R,R^+)$.
\begin{enumerate}
\item[(1)] For any $r\in \mathbb N$, there is a natural morphism $\Fil^r\AAcr(R,R^+)\ra \Fil^r\AAcr(U)$ of $A_{\cris}$-modules whose kernel and cokernel are killed by $\mathcal I^2$. Moreover, $\mathcal I^2\cdot R^i\varprojlim_n \Fil^r\AAcr/p^n=0$ and $\mathcal I^{2i+1}\cdot H^i(U,\Fil^r\AAcr)=0$ for $i>0$.
\item[(2)] The natural morphisms in \emph{(1)} induce isomorphisms
\[
\BBcr(R,R^+)\stackrel{\sim}{\lra}\BBcr(U),\quad \textrm{and}\quad \Fil^r \BBcr(R,R^+)\stackrel{\sim}{\lra} \mathcal \Fil^r\BBcr(U)
\]
for all $r\in \mathbb Z$. Moreover, $H^i(U,\BBcr)=H^i(U,\Fil^r \BBcr)=0$ for $i\geq 1$. 
\end{enumerate}
\end{cor}

\begin{proof} (1) This follows directly from Lemma \ref{lem.technical-I} and Lemma \ref{vanish}.

(2) As $U$ is qcqs, inverting $t$, we deduce from (1) a morphism of $B_{\cris}$-modules $\BBcr(R,R^+)\ra \BBcr(U)$, with kernel and cokernel killed by $\mathcal I^2$. Moreover, the $B_{\cris}$-module $H^i(U,\BBcr)$ is annihilated by $\mathcal I^{2i+1}$ for $i>0$. Note that $t$ is divisible by $[\epsilon]-1$ in $A_{\cris}$ (see for example the proof of Theorem \ref{withoutfil}), so the assertions for $\BBcr$ follow as $\mathcal I\cdot B_{\cris}=B_{\cris}$.

To prove our assertions for $\Fil^r\BBcr$, observe first that the following two properties hold. For $s\in  \mathbb N$,  (a) the canonical map $\mathrm{gr}^s\BBcr^+(R,R^+)\ra \mathrm{gr}^s\BBcr^+(U)$ is an isomorphism; and (b) $H^i(U,\mathrm{gr}^s\BBcr^+)=0$. Indeed, over $X_{\proet}/X_{\bk}$, we have $\mathrm{gr}^s\BBcr^+=\widehat{\cO}_X\cdot \xi^{[s]}$ by \eqref{eq.sesFilAcris}. Similarly, $\mathrm{gr}^s\BBcr^+(R,R^+)=R\cdot \xi^{[s]}$. Therefore the two properties above follow from \cite[Lemma 4.10]{Sch}.

Now, let us begin the proof for $\Fil^r\BBcr$. Twisting by $t^{-r}$ if necessary, we shall assume $r=0$. Inverting $p$, we get from (1) a morphism of $B_{\cris}^+$-modules:
\[
\alpha_s: \Fil^s\BBcrp(R,R^+)\ra \Fil^s\BBcrp(U)
\]
whose kernel and cokernel are killed by $\mathcal I^2$. Passing to direct limits (with respect to multiplication-by-$t$), we deduce a natural map of $B_{\cris}^+$-modules, denoted by $\beta$:
\[
\Fil^0\BBcr(R,R^+)=\varinjlim_{s\geq 0} \Fil^{s}\BBcrp(R,R^+)\lra \Fil^0\BBcr(U)=\varinjlim_{s\geq 0} \Fil^{s}\BBcrp(U),
\]
whose kernel and cokernel are killed by $\mathcal I^2$, hence by $t^2$. One needs to show that this map is an isomorphism. The injectivity of $\beta$ is clear as $\ker(\beta)\subset \BBcr(R,R^+)$ is $t$-torsion free.  So it is enough to check its surjectivity. Note that we have the following commutative diagram with exact rows
\[
\xymatrix{0\ar[r] & \Fil^{s+1}\BBcrp(U) \ar[r] &  \Fil^{s}\BBcrp(U) \ar[r]  &  \mathrm{gr}^s\BBcrp(U)&   \\ 0\ar[r] & \Fil^{s+1}\BBcrp(R,R^+) \ar[r] \ar[u]^{\alpha_{s+1}}&  \Fil^{s}\BBcrp(R,R^+) \ar[r] \ar[u]^{\alpha_{s}} & \mathrm{gr}^s\BBcrp(R,R^+) \ar[r]\ar[u]^{\simeq} & 0.}
\]
Here the right vertical map is an isomorphism because of the property (a) above. Then, by the Snake Lemma, the inclusion $\Fil^{s+1}\BBcrp(U)\subset \Fil^s\BBcrp(U)$ induces an isomorphism $\mathrm{coker}(\alpha_{s+1})\stackrel{\sim}{\ra}\mathrm{coker}(\alpha_{s})$. So we get identification $\mathrm{coker}(\alpha_{s})\stackrel{\sim}{\ra}\mathrm{coker}(\alpha_0)=:C$ induced by the inclusion $\Fil^s\BBcrp(U)\subset \Fil^0\BBcrp(U)= \BBcrp(U)$ for all $s\geq 0$. With these identifications, we have
\[
\mathrm{coker}(\beta)=\varinjlim_{s\geq 0} \mathrm{coker}(\alpha_{s})=\varinjlim_{s\geq 0} C
\]
where, in the last direct limit, the transition maps are multiplication-by-$t$. Since $C=\mathrm{coker}(\alpha_0)$ is killed by $t^2$, necessarily $\mathrm{coker}(\beta)=0$. In other words, $\beta$ is surjective, thus is an isomorphism.

Finally, it remains to show $H^i(U,\Fil^0\BBcr)=0$ when $i>0$. For $s\in \mathbb N$, from the commutative diagram
\[
\xymatrix{\Fil^0\BBcrp\ar[r]^{\cdot t^s} & \Fil^{s}\BBcrp \\ \Fil^{s}\BBcrp\ar@{^(->}[u]^{\rm canonical}\ar[ru]_{\cdot t^s} & },
\]
we get a commutative diagram of cohomology groups
\begin{equation}\label{eq.vanish-Fil-Bcris}
\xymatrix{H^i(U,\Fil^0\BBcrp\ar[r]^{\cdot t^s}) & H^i(U,\Fil^{s}\BBcrp) \\H^i(U, \Fil^{s}\BBcrp)\ar[u]\ar[ru]_{\cdot t^s} & }.
\end{equation}
We claim that, for $i>0$, the vertical map above is surjective. To see this, it suffices to check the surjectivity of the map
\[
H^i(U,\Fil^{s+1}\BBcrp)\lra H^i(U,\Fil^s\BBcrp)
\]
induced by the inclusion $\Fil^{s+1}\BBcrp\subset \Fil^s\BBcrp$ for any $s\geq 0$. So, one only needs to show $H^i(U,\mathrm{gr}^s\BBcrp)=0$ for $i>0$, as claimed by the property (b) above.
Thus the vertical map in \eqref{eq.vanish-Fil-Bcris} is surjective. On the other hand, the $B_{\cris}^+$-module $H^i(U,\Fil^{s}\BBcrp)$ is killed by $\mathcal I^{2i+1}$ and $t$ is a multiple of $[\epsilon]-1\in \mathcal I$, so the map
\[
H^{i}(U,\Fil^{s}\BBcrp)\lra H^i(U,\Fil^{s}\BBcrp), \quad x\mapsto t^{s}x
\]
is zero whenever $s\geq 2i+1$. Thus, the horizontal map in \eqref{eq.vanish-Fil-Bcris} is trivial when $s\geq 2i+1$. We conclude $H^i(U,\Fil^0\BBcr)=\varinjlim_s H^i(U,\Fil^{s}\BBcrp)=0$ for $i>0$.
\end{proof}

\subsection{Period sheaves with connections}  In this section, assume that the $p$-adic field $k$ is \emph{absolutely unramified}.
Let $\cX$ be a smooth formal scheme over $\cO_k$. Set $X:=\cX_k$ the generic fiber of $\cX$, viewed as an adic space over $\Spa(k,\cO_k)$. For any \'etale morphism $\mathcal Y\to \cX$, by taking the generic fiber, we obtain an \'etale morphism $\mathcal Y_k\to X$ of adic spaces, hence an object of the pro-\'etale site $X_{\proet}$. In this way, we get a morphism of sites $\cX_{\et}\to X_{\proet}$, with the induced morphism of topoi
\[
w\colon X_{\proet}^{\sim}\longrightarrow \cX_{\et}^{\sim}.
\]
Let $\cO_{\cX_{\et}}$ denote the structural sheaf of the \'etale site $\cX_{\et}$: for any \'etale morphism $\mathcal Y\to \cX$ of formal schemes over $\cO_k$, $\cO_{\cX_{\et}}(\mathcal Y)=\Gamma(\mathcal Y,\cO_{\mathcal Y})$. Define $\cO_{X}^{\ur+}:=w^{-1}\cO_{\cX_{\et}}$  and $\cO_{X}^{\ur}:=w^{-1 }\cO_{\cX_{\et}}[1/p]$.  Thus $\cO_{X}^{\ur+}$ is the associated sheaf of the presheaf $\widetilde{\cO_{X}^{\ur+}}$:
\[
X_{\proet} \ni U \mapsto \varinjlim_{(\mathcal Y,a)}\cO_{\cX_{\et}}(\mathcal Y)=:\widetilde{\cO_{X}^{\ur+}}(U),
\]
where the limit runs through all pairs $(\mathcal Y,a)$ with $\mathcal Y\in \cX_{\et}$ and $a\colon U\to \mathcal Y_k$ a morphism making the following diagram commutative
\begin{equation}\label{eq.factorization}
\xymatrix{U\ar[r]\ar[rd]_a & X=\cX_k \\ & \mathcal Y_k\ar[u]}.
\end{equation}
The morphism $a\colon U\to \mathcal Y_k$ induces a map $\Gamma(\mathcal Y,\cO_{\mathcal Y})\to \cO_{X}(U)$.
There is then a morphism of presheaves $\widetilde{\cO_{X}^{\ur+}}\to \cO_X^+$, whence a morphism of sheaves
\begin{equation}\label{eq.OurAlg}
\cO_{X}^{\ur+}\to \cO_{X}^{+}.
\end{equation}
Recall $\AAinf:=W(\cO_{X}^{\flat+})$. Set $\cO\AAinf:= \cO_{X}^{\rm ur +}\otimes_{\cO_k}\mathbb A_{\inf}$ and
\begin{equation}\label{eq.thetaX}
\theta_{X}\colon \cO\AAinf \longrightarrow  \widehat{\cO}_X^+
\end{equation}
to be the map induced from $\theta\colon \mathbb A_{\inf}\to \widehat{\cO}_X^+$ of \eqref{theta} by extension of scalars.

\begin{defn} Consider the following sheaves on $X_{\proet}$.
\begin{enumerate}
\item Let $\cO\mathbb A_{\cris}$ be the $p$-adic completion of the PD-envelope $\cO\mathbb A_{\cris}^0$ of $\cO\mathbb A_{\inf}$ with respect to the ideal sheaf $\ker(\theta_X)\subset \cO\mathbb{A}_{\inf}$, $\cO\mathbb B_{\cris}^+:=\cO\mathbb A_{\cris}[1/p]$, and $\cO\mathbb B_{\cris}:=\cO\mathbb B_{\cris}^+[1/t]$ with $t=\log([\epsilon])$ defined in \eqref{t}.
\item For $r\geq 0$ an integer, define $\Fil^{r}\cO\AAcr^0\subset \cO\AAcr^ 0$ to be the $r$-th PD-ideal $\ker(\theta_X)^{[r]}$, and $\Fil^r\cO\AAcr$  the image of the canonical map
\[
\varprojlim \Fil^r\cO\AAcr^0/p^n \longrightarrow \varprojlim \cO\AAcr^0/p^n=\cO\AAcr.
\]
Also set $\Fil^{-r}\cO\AAcr=\cO\AAcr$ for $r>0$.
\item For any integer $r$, set $
\Fil^r\cO\BBcrp:=\Fil^r\cO\AAcr[1/p]$ and $\Fil^r\cO\BBcr:=\sum_{s\in \ZZ}t^{-s}\Fil^{r+s}\cO\BBcrp$.
\end{enumerate}
\end{defn}

\begin{rk} As $t^{p}=p!\cdot t^{[p]}$ in $A_{\cris}=\AAcr(\widehat{\bk},\cO_{\widehat{\bk}})$, one can also define $\Fil^r\cO\BBcr$ as $\sum_{s\in \mathbb N}t^{-s}\Fil^{r+s}\cO\AAcr$. A similar observation holds for $\Fil^r\BBcr$.
\end{rk}

\begin{rk}\label{rk.Fil1OAcris} \begin{enumerate}
\item We shall see later that $\cO\AAcr$ has neither $p$-torsion (Corollary \ref{cor.no-p-torsion-for-OAcris}) nor $t$-torsion (Corollary \ref{cor.no-t-torsion}). So $\cO\AAcr\subset \cO\BBcrp\subset \cO\BBcr$.

\item The morphism $\theta_X$ of \eqref{eq.thetaX} extends to a surjective morphism $\cO\AAcr^0\ra \widehat{\cO}_X^+$ with kernel $\Fil^1\cO\AAcr^0$, hence a morphism $\cO\AAcr\ra \widehat{\cO}_X^+$. Let us denote them again by $\theta_X$. As $\widehat{\cO}_X^+$ is $p$-adically complete and has no $p$-torsion, using the snake lemma and passing to limits one can deduce the following short exact sequence
\[
0\longrightarrow \varprojlim_n (\Fil^1\cO\AAcr^0/p^n) \longrightarrow \cO\AAcr \stackrel{\theta_X}{\longrightarrow} \widehat{\cO}_X^+ \longrightarrow 0.
\]
In particular, $\Fil^1\cO\AAcr=\ker(\theta_X)$.
\end{enumerate}
\end{rk}

\begin{defn} Consider the following sheaves on $X_{\proet}$.
\begin{enumerate}
\item Let $\AAcr\{\langle u_1,\ldots, u_d\rangle\}$ be the $p$-adic completion of the sheaf of PD polynomial rings $\AAcr^0\langle u_1,\ldots, u_d\rangle \subset \BBinf[u_1,\ldots, u_d]$. Set $\BBcrp\{\langle u_1,\ldots, u_d\rangle \}:=\AAcr\{\langle u_1,\ldots, u_d\rangle \}[1/p]$ and $\BBcr\{\langle u_1,\ldots, u_d\rangle \}:=\AAcr\{\langle u_1,\ldots, u_d\rangle\}[1/t]$.

\item For $r$ an integer, let $\Fil^{r}\AAcr^0\langle u_1,\ldots, u_d\rangle\subset \AAcr^0\langle u_1,\ldots, u_d\rangle$ be the ideal sheaf $\sum_{i_1,\ldots,i_d\geq 0}\Fil^{r-(i_1+\ldots+i_d)}\mathbb A_{\cris}^0\cdot u_1^{[i_1]}\cdots u_{d}^{[i_d]}\subset \AAcr^0\langle u_1,\ldots, u_d\rangle$,
and $\Fil^r(\mathbb A_{\cris}\{\langle u_1,\ldots, u_d\rangle\})\subset \AAcr\{\langle u_1,\ldots, u_d\rangle \}$ the image of the morphism
\[
\varprojlim_n \left(\Fil^r\AAcr^0\{\langle u_1,\ldots, u_d\rangle \}/p^n\right) \longrightarrow \AAcr\{\langle u_1,\ldots, u_d\rangle\}.
\]
The family $\{\Fil^r(\AAcr\{\langle u_1,\ldots, u_d\rangle \}): r\in \mathbb Z\}$ gives a descending filtration of $\AAcr\{\langle u_1,\ldots, u_d\rangle\}$. Inverting $p$, we obtain $\Fil^r(\BBcrp\{\langle u_1,\ldots, u_d\rangle\})$. Set finally $
\Fil^r(\BBcr\{\langle u_1,\ldots, u_d\rangle \}):=\sum_{s\in \mathbb Z}t^{-s} \Fil^{r+s}(\BBcrp\{\langle u_1,\ldots, u_d\rangle \})$.
\end{enumerate}
\end{defn}

To describe $\cO\AAcr$ more explicitly, assume that $\cX$ is \emph{small}, i.e., there is an \'etale morphism $\cX\rightarrow \Spf(\cO_k\{T_1^{\pm 1},\ldots, T_{d}^{\pm 1}\})=:\mathcal T^d$ of formal schemes over $\cO_k$, where we have used $\{-,\ldots,-\}$ to denote convergent power series. Let $\TT^d$ denote the generic fiber of $\mathcal T^d$ and $\widetilde{\TT^d}$ be obtained from $\TT^d$ by adding a compatible system of $p^{n}$-th roots of $T_i$ for $1\leq i\leq d$ and $n\geq 1$:
\[
\widetilde{\TT^d}:=\Spa(k\{T_{1}^{\pm1/p^{\infty}},\ldots, T_d^{\pm 1/p^{\infty}}\},\cO_k\{T_{1}^{\pm1/p^{\infty}},\ldots, T_d^{\pm 1/p^{\infty}}\}).
\]
Set $\widetilde X:=X\times_{\TT^d}\widetilde{\TT^d}$. Let $T_i^{\flat}\in \cO_{X}^{\flat+}|_{\widetilde X}$ be the element $(T_i, T_i^{1/p},\ldots,T_i^{1/p^n},\ldots)$. Then $\theta_{X}(T_i\otimes 1-1\otimes [T_{i}^{\flat}])=0$, giving an $\mathbb A_{\cris}$-linear morphism
\begin{equation}\label{alpha}
\alpha\colon \mathbb A_{\cris}\{\langle u_1,\ldots, u_d\rangle\}|_{\widetilde{X}}\longrightarrow \cO\mathbb A_{\cris}|_{\widetilde X},\quad u_i\mapsto T_i\otimes 1-1\otimes [T_i^{\flat}].
\end{equation}
Clearly, $\alpha$ respects the filtrations on both sides.
\begin{prop} \label{iso}The morphism $\alpha$ of \eqref{alpha} is an isomorphism. Moreover, $\alpha$ is strictly compatible with the filtrations on both sides, i.e., the inverse of the isomorphism $\alpha$ respects also the filtrations of both sides. 
\end{prop}

\begin{lemma}\label{algebra} Let $\bk$ be an algebraic closure of $k$. Then $\mathbb A_{\cris} \{\langle u_1,\ldots, u_d\rangle \}|_{\widetilde X_{\bk}}$ has an $\cO_X^{\ur+}|_{\widetilde X_{\bk}}$-algebra structure, sending $T_i$ to $u_i+[T_i^{\flat}]$, such that the composition
\[
\cO_X^{\ur+}|_{\widetilde X_{\bk}}\longrightarrow \mathbb A_{\cris}\{\langle u_1,\ldots, u_d\rangle\}|_{\widetilde X_{\bk}}\stackrel{\theta'|_{\widetilde X_{\bk}}}{\longrightarrow} \widehat{\cO}_X^+|_{\widetilde X_{\bk}}
\]
is the map \eqref{eq.OurAlg} composed with $\cO_{X}^{+}\to \widehat{\cO}_X^+$.  Here $\theta'\colon \AAcr\{\langle u_1,\ldots, u_d\rangle\}\to \widehat{\cO}_X^+$ is induced from the map $\mathbb A_{\cris}\stackrel{\theta}{\ra}  \widehat{\cO}_X^+$ by sending $u_i$'s to $0$.
\end{lemma}

\begin{proof}
Let $U$ be an affinoid perfectoid lying above $\widetilde{X}_{\bk}$, and $\mathcal Y\in \cX_{\et}$, equipped with a map $a\colon U\to \cY_k$ as in \eqref{eq.factorization}. We shall first construct a morphism of $\cO_k$-algebras
\begin{equation}\label{eq.cY}
\cO_{\cY}(\cY)\longrightarrow \left( \AAcr\{\langle u_1,\ldots, u_d\rangle \}\right) (U)
\end{equation}
sending $T_i$ to $u_i+[T_i^{\flat}]$. As our construction is functorial and as $\AAcr\{\langle u_1,\ldots, u_d\rangle\}$ is a sheaf, shrinking $U$ and $\cY$ if necessary, we may and we do assume $\cY=\Spf(A)$ affine. Then, the map $a:U\ra \cY_k$ gives us a morphism of $\cO_k$-algebras $a^{\#}: A\to R^+$. Moreover, $U$ being qcqs, $
\AAcr\{\langle u_1,\ldots, u_d\rangle\}(U)=\varprojlim_n\left(\left(\AAcr/p^n\right)(U)\langle u_1,\ldots, u_d\rangle\right)$. Consequently, the morphisms $\AAcr(R,R^+)/p^n\ra (\AAcr/p^n)(U)$ for all $n\geq 1$ in Lemma \ref{vanish} induce a natural filtered map
\begin{equation}\label{eq.MapBetweenPDPolynomials}
\AAcr(R,R^+)\{\langle u_1,\ldots, u_d\rangle\}\lra \AAcr\{\langle u_1,\ldots, u_d\rangle \}(U).
\end{equation}
Therefore, to obtain \eqref{eq.cY}, it suffices to construct a natural map
\begin{equation}\label{eq.cY1}
A=\cO_{\cY}(\cY)\longrightarrow \AAcr(R,R^+)\{\langle u_1,\ldots, u_d\rangle \}.
\end{equation}
of $\cO_k$-algebras mapping $T_i$ to $u_i+[T_i^{\flat}]$. To do so, composing the map $\cY\to \cX$ with $\cX\to \mathcal T^d$, we obtain an \'etale morphism $b:\mathcal Y\to \mathcal T^d$ of $p$-adic formal schemes, whence an \'etale morphism $b^{\#}:\cO_k\{T_1^{\pm 1},\ldots, T_d^{\pm 1}\}\ra A$ of $\cO_k$-algebras. On the other hand, $u_i$ has divided powers in $\AAcr(R,R^+)\{\langle u_1,\ldots, u_d\rangle\}$, so $[T_{i}^{\flat}]+u_i\in \AAcr(R,R^+)\{\langle u_1,\ldots,u_d\rangle\}$ is invertible. This allows to define a map $f:\cO_k\{T_1^{\pm 1},\ldots, T_d^{\pm 1}\} \ra \AAcr(R,R^+)\{\langle u_1,\ldots,u_d\rangle \}$ of $\cO_k$-algebras, sending each $T_i$ to $u_i+[T_i^{\flat}]$. Let $f_n$ be its reduction modulo $p^n$ for $n\geq 1$. Then, we have the following diagram, which is commutative without the dotted map:
\begin{equation}\label{eq.diagramforgn}
\xymatrix{R^{+}/p^n & A/p^n\ar[l]_{a^{\#} \textrm{ mod }p^n} \ar@{.>}[ld]_{\exists !\  g_n}  \\  \AAcr(R,R^+)\{\langle u_1,\ldots, u_d\rangle \}/p^n\ar[u]^{\theta_{(R,R^+)}' \textrm{ mod }p^n } & \cO_k\{T_1^{\pm 1},\ldots,T_d^{\pm 1}\}/p^n.\ar[l]_<<<<<{f_n}\ar[u]_{b^{\#} \textrm{ mod } p^n}}
\end{equation}
Here, the left vertical map is the reduction modulo $p^n$ of the ring homomorphism
\begin{equation*}
\theta_{(R,R^+)}': \AAcr(R,R^+)\{\langle u_1,\ldots, u_d\rangle\}\lra R^+
\end{equation*}
which sends each $u_i$ to $0$ and extends the usual map $\theta_{(R,R^+)}:\AAcr(R,R^+)\ra R^+$. Since $\ker(\theta_{(R,R^+)}')$ has PD-structure, the left vertical map of \eqref{eq.diagramforgn} has a nilpotent kernel. Then, by the \'etaleness of $b^{\#}$, we get a unique dotted map $g_n$, making the whole diagram \eqref{eq.diagramforgn} commutative. These $g_n$'s are compatible with each other, and the limit $\varprojlim_n g_n$ gives the morphism \eqref{eq.cY1} and thus \eqref{eq.cY}. Since $\widetilde{\cO_{X}^{\ur+}}(U) = \varinjlim \cO_{\mathcal Y}(\cY)$, where the direct limit runs through the diagrams \eqref{eq.factorization}, we get  from \eqref{eq.cY} a morphism of $\cO_k$-algebras, sending $T_i$ to $u_i+[T_i^{\flat}]$:
\[
\widetilde{\cO_{X}^{\ur+}}(U)\longrightarrow \mathbb A_{\cris}\{\langle u_1,\ldots, u_d\rangle\}(U),
\]
whose construction is functorial with respect to affinoid perfectoid $U\in X_{\proet}$ lying above $\widetilde{X}_{\bk}$. As such affinoids perfectoids form a basis of the topology on $X_{\proet}/\widetilde{X}_{\bk}$, by passing to the associated sheaf, we obtain the required morphism of sheaves of $\cO_{k}$-algebras $\cO_{X}^{\ur}|_{\widetilde{X}_{\bk}}\to \mathbb A_{\cris}|_{\widetilde{X}_{\bk}}\{\langle u_1,\ldots, u_d\rangle\}$ sending $T_i$ to $u_i+[T_i^{\flat}]$. The last statement follows from the assignment $\theta'(U_i)=0$ and the fact that $\theta( [T_i^{\flat}])=T_i$.
\end{proof}

\begin{proof}[Proof of Proposition~\ref{iso}] As $\widetilde X_{\bk}\to \widetilde X$ is a covering in $X_{\proet}$, it is enough to show that $\alpha|_{\widetilde{X}_{\bk}}$ is an isomorphism.
By Lemma \ref{algebra}, there exists a morphism of sheaves of $\cO_k$-algebras $
\cO_X^{\ur+}|_{\widetilde{X}_{\bk}}\rightarrow\mathbb A_{\cris}\{\langle u_1,\ldots, u_d\rangle\} |_{\widetilde{X}_{\bk}}
$
sending $T_i$ to $u_i+[T_{i}^{\flat}]$. By extension of scalars, we find the morphism
\[
\beta: \cO\AAinf|_{\widetilde{X}_{\bk}}=\left(\cO_{X}^{\ur+}\otimes_{\cO_k} \AAinf \right)|_{\widetilde X_{\bk}}\longrightarrow \mathbb A_{\cris}\{\langle u_1,\ldots, u_d\rangle\}|_{\widetilde X_{\bk}}
\]
which maps $T_i\otimes 1$ to $u_i+[T_{i}^{\flat}]$. Consider the composite (with $\theta'$ as in Lemma \ref{algebra})
\[
\theta'|_{\widetilde X_{\bk}}\circ \beta \colon \cO\AAinf|_{\widetilde X_{\bk}} \longrightarrow \mathbb A_{\cris}\{\langle u_1,\ldots, u_d\rangle\}|_{\widetilde X_{\bk}}{\longrightarrow} \widehat{\cO}_X^+|_{\widetilde{X}_{\bk}},
\]
which is $\theta_X|_{\widetilde X_{\bk}}$ by Lemma \ref{algebra}. Therefore, $\beta \left(\ker(\theta_X|_{\widetilde X_{\bk}})\right)\subset  \ker( \theta'|_{\widetilde X_{\bk}})$. Since $\ker(\theta'|_{\widetilde{X}_{\bk}})$ has a PD-structure, $\beta$ extends to the PD-envelope $\cO\AAcr^0|_{\widetilde X_{\bk}}$ of the source, and thus to $\cO\AAcr|_{\widetilde X_{\bk}}$ as $\mathbb A_{\cris}\{\langle u_1,\ldots, u_d\rangle\}$ is $p$-adically complete.  Thus we obtain the morphism below, still denoted by $\beta$, sending $T_i\otimes 1$ to $u_i+[T_i^{\flat}]$:
\[
\beta \colon \cO\mathbb{A}_{\cris}|_{\widetilde X_{\bk}}\longrightarrow \mathbb A_{\cris}\{\langle u_1,\ldots, u_d\rangle\}|_{\widetilde X_{\bk}}.
\]
The morphism $\beta$ above preserves $\Fil^1$, hence all the $\Fil^r$'s. One shows that $\beta$ and $\alpha $ are inverse to each other, giving our proposition.
\end{proof}

\begin{cor}\label{BcrisIso} Keep the notations above. There is a natural $\BBcrp$-linear \emph(resp. $\BBcr$-linear\emph) isomorphism sending $u_i$ to $T_i\otimes 1-1\otimes [T_i^{\flat}]$,
\[
\mathbb B_{\cris}^{+}\{ \langle u_1,\ldots, u_d \rangle\}|_{\widetilde{X}}\stackrel{\sim}{\longrightarrow} \cO\mathbb B_{\cris}^{+}|_{\widetilde X},\  \emph(\textrm{resp. }\BBcr\{\langle u_1,\ldots, u_d \rangle \}|_{\widetilde X}\stackrel{\sim}{\longrightarrow} \cO\BBcr|_{\widetilde X}\emph)
\]
which is strictly compatible with filtrations on both sides.
\end{cor}

\begin{cor}\label{cor.no-p-torsion-for-OAcris} Let $\cX$ be a smooth formal scheme over $\cO_k$. Then $\cO\AAcr$ has no $p$-torsion. In particular, $\cO\AAcr\subset \cO\BBcrp$.
\end{cor}

\begin{proof} This is a local question on $X$. Hence we may and do assume there is an \'etale morphism $X\to \Spf(\cO_k\{T_1^{\pm 1},\ldots,T_d^{\pm 1}\})$. So we reduces ourselves to the corresponding statement for $\AAcr\{\langle u_1,\ldots, u_d\rangle\}$. As the latter is the $p$-adic completion of $\AAcr^0\langle u_1,\ldots, u_d\rangle$, one reduces further to the fact that $\AAcr^0$ has no $p$-torsion, as shown in Corollary \ref{cor.DespOfAcris0} (1).
\end{proof}

An important feature of $\cO\mathbb A_{\cris}$ is that it has an $\AAcr$-linear connection on it. To see this, set $\Omega_{X/k}^{1,\ur+}:=w^{-1}\Omega^1_{\mathcal{X}_{\et}/\cO_k}$, which is  locally free of finite rank over $\cO_X^{\ur+ }$.  Let
\[
\Omega^{i,\ur+}_{X/k}:=\wedge_{\cO_{X}^{\ur +}}^i \Omega_{X/k}^{1,\ur+}, \quad \textrm{and}\quad \Omega_{X/k}^{i,\ur}:=\Omega_{X/k}^{1,\ur +}[1/p] \quad \forall i\geq 0.
\]
Then $\cO\AAinf$ admits a unique $\AAinf$-linear connection $
\nabla\colon \cO\AAinf\rightarrow \cO\AAinf\otimes_{\cO_X^{\ur+}}\Omega_{X/k}^{1,\ur+}$
induced from the usual one on $\cO_{\cX_{\et}}$. This connection extends uniquely to $\cO\AAcr$
\[
\nabla\colon \cO\AAcr\longrightarrow \cO\AAcr\otimes_{\cO_X^{\ur+}}\Omega_{X/k}^{1,\ur+},
\]
which is $\AAcr$-linear. Inverting $p$ (resp. $t$), we get  also a $\BBcrp$-linear (resp. $\BBcr$-linear) connection on $\cO\BBcrp$ (resp. on $\cO\BBcr$):
\[
\nabla\colon \cO\BBcrp\longrightarrow \cO\BBcrp\otimes_{\cO_X^{\ur}}\Omega_{X/k}^{1,\ur},\quad  \textrm{and}\quad \nabla\colon \cO\BBcr\longrightarrow \cO\BBcr\otimes_{\cO_X^{\ur}}\Omega_{X/k}^{1,\ur}.
\]
From Proposition \ref{iso}, we obtain

\begin{cor}[Crystalline Poincar\'e lemma]\label{poincare}  Let $\cX$ be a smooth formal scheme of dimension $d$ over $\cO_k$. Then there is  an exact sequence of pro-\'etale sheaves:
\[
0\to \mathbb A_{\cris}\to \cO\mathbb A_{\cris}\stackrel{\nabla}{\to }\cO\mathbb A_{\cris}\otimes_{\cO_{X}^{\ur+}}\Omega^{1,\ur+}_{X/k}\stackrel{\nabla}{\to}\ldots \stackrel{\nabla}{\to} \cO\mathbb A_{\cris}\otimes_{\cO_{X}^{\ur+}}\Omega^{d,\ur+}_{X/k} \to 0,
\]
which is strictly exact with respect to the filtration giving $\Omega^{i,\ur+}_{X/k}$ degree $i$. In particular, the connection $\nabla$ is integrable and satisfies Griffiths transversality with respect to the filtration on $\cO\mathbb A_{\cris}$, i.e. $\nabla (\Fil^i\cO\mathbb A_{\cris})\subset \Fil^{i-1}\cO\mathbb A_{\cris}\otimes_{\cO_{X}^{\ur+}}\Omega^{1,\ur+}_{X/k}$. Moreover, the similar results for $\cO\BBcrp$ and $\cO\BBcr$ hold.
\end{cor}

\begin{proof} It suffices to prove the assertion for $\cO\AAcr$. The question is local on $\cX$, so we assume there is an \'etale morphism $\cX \ra \Spf(\cO_k\{T_1^{\pm1},\ldots, T_d^{\pm 1}\})$. Under the isomorphism \eqref{alpha} of Proposition \ref{iso}, $\Fil^i\cO\mathbb{A}_{\cris}|_{\widetilde X}$ is the $p$-adic completion of
 \[
 \sum_{i_1,\ldots i_{d}\geq 0}\Fil^{i-(i_0+\ldots +i_{d })}\mathbb A_{\cris}|_{\widetilde X} u_1^{[i_1]}\cdots u_{d}^{[i_{d}]}
 \]
 with $T_i\otimes 1-1\otimes [T_i^{\flat}]$ sent to $u_i$. Moreover $\nabla(u_i^{[n]})=u_i^{[n-1]}\otimes d T_i$ for any $i, n\geq 1$, since the connection $\nabla$ on $\cO\mathbb A_{\cris}$ is $\mathbb A_{\cris}$-linear. The strict exactness and Griffiths transversality then follow.
\end{proof}

Using Proposition~\ref{iso}, we can establish an acyclicity result for $\cO\AAcr$ as in Lemma \ref{vanish}. Let $\cU=\Spf(A)$ be an affine open subset of $\cX$, admitting an \'etale morphism to $\mathcal T^d=\Spf(\cO_k\{T_1^{\pm 1},\ldots, T_d^{\pm 1}\})$. Let $U$ be the generic fiber, and set $
\widetilde{U}:=U\times_{\mathbb T^d}\widetilde{\mathbb T}^d$. Let $V$ be an affinoid perfectoid of $X_{\proet}$ lying above $\widetilde{U}_{\bk}$. Write $\widehat{V}=\Spa(R,R^+)$. Let $\cO\AAcr(R,R^+)$ be the $p$-adic completion of the PD-envelope $\cO\AAcr^0(R,R^+)$ of $A\otimes_{\cO_k}W(R^{\flat+})$ with respect to the kernel of the following morphism  induced from $\theta_{(R,R^+)}$ by extending scalars to $A$:
\[
\theta_{A}\colon A\otimes_{\cO_k}W(R^{\flat+})\longrightarrow R^+.
\]
Set $\cO\BBcrp(R,R^+):=\cO\AAcr(R,R^+)[1/p]$, $\cO\BBcr(R,R^+):=\cO\BBcrp(R,R^+)[1/t]$. For $r\in \mathbb Z$, define $\Fil^r\cO\AAcr(R,R^+)$ to be the closure inside $\cO\AAcr(R,R^+)$ for the $p$-adic topology of the $r$-th PD-ideal of $\cO\AAcr^0(R,R^+)$. Finally, set  $\Fil^r\cO\BBcrp(R,R^+):=\Fil^r\cO\AAcr(R,R^+)[1/p]$ and  $
\Fil^r\cO\BBcr(R,R^+):=\sum_{s\in \ZZ}t^{-s}\Fil^{r+s}\cO\BBcrp(R,R^+)$.

\begin{lemma} \label{2obcris} There is a  natural filtered morphism $\cO\AAcr(R,R^+)\ra \cO\AAcr(V)$ of $R^+\otimes_{\cO_k}A_{\cris}$-algebras, inducing an almost isomorphism $\Fil^r\cO\AAcr(R,R^+)/p^n\stackrel{\approx}{\to}(\Fil^r\cO\AAcr/p^n)(V)$ for each $r,n\geq 0$. Moreover, $H^i(V,\Fil^r\cO\AAcr/p^n)$  is almost zero whenever $i>0$.
\end{lemma}

\begin{proof} Let $\iota$ be the composition of the two natural ring homomorphisms below
\[
A\otimes_{\cO_k}W(R^{\flat +})\lra A\otimes_{\cO_k}\AAcr(V)\lra \cO\AAcr(V).
\]
Then, $\theta_{X}(V)\circ\iota=\theta_{A}$, with $\theta_X(V)$ the map obtained by taking sections at $V$ of $\theta_X:\cO\AAcr\ra \widehat{\cO}_X^+$. So $\iota(\ker(\theta_{A}))\subset \ker(\theta_X(V))$. As $\ker(\theta_X(V))$ has PD-structure, $\iota$ extends to $\cO\AAcr^0(R,R^+)$ and the resulting morphism $\cO\AAcr^0(R,R^+)\ra \cO\AAcr(V)$ respects the filtrations. Passing to $p$-adic completions, we obtain a filtered morphism $\cO\AAcr(R,R^+)\ra \cO\AAcr(V)$ of $A\otimes_{\cO_k}A_{\cris}$-algebras, still noted by $\iota$ in the following.

Let $\alpha_{(R,R^+)}:\AAcr(R,R^+)\{\langle u_1,\ldots, u_d\rangle\}\longrightarrow\cO\AAcr(R,R^+)$ be the $\AAcr(R,R^+)$-linear map, mapping $u_i$ to $T_i\otimes 1-1\otimes [T_i^{\flat}]$.
As in Proposition~\ref{iso}, it is an isomorphism, strictly compatible with the filtrations. Here $\Fil^r(\AAcr(R,R^+)\{\langle u_1,\ldots, u_d\rangle\})$ is defined to be the $p$-adic completion of
\[
\sum_{i_1,\ldots, i_d\geq 0} \Fil^{r-(i_1+\ldots +i_d)}\AAcr^0(R,R^+)u_{1}^{[i_1]}\cdots u_d^{[i_d]}\subset \AAcr(R,R^+)\{\langle u_1,\ldots, u_d\rangle \}.
\]
Therefore, we have the following commutative diagram
\[
\xymatrix{\AAcr(R,R^+)\{\langle u_1,\ldots, u_d\rangle\}\ar[r]^<<<<<{\simeq}_<<<<<{\alpha_{(R,R^+)}}\ar[d]_{\eqref{eq.MapBetweenPDPolynomials}} & \cO\AAcr(R,R^+) \ar[d]^{\iota}\\ \AAcr\{\langle u_1,\ldots, u_d\rangle\}(V)\ar[r]^<<<<<<<{\simeq}_<<<<<<<{\alpha(V)} & \cO\AAcr(V)},
\]
whose horizontal arrows are filtered isomorphisms. To prove the first part of our lemma, it suffices to show that \eqref{eq.MapBetweenPDPolynomials} induces an almost isomorphism
\begin{equation}\label{eq.MapBetweenFilPDPoly}
\Fil^r\AAcr(R,R^+)\{\langle u_1,\ldots, u_d\rangle\}/p^n\lra (\Fil^r\AAcr\{\langle u_1\ldots, u_d\rangle\}/p^n)(V)
\end{equation}
for any $r,n\geq 0$. By definition, the left-hand side of the morphism above is
\[
\bigoplus_{i_1,\ldots, i_d\geq 0} (\Fil^{r-i_1-\ldots-i_d}\AAcr(R,R^+)/p^n)\cdot u_1^{[i_1]}\cdots u_d^{[i_d]},
\]
while the right-hand side is given by (recall $V$ is qcqs)
\[
\bigoplus_{i_1,\ldots, i_d\geq 0}(\Fil^{r-i_1-\ldots-i_d}\AAcr/p^n)(V)\cdot u_1^{[i_1]}\cdots u_d^{[i_d]}.
\]
Under these descriptions, the map \eqref{eq.MapBetweenFilPDPoly} is induced from the natural maps
\[
\Fil^{r-i_1-\ldots -i_d}\AAcr(R,R^+)/p^n\lra (\Fil^{r-i_1-\ldots -i_d}\AAcr/p^n)(V), \quad i_1,\ldots, i_d\geq 0.
\]
So, that \eqref{eq.MapBetweenFilPDPoly} is an almost isomorphism follows from Lemma \ref{vanish}.

It remains to prove $H^i(V,\Fil^r\cO\AAcr/p^n)^a=0$ for $i>0$ and $n,r\geq 0$. Using the isomorphism $\alpha$ of Proposition \ref{iso}, we are reduced to the similar statement for $\Fil^r \AAcr\{\langle u_1,\ldots, u_d\rangle\}/p^n$. As $V$ is qcqs, we have
\[
H^i(V,\Fil^r\AAcr\{\langle u_1,\ldots, u_d\rangle \}/p^n)\simeq
\bigoplus_{i_1,\ldots, i_d\geq 0}H^i(V, \Fil^{r-i_1-\ldots-i_d}\AAcr/p^r),
\]
which vanishes by Lemma \ref{vanish}.
\end{proof}

\begin{cor}\label{cor.acyclicityOBcris} Keep the notation of Lemma \ref{2obcris}. Then the filtered morphism $\cO\AAcr(R,R^+)\ra \cO\AAcr(V)$ of Lemma \ref{2obcris} induces an isomorphism of $R^+\otimes_{\cO_k}B_{\cris}$-modules $\cO\BBcr(R,R^+)\stackrel{\sim}{\ra}\cO\BBcr(V)$, strictly compatible with filtrations. Moreover, $H^i(V,\Fil^r\cO\BBcr)=H^i(V,\cO\BBcr)=0$ for any $i>0$ and $r\in \mathbb Z$. 
\end{cor}

\begin{proof} The proof is the same as that of Corollary \ref{cor.acyclicityBcris} (2). Indeed, as $\cO\AAcr(R,R^+)$ and $\cO\AAcr$ are $p$-adically complete and flat over $\Zp$, by Lemma \ref{lem.technical-I} and Lemma \ref{2obcris}, the filtered morphism $\cO\AAcr(R,R^+)\ra \cO\AAcr(V)$ induces a filtered morphism $\cO\BBcr(R,R^+)\ra \cO\BBcr(V)$ of $A\otimes_{\cO_k}B_{\cris}$-modules, with kernel and cokernel killed by $\mathcal I^2$. Therefore, the latter is an isomorphism since $\mathcal I \cdot  (A\otimes_{\cO_k}B_{\cris})=A\otimes_{\cO_k}B_{\cris}$. To prove the statements for $\Fil^r\cO\BBcr$, as in the proof of Corollary \ref{cor.acyclicityBcris}, it suffices to establish the similar properties (a) and (b) hold for $\cO\BBcrp$. Let $s\geq 0$ be an integer. By Corollary \ref{BcrisIso}, one checks that, over $X_{\proet}/\widetilde{U}_{\bk}$, $\mathrm{gr}^s\cO\BBcr^+$ is a free module over $\widehat{\cO}_X$, with a basis given by the images of the elements
\[
\xi^{i_1}u_1^{i_1}\cdots u_d^{i_d}, \quad \textrm{ where } i_0,\ldots,i_d\in \mathbb N \textrm{  and } i_0+\ldots +i_d=s.
\]
Similar observation holds for $\mathrm{gr}^s\cO\BBcr^+(R,R^+)$. Consequently, by \cite[Lemma 4.10]{Sch}, the canonical map $\mathrm{gr}^s\cO\BBcr^+(R,R^+)\ra \mathrm{gr}^s\cO\BBcr^+(V)$ is an isomorphism and $H^i(V,\mathrm{gr}^r\cO\BBcr^+)=0$ for $i>0$, as wanted.
\end{proof}

\subsection{Frobenius on crystalline period sheaves}\label{localfrob}
We keep the notations in the previous subsection. So $k$ is \emph{absolutely unramified} and $\cX$ is a smooth formal scheme of dimension $d$ over $\cO_k$. We want to endow Frobenius endomorphisms on the crystalline period sheaves.

On $\AAinf=W(\cO_X^{\flat+})$, we have the Frobenius map
\[
\varphi\colon \AAinf\longrightarrow \AAinf, \quad (a_0, a_1,\ldots, a_n,\ldots)\mapsto (a_0^p,a_1^p,\ldots, a_n^p,\ldots).
\]
Then for any $a\in \AAinf$, we have $\varphi(a)\equiv a^p \textrm{ mod }p$. Thus, $\varphi(\xi)=\xi^p+p \cdot b$ with $b\in \AAinf|_{X_{\bk}}$. In particular $\varphi(\xi)\in \AAcr^0|_{X_{\bk}}$ has all divided powers. As a consequence we obtain a Frobenius $\varphi$ on $\AAcr^0$ extending that on $\AAinf$. By continuity, $\varphi$ extends to $\AAcr$ and thus to $\BBcrp$. Note that $\varphi(t)=\log([\epsilon^p])=pt$. Consequently $\varphi$ is extended to $\BBcr$ by setting $\varphi(\frac{1}{t})=\frac{1}{pt}$.

To endow a Frobenius on $\cO\AAcr$, we first assume that the Frobenius of $\cX_0=\cX\otimes_{\cO_k}\kappa$ lifts to a morphism $\sigma$ on $\cX$, which is compatible with the Frobenius on $\cO_k$. Then for $\mathcal Y\in \cX_{\et}$, consider the following diagram:
\[
\xymatrix{\mathcal Y_{\kappa}\ar@{^(->}[rrr] & & &  \mathcal Y\ar[d]^{\textrm{\'etale}} \\ \mathcal Y_{\kappa}\ar@{^(->}[r]\ar[u]^{\textrm{absolute Frobenius}} & \mathcal Y \ar[r]\ar@{.>}[urr]^{\exists \sigma_{\cY}} & \cX\ar[r]^{\sigma} & \cX.}
\]
As the right vertical map is \'etale, there is a unique dotted morphism above making the diagram commute. When $\cY$ varies in $\cX_{\et}$, the $\sigma_{\cY}$'s give rise to a $\sigma$-semilinear endomorphism on $\cO_{\cX_{\et}}$  whence a $\sigma$-semilinear endomorphism $\varphi$ on $\cO_X^{\ur+}$.

\begin{rk} In general $\cX$ does not admit a lifting of Frobenius. But as $\cX$ is smooth over $\cO_k$, for each open subset $\cU\subset \cX$ admitting an \'etale morphism $\cU\to \Spf(\cO_k\{T_1^{\pm 1},\ldots, T_d^{\pm 1}\})$, a similar argument as above shows that there exists a unique lifting of Frobenius on $\cU$ mapping $T_i$ to $T_i^p$.
\end{rk}

We deduce from above a Frobenius on $\cO\mathbb A_{\inf}=\cO_{X}^{\ur+}\otimes_{\cO_k}\AAinf$ given by $\varphi\otimes \varphi$. Abusing notation, we will denote it again  by $\varphi$. A similar argument as in the previous paragraphs shows that $\varphi$ extends to $\cO\AAcr^0$, hence to $\cO\AAcr$ by continuity, and finally to $\cO\mathbb B_{\cris}^{+}$ and $\cO\BBcr$. Moreover, under the isomorphism \eqref{alpha}, the Frobenius on $\AAcr\{\langle u_1,\ldots, u_d\rangle \}\stackrel{\sim}{\to}\cO\AAcr$ sends $u_i$ to $\varphi(u_i)=\sigma(T_i)-[T_i^{\flat}]^p$.

\begin{lemma}\label{FrobHorizontal} Assume as above that the Frobenius of $\cX_0=\cX\otimes_{\cO_k}\kappa$ lifts to a morphism $\sigma$ on $\cX$ compatible with the Frobenius on $\cO_k$. The Frobenius $\varphi$ on $\cO\AAcr$ is horizontal with respect to the connection $\nabla \colon \cO\AAcr\to \cO\AAcr\otimes \Omega^{1,\ur+}_{X/k}$. Similar assertions hold for $\cO\BBcrp$ and for $\cO\BBcr$.
\end{lemma}

\begin{proof}We need to check $\nabla\circ\varphi=(\varphi\otimes d\sigma)\circ\nabla$ on $\cO\mathbb{A}_{\cris}$. It is enough to do this locally. Thus we assume there exists an \'etale morphism $\cX\to \Spf(\cO_k\{T_1^{\pm 1},\ldots, T_d^{\pm 1}\})$. Recall the isomorphism \eqref{alpha}. By $\mathbb{A}_{\cris}$-linearity, it suffices to check the equality on the $u_i^{[n]}$'s.  We have
\begin{eqnarray*}
(\nabla\circ\varphi)(u_i^{[n]})&=&\nabla(\varphi(u_i)^{[n]})= \varphi(u_i)^{[n-1]}\nabla(\varphi(u_i))
\end{eqnarray*}
Meanwhile, $\varphi(u_i)-\sigma(T_i)=-[T_i^{\flat}]^p\in \AAcr$, hence $\nabla(\varphi(u_i))=d\sigma(T_i)$. Thus
\begin{eqnarray*}
\left((\varphi\otimes d\sigma)\circ\nabla\right)(u_i^{[n]})
&=& (\varphi\otimes d\sigma)(u_i^{[n-1]}\otimes dT_i)\\ & =&\varphi(u_i)^{[n-1]}\otimes d \sigma(T_i) \\
&=&(  \nabla\circ\varphi )(u_i^{[n]}),
\end{eqnarray*}
as desired.
\end{proof}

The Frobenius on $\cO\AAcr$ above depends on the initial lifting of Frobenius on $\cX$. For different choices of liftings of Frobenius on $\cX$, it is possible to compare explicitly the resulting Frobenius endomorphisms on $\cO\mathbb A_{\cris}$ with the help of the connection on it, at least when the formal scheme $\cX$ admits an \'etale morphism to $\Spf(\cO_k\{T_1^{\pm 1}, \ldots, T_d^{\pm 1}\})$.

\begin{lemma} \label{sigma12}Assume there is an \'etale morphism $\cX\to \Spf(\cO_k\{T_1^{\pm 1}, \ldots, T_d^{\pm 1}\})$. Let $\sigma_1,\sigma_2$ be two Frobenius liftings on $\cX$, and let $\varphi_1$ and $\varphi_2$ be the induced Frobenius maps on $\cO\mathbb A_{\cris}$, respectively. Then %for any quasi-compact $U\in X_{\proet}$,
we have the following relation on $\cO\mathbb A_{\cris}$:
\begin{equation}\label{phi2}
\varphi_2=\sum_{(n_1,\ldots, n_d)\in \NN^d}(\prod_{i=1}^d(\sigma_2(T_i)-\sigma_1(T_i))^{[n_i]})(\varphi_1\circ(\prod_{i=1}^dN_i^{n_i})) \end{equation}
where the $N_i$'s are the endomorphisms of $\cO \mathbb{A}_{\cris}$ such that $\nabla=\sum_{i=1}^d N_i\otimes d T_i$.
\end{lemma}

\begin{proof}
To simplify the notations, we will use the multi-index: for $\underline{m}=(m_1,\ldots, m_d)\in \mathbb N^d$, set $\underline N^{\underline m}:=\prod_{i=1}^d N_i^{m_i}$ and $|\underline m|:=\sum_im_i $.  Let us first observe that the series on the right-hand side of \eqref{phi2} gives a well-defined map on $\cO\AAcr$. As $\cO\AAcr$ is $p$-adically complete, it suffices to show that this is the case for $\cO\AAcr/p^n$ for any $n\geq 1$. Identify $\cO\AAcr/p^n$ with $(\AAcr/p^n)\langle u_1,\ldots, u_d\rangle$ using Proposition \ref{iso}. Thus, a local section $a$ of $\cO\AAcr/p^n$ can locally be written as a \emph{finite} sum
\[
a=\sum_{\underline m\in \mathbb N^d} b_{\underline m} \cdot \underline{u}^{[\underline m]}, \quad b_{\underline m}\in \AAcr/p^n. %\textrm{ and }\lim_{|m|\to \infty} b_{\underline m}=0.
\]
A calculation shows
\[
\underline N^{\underline l}(a)=\sum_{\underline m\geq \underline l}b_{\underline m}\underline{u}^{[\underline m-\underline l]}=\sum_{\underline m \in \mathbb N^d} b_{\underline m+\underline l} \underline{u}^ {[\underline{m}]}.
\]
As there are only finitely many non-zero coefficients $b_{\underline m}$, $\underline N^{\underline l}(a)=0$ in $\cO\AAcr/p^n$ when $|\underline l|\gg 0$. Meanwhile, note that $\sigma_2(T_i)-\sigma_1(T_i)\in p\cO_X^{\ur+}$, hence  their divided powers lie in $\cO_X^{\ur+}$. Therefore the series of the right-hand side of \eqref{phi2} applied to $a$ does make sense and gives a well-defined additive map on $\cO\AAcr/p^n$. Consequently, the series on the right-hand side of \eqref{phi2} gives a well-defined additive map on $\cO\AAcr$, which is also semilinear relative to the Frobenius on $\AAcr$.

It remains to verify the formula \eqref{phi2}. Since both sides of \eqref{phi2} are semilinear relative to the Frobenius of $\AAcr$, it suffices to check the equality for the $\underline{u}^{[\underline{m}]}$'s. In fact, we have
\[
\begin{array}{cl}
& \left(\sum_{(n_1,\ldots, n_d)\in \NN^d}(\prod_{i=1}^d(\sigma_2(T_i)-\sigma_1(T_i))^{[n_i]})(\varphi_1\circ(\prod_{i=1}^dN_i^{n_i}))\right)(u^{[\underline{m}]}) \\ = & \sum_{\underline n\in \mathbb N^d} (\sigma_2(\underline T)-\sigma_1(\underline T))^{[\underline n]}(\varphi_1(\underline N^{\underline n}(\underline{u}^{[\underline{m}]}))) \\ = & \sum_{\underline n\in \mathbb N^d} (\varphi_2(\underline{u})-\varphi_1(\underline{u}))^{[\underline n]}(\varphi_1(\underline N^{\underline n}(\underline{u}^{[\underline m]}))) \\ = & \sum_{\underline n\in \mathbb N^d \text{ s.t. }\underline n\leq \underline m} (\varphi_2(\underline{u})-\varphi_1(\underline{u}))^{[\underline n]}\cdot \varphi_1(\underline{u})^{[\underline m-\underline n]} \\ =& (\varphi_2(\underline{u})-\varphi_1(\underline{u})+\varphi_1(\underline{u}))^{[\underline m]} \\ = & \varphi_2(u^{[\underline m]}).
\end{array}
\]
This finishes the proof.
\end{proof}

\subsection{Comparison with de Rham period sheaves}

%The materials in this section will not be needed for the proof of the main theorem. We keep them for the fundamental roles they play in $p$-adic Hodge theory.
Let $X$ be a locally noetherian adic space over $\mathrm{Spa}(k,\cO_k)$ and recall the map $\theta$ in \eqref{theta}. Set
$\BBdrp=\varprojlim W(\cO_X^{\flat+})[1/p]/(\ker\theta)^n$, and $\BBdr=\BBdrp[1/t]$. For $r\in \mathbb Z$, let $\Fil^r\BBdr=t^r\BBdrp$. By its very definition, the filtration on $\BBdr$ is decreasing, separated and exhaustive. On the other hand, we can define the de Rham period sheaves with connection $\cO\BBdrp$ and $\cO\BBdr$ (see the erratum to \cite[Definition 6.8 (iii)]{Sch} available on Scholze's website). The filtration on $\cO\BBdrp$ is decreasing, separated and exhaustive. Moreover, as in \cite[5.2.8, 5.2.9]{Bri}, one shows that
\[
\cO\BBdrp\cap \Fil^r\cO\BBdr=\Fil^r\cO\BBdrp,
\]
implying that the filtration on $\cO\BBdr$ is also decreasing, separated and exhaustive.

In the rest of this subsection, assume that $k/\Qp$ is absolutely unramified.

\begin{prop}\label{BcrisBdR} Let $\cX$ be a smooth formal scheme over $\cO_k$.
\begin{enumerate}
\item There are injective filtered morphisms $\mathbb{B}_{\rm cris}^+\hookrightarrow \BBdr^+$ and $\cO\mathbb{B}_{\rm cris}^+\hookrightarrow \cO\BBdr^+$.
In this way, we view $\mathbb B_{\cris}^{+}$ and $\cO\mathbb B_{\cris}^{+}$ respectively as a subsheaf of rings of $\mathbb B_{\dR}^{+}$ and $\cO\mathbb B_{\dR}^{+}$.

\item For  any $i\in \mathbb N$, one has $
\Fil^i \mathbb B_{\cris}^{+}=\Fil^i\mathbb B_{\dR}^{+}\cap \mathbb B_{\cris}^{+}$ and $\Fil^i\cO\mathbb B_{\cris}^{+}=\Fil^i\cO\mathbb B_{\dR}^{+}\cap \cO\mathbb B_{\cris}^{+}$. In particular, the filtrations on $\BBcrp$ and on $\cO\BBcrp$ are decreasing, separated and exhaustive. Furthermore, the filtered morphisms in \emph{(1)} induce isomorphisms $
\mathrm{gr}^i \mathbb B_{\cris}^{+}\stackrel{\sim}{\ra}\mathrm{gr}^i\mathbb B_{\dR}^{+}$, and $\mathrm{gr}^i\cO\mathbb B_{\cris}^{+}\stackrel{\sim}{\ra}\mathrm{gr}^i\cO\mathbb B_{\dR}^{+}$.
\end{enumerate}
\end{prop}

\begin{proof} (1) Recall that $\mathbb B_{\dR}^{+}$ is a sheaf of $\Qp$-algebras, so the natural morphism $
\mathbb A_{\inf}=W(\cO_{X}^{\flat+})\rightarrow \mathbb B_{\dR}^{+}
$
extends to the PD-envelope $\AAcr^0$ of $\mathbb A_{\inf}$ with respect to the kernel of the map $\theta$ in \eqref{theta}. The resulting map $\AAcr^0\ra \mathbb B_{\dR}^+$ respects the filtrations. On the other hand, for each $n\in \mathbb N$, the composite
\[
\mathbb A_{\cris}^{0}\longrightarrow \mathbb B^{+}_{\dR}\longrightarrow \mathbb B_{\dR}^{+}/\Fil^n \mathbb B_{\dR}^{+}=W(\cO_{X}^{\flat+})[1/p]/(\ker(\theta))^n
\]
extends to the $p$-adic completion $\mathbb A_{\cris}$ of $\mathbb A_{\cris}^{0}$. Indeed, this is because the image of the composite above is contained in $\frac{1}{p^n}(W(\cO_{X}^{\flat+})/(\ker(\theta))^n)\subset \mathbb B_{\dR}^+/\Fil^n\mathbb B_{\dR}^+$ and the latter is $p$-adically complete. On passing to the projective limit relative to $n$, we obtain a filtered morphism $\mathbb A_{\cris}\to \mathbb B_{\dR}^{+}$, whence the required filtered morphism $\mathbb B_{\cris}^{+}\to \mathbb B_{\dR}^{+}$ by inverting $p\in \mathbb A_{\cris}$.

To define a natural filtered morphism $\cO\mathbb B_{\cris}^+\to \cO\mathbb B_{\dR}^{+}$, observe that $\cO\mathbb B_{\dR}^+$ is an algebra over $\cO_{X}\otimes_{\cO_k}W(\cO_{X}^{\flat+})$, so we have a natural morphism
\[
\cO\mathbb A_{\inf}=\cO_{X}^{\ur+}\otimes_{\cO_k}W(\cO_{X}^{\flat+}) \lra \cO\mathbb B_{\dR}^+,
\]
which extends to the PD-envelope $\cO\mathbb A_{\cris}^0$ of $\cO\mathbb A_{\inf}$ relative to the kernel of the map $\theta_X$ in \eqref{eq.thetaX}. For $n\in \mathbb N$, consider the composed morphism
\begin{equation}\label{eq.composite-OBcris-modulo-Filn}
\cO\mathbb A_{\cris}^0\lra \cO\mathbb B_{\dR}^+\lra \cO\mathbb B_{\dR}^+/\Fil^n\cO\mathbb B_{\dR}^+.
\end{equation}
As above, it extends to $\cO\mathbb A_{\cris}$, the $p$-adic completion of $\cO\mathbb A_{\cris}^0$.
To check this assertion, assume that $\cX$ is affine and \'etale over $\Spf(\cO_k\{T_1^{\pm 1}, \ldots, T_d^{\pm 1}\})$, and let $\widetilde{X}$ be the affinoid perfectoid obtained by joining to $X$ a compatible family of $p^n$-th roots of $T_i$ for $n\in \mathbb N$ and $1\leq i\leq d$. It suffices to show that, for any affinoid perfectoid $V$ above $\widetilde{X}_{\bk}$, the restriction to $V$ of \eqref{eq.composite-OBcris-modulo-Filn} extends in a uniquely way to a morphism $\cO\mathbb A_{\cris}|_{V}\ra (\cO\BBdrp/\Fil^n\cO\BBdrp)|_{V}$, with image contained in a $W(\cO_X^{\flat+})|_{V}$-submodule of finite type. By \cite[Proposition 6.10]{Sch}, we have $\cO\mathbb B_{\dR}^+|_{V}=\mathbb B_{\dR}^+|_{V}[[u_1,\ldots, u_d]]$, with $u_i=T_i\otimes 1-1\otimes[T^{\flat}_i]$. So
\[
\left(\cO\mathbb B_{\dR}^+/\Fil^n\cO\mathbb B_{\dR}^+\right)|_{V}=\bigoplus_{\underline m\in \mathbb N^d,|\underline m|\leq n}\left(W(\cO_X^{\flat+})[1/p]/(\xi)^{n-|m|}\right)|_{V}\cdot \underline u^{\underline m}.
\]
Through this identification, the image of \eqref{eq.composite-OBcris-modulo-Filn} (restricted to $V$) is contained in
\[
\frac{1}{p^a}\left(\bigoplus_{\underline m\in \mathbb N^d,|\underline m|\leq n}\left(W(\cO_{X}^{\flat+})/(\xi)^{n-|m|}\right)|_{V}\cdot \underline u^{\underline m}\right)\subset \left(\cO\mathbb B_{\dR}^+/\Fil^n\cO\mathbb B_{\dR}^+\right)|_{V}
\]
for some $a\in \mathbb N$. Since the latter is $p$-adically complete, the restriction to $V$ of \eqref{eq.composite-OBcris-modulo-Filn} extends to $\cO\mathbb A_{\cris}|_V$, and the image of this extension is contained in the $W(\cO_X^{\flat+})$-submodule of finite type above. If we have two such extensions, the images of both extensions are contained in some $W(\cO_X^{\flat+})$-submodule of finite type of the form above for some $a$ (large enough). As the latter is $p$-adically complete, these two extensions must coincide. This proves our assertion. So \eqref{eq.composite-OBcris-modulo-Filn} extends to a morphism $\cO\mathbb A_{\cris}\ra \cO\mathbb B_{\dR}^+/\Fil^n\cO\BBdrp$. Passing to the projective limit relative to $n$, we obtain the required filtered morphism $\cO\mathbb A_{\cris}\ra \cO\mathbb B_{\dR}^+$.

The two morphisms constructed above are compatible with the isomorphisms in Corollary \ref{BcrisIso} and its de Rham analogue \cite[Proposition 6.10]{Sch}. To finish the proof of (1), we only need to show the morphism $\mathbb B_{\cris}^{+}\to \mathbb B_{\dR}^+$ constructed above is injective. This can be done in the same way as \cite[Proposition 6.2.1]{Bri}, and we omit the detail here.

(2) By (1), the corresponding statement for $\mathbb B_{\cris}^+$ follows from the fact that the natural induced map
\[
\mathrm{gr}^r\mathbb B_{\cris}^+|_{X_{\bk}}=\widehat{\cO}_{X_{\bk}}\cdot (\xi^r/r!)\lra \mathrm{gr}^r\BBdrp|_{X_{\bk}}=\widehat{\cO}_{X_{\bk}}\cdot \xi^r
\]
is an isomorphism. To show the statements for $\cO\BBcrp$, assume that $\cX$ admits an \'etale map to $\Spf(\cO_k\{T_1^{\pm 1},\ldots, T_d^{\pm 1}\})$. Then we conclude by Corollary~\ref{BcrisIso} and its de Rham analogue, and by what we have just shown for $\BBcrp$.
\end{proof}

\begin{cor}\label{cor.no-t-torsion} Let $\cX$ be a smooth formal scheme over $\cO_k$. Then, over $X_{\proet}/X_{\bk}$, the sheaves of $A_{\cris}$-modules $\AAcr,\BBcrp, \cO\AAcr$ and $\cO\BBcrp$ have no $t$-torsion.
\end{cor}

\begin{proof} As $\AAcr$  and $\cO\AAcr$ have no $p$-torsion, they are included respectively in $\BBcrp$ and $\cO\BBcrp$. Hence, to prove our corollary, by Proposition \ref{BcrisBdR}, it is enough to show that, over $X_{\proet}/X_{\bk}$, $\BB_{\dR}^+$ and $\cO\BB_{\dR}^+$ have no $t$-torsion. These two statements are contained in \cite[Remark 6.2, Remark 6.9]{Sch}.
\end{proof}

\begin{cor}\label{GradedOfBcris}Let $\cX$ be a smooth formal scheme over $\cO_k$.
\begin{enumerate}
\item There are natural inclusions $
\mathbb B_{\cris}\hookrightarrow \mathbb B_{\dR}$ and $ \cO\mathbb B_{\cris}\hookrightarrow \cO\mathbb B_{\dR}$.

\item For any $i\in \mathbb Z$, we have $\Fil^i \mathbb B_{\cris}=\mathbb B_{\cris}\bigcap \Fil^i\mathbb B_{\dR}$ and $\Fil^i \cO\mathbb B_{\cris}=\cO\mathbb B_{\cris}\bigcap \Fil^i\cO\mathbb B_{\dR}$. In particular, the filtrations on $\BBcr$ and on $\cO\BBcr$ are decreasing, separated and exhaustive. Furthermore, the inclusions in \emph{(1)} induce isomorphisms $\mathrm{gr}^i\mathbb B_{\cris}\stackrel{\sim}{\ra} \mathrm{gr}^i\mathbb B_{\dR}$ and $\mathrm{gr}^i\cO\mathbb B_{\cris}\stackrel{\sim}{\ra}\mathrm{gr}^i\cO\mathbb B_{\dR}$.
\end{enumerate}
\end{cor}

%As a result, we can compute the cohomology of the graded quotients $\mathrm{gr}^{i}\mathcal F$ for $\mathcal F\in \{\BBcr^+, \BBcr,\cO\BBcrp, \cO\BBcr\}$: one reduces to its de Rham analogue such as \cite[Proposition 6.16]{Sch} etc.

\begin{cor}\label{higherox} Let $\cX$ be a smooth formal scheme over $\cO_k$, with $X$ its generic fiber. Then $w_*\cO\BBcr\simeq\cO_{\cX_{\et}}[1/p]$.
\end{cor}

\begin{proof}
Let $\nu: X_{\proet}^{\sim}\ra X_{\et}^{\sim}$ and $\nu' \colon X_{\et}^{\sim}\to \cX_{\et}^{\sim}$ the natural morphisms of topoi. Then $w=\nu'\circ \nu$. Therefore $
\cO_{\cX_{\et}}[1/p]\simto\nu_{\ast}'\cO_{X_{\et}}\simto \nu_{\ast}'\nu_{\ast}\cO_X=w_{\ast}\cO_{X}$. By \cite[Corollary 6.19]{Sch}, the natural map $\cO_{X_{\et}}\to \nu_{\ast}\cO\mathbb B_{\dR} $ is an isomorphism. Thus, $w_{\ast}\cO\mathbb B_{\dR}=\nu_{\ast}'(\nu_{\ast}\cO\mathbb B_{\dR})\simeq \nu_*'\cO_{X_{\et}}\simeq \cO_{\cX_{\et}}[1/p]$.
On the other hand, we have the injection of $\cO_{\cX_{\et}}[1/p]$-algebras $w_{\ast}\cO\mathbb B_{\cris}\hookrightarrow w_{\ast}\cO\mathbb B_{\dR}$. Thereby $\cO_{\cX_{\et}}[1/p]\simto w_{\ast}\cO\BBcr$.
\end{proof}

\section{Crystalline cohomology and pro-\'etale cohomology}

In this section, assume that $k$ is \emph{absolutely unramified}. Let $\sigma$ denote the Frobenius on $\cO_k$ and on $k$, lifting the Frobenius of the residue field $\kappa$. The ideal $(p)\subset \cO_k$ is endowed with a PD-structure and $\cO_k$ becomes a PD-ring in this way.

\subsection{A reminder on convergent $F$-isocrystals}
Let $\cX_0$ be a $\kappa$-scheme of finite type. %be a smooth formal scheme over $\cO_k$, with $X=\cX_k$ the generic fiber of $\cX$ in the sense of Huber, and $\cX_0$ its closed fiber.
Let us begin with some general definitions about crystals on the small crystalline site $\left(\mathcal{X}_0/\cO_k\right)_{\cris}$ endowed with \'etale topology. For basics of crystals, we refer to \cite{Ber}, \cite{BO}. Recall that a \emph{crystal of $\cO_{\cX_0/\cO_k}$-modules} is an $\cO_{\cX_0/\cO_k}$-module $\mathbb E$ on $(\cX_0/\cO_k)_{\cris}$ such that (i) for any object $(U,T)\in (\cX_0/\cO_k)_{\cris}$, the restriction $\mathbb E_{T}$ of $\mathbb E$ to the \'etale site of $T$ is a coherent $\cO_T$-module; and (ii) for any morphism $u:(U',T')\ra (U,T)$ in $\left(\mathcal{X}_0/\cO_k\right)_{\cris}$, the canonical morphism $u^*\mathbb{E}_{T}\simto \mathbb{E}_{T'}$ is an isomorphism.

\begin{rk}\label{rk.Crystals} Let $\cX_0$ be the closed fiber of a smooth formal scheme $\cX$ over $\cO_k$. The category of crystals on $(\cX_0/\cO_k)_{\cris}$ is equivalent to that of coherent $\cO_{\cX}$-modules $\mathcal M$ equipped with an integrable and quasi-nilpotent connection $\nabla\colon \mathcal M\to \mathcal M\otimes_{\cO_{\cX}}\Omega^1_{\cX/\cO_k}$. Here the connection $\nabla$ is said to be \emph{quasi-nilpotent} if its reduction modulo $p$ is quasi-nilpotent in the sense of \cite[Definition 4.10]{BO}. \if false The correspondence between these two categories is given as follows. For $\mathbb E$ a crystal on $\cX_0/\cO_k$, as $\cX_0\hookrightarrow \cX$ is a $p$-adic PD-thickening, we can evalue $\mathbb E$ at it: set $
\mathbb E_{\cX}:=\varprojlim_n \mathbb E_{\cX\otimes \cO_k/p^n}$. Let $\Delta_1\hookrightarrow \cX\times \cX$ be the PD-thickening of order $1$ of the diagonal embedding $\cX\hookrightarrow \cX\times \cX$. The two projections $p_i\colon \Delta_1\to \cX$ are PD-morphisms. So we have two isomorphisms $p_{i}^{\ast}\mathbb E_{\cX}\simto \mathbb E_{\Delta_1}:=\varprojlim_n \mathbb E_{\Delta_1\otimes \cO_k/p^n}$, whence a natural isomorphism $p_{2}^{\ast}\mathbb E_{\cX}\simto p_{1}^{\ast} \mathbb E_{\cX}$. The latter gives a connection $\nabla\colon \mathbb E_{\cX}\rightarrow \mathbb E_{\cX}\otimes \Omega^1_{\cX/\cO_k}$ on $\mathbb E_{\cX}$. Together with a limit argument, that $\nabla$ is integrable and quasi-nilpotent is due to \cite[Theorem 6.6]{BO}.\fi
\end{rk}

The absolute Frobenius $F\colon \cX_0\to \cX_0$ is a morphism over the Frobenius $\sigma$ on $\cO_k=W(\kappa)$, hence it induces a morphism of topoi, still denoted by $F$:
\[
F \colon \left(\cX_0/\cO_k\right)_{\cris}^{\sim}\longrightarrow \left(\cX_0/\cO_k\right)_{\cris}^{\sim}.\]
An \emph{$F$-crystal} on $(\cX_0/\cO_k)_{\cris}$ is a crystal $\mathbb E$  equipped with a morphism $\varphi\colon F^{\ast}\mathbb E\ra \mathbb E$ of $\cO_{\cX_0/\cO_k}$-modules, which is nondegenerate, i.e. there exists a map $V:\mathbb{E}\ra F^{\ast}\mathbb E$ of $\cO_{\cX_0/\cO_k}$-modules such that $\varphi V=V\varphi=p^m$ for some $m\in \N$. In the following, we will denote by $F\textrm{-Cris}(\cX_0,\cO_k)$ the category of $F$-crystals on $\cX_0/\cO_k$.

Before discussing isocrystals, let us observe the following facts.

\begin{rk}\label{O[1/p]} (1) Let $\cX$ be a quasi-compact smooth formal scheme over $\cO_k$. Let $X^{\rm rig}$ be its rigid generic fiber, which is a rigid analytic variety over $k$. Let $\mathrm{Coh}(\cO_{\cX}[1/p])$ denote the category of coherent $\cO_{\cX}[1/p]$-modules on $\cX$, or equivalently, the category of coherent sheaves on $\cX$ up to isogeny. Denote by $\mathrm{Coh}(X^{\rm rig})$ the category of coherent sheaves on $X^{\rm rig}$. Then, we have the functor below, obtained by taking the rigid generic fiber of a coherent $\cO_{\cX}[1/p]$-module
\[
\mathrm{Coh}(\cO_{\cX}[1/p]) \lra \mathrm{Coh}(X^{\rm rig}).
\]
This is an equivalence of categories. Indeed, it is a consequence of the fact that any coherent sheaf on $X^{\rm rig}$ extends to a coherent sheaf on $\cX$ (\cite[Lemma 2.2]{Lut}).

(2) Let $Y$ be a rigid analytic variety. In \cite[Proposition 4.3]{Hub94}, Huber constructed from $Y$ an adic space $Y^{\ad}$, together with a locally coherent morphism  $\rho:(|Y^{\ad}|,\cO_{Y^{\ad}})\ra (|Y|,\cO_{Y})$ of ringed sites, satisfying some universal property. We call $Y^{\ad}$ the \emph{associated adic space} of $Y$. The morphism $\rho$ gives rise to an equivalence between the category of sheaves on the Grothendieck site associated to $Y$ and that of sheaves on the sober topological space $Y^{\ad}$ (\cite[1.1.11]{Hub96}). Moreover, under this equivalence, the notion of coherent sheaves on the ringed site $Y=(|Y|,\cO_Y)$ is the same as the one on $Y^{\ad}=(|Y^{\ad}|,\cO_{Y^{\ad}})$ since in the case of $Y=\Sp A$, with $A$ a complete topologically finitely generated Tate algebra over $k$, both of them are naturally equivalent to that of finite $A$-modules. (Cf. \cite[Theorem 9.1]{Sch} and the references therein.)
\end{rk}

Let $\cX_0$ be a $\ka$-scheme of finite type, embedded as a closed subscheme into a smooth formal scheme $\mathcal P$ over $\cO_k$. Let $P$ be the adic generic fiber of $\mathcal P$ and $]\cX_0[_{\mathcal P}\subset P$ the pre-image of the closed subset $\cX_0\subset \mathcal P$ under the specialization map. Following \cite[2.3.2 (i)]{Ber} (with Remark \ref{O[1/p]} (2) in mind), the \emph{realization on $\mathcal P$ of a convergent isocrystal} on $\cX_0/\cO_k$ is a coherent $\cO_{]\cX_0[_{\mathcal P}}$-module $\mathcal E$ equipped with an integrable and \emph{convergent} connection $\nabla\colon \mathcal E\rightarrow \mathcal E\otimes_{\cO_{]\cX_0[_{\mathcal P}}} \Omega^{1}_{]\cX_0[_{\mathcal P}/k}$ (we refer to \cite[2.2.5]{Ber} for the definition of convergent connections). Being a coherent $\cO_{]\cX_0[_{\mathcal P}}$-module with integrable connection, $\cE$ is locally free of finite rank by \cite[2.2.3 (ii)]{Ber}. The category of realizations on $\cP$ of convergent isocrystals on $\cX_0/\cO_k$ is denoted by $\mathrm{Isoc}(\cX_0/\cO_k,\cP)$, where the morphisms are morphisms of $\cO_{]\cX_0[_{\mathcal P}}$-modules which commute with connections.

Let $\cX_0\hookrightarrow \cP' $ be a second embedding of $\cX_0$ into a smooth formal scheme $\cP' $ over $\cO_k$, and assume there exists a morphism $u\colon \cP'\to \cP $ of formal schemes inducing identity on $\cX_0$. The generic fiber of $u$ gives a morphism of adic spaces $u_k\colon ]\cX_0[_{\mathcal P'}\to ]\cX_0[_{\mathcal P}$, hence a natural functor
\[
u_k^{\ast}\colon \mathrm{Isoc}(\cX_0/\cO_k,\cP) \longrightarrow \mathrm{Isoc}(\cX_0/\cO_k,\cP'), \quad (\cE,\nabla)\mapsto (u_k^{\ast}\cE, u_k^{\ast}\nabla).
\]
By \cite[2.3.2 (i)]{Ber}, the functor $u_k^{\ast}$ is an equivalence of categories. Furthermore, for a second morphism $v\colon \cP'\to\cP$ of formal schemes inducing identity on $\cX_0$, the two equivalence $u_k^{\ast},v_k^{\ast}$ are canonically isomorphic (\cite[2.2.17 (i)]{Ber}). The category of \emph{convergent isocrystal on $\cX_0/\cO_k$}, denoted by $\mathrm{Isoc}(\cX_0/\cO_k)$, is defined as
\[
\mathrm{Isoc}(\cX_0/\cO_k):=\textrm{2-}\!\varinjlim_{\cP} \mathrm{Isoc}(\cX_0/\cO_k,\cP),
\]
where the limit runs through all smooth formal embeddings $\cX_0\hookrightarrow \cP$ of $\cX_0$.

\begin{rk} In general, $\cX_0$ does not necessarily admit a global formal embedding. In this case, the category of convergent isocrystals on $\cX_0/\cO_k$ can still be defined by a gluing argument (see \cite[2.3.2(iii)]{Ber}). But the definition recalled above will be enough for our purpose.
\end{rk}

As for the category of crystals on $\cX_0/\cO_k$, the Frobenius morphism $F\colon \cX_0\to \cX_0$ induces a natural functor (see \cite[2.3.7]{Ber} for the construction):
\[
F^{\ast}\colon \mathrm{Isoc}(\cX_0/\cO_k)\longrightarrow \mathrm{Isoc}(\cX_0/\cO_k).
\]
A \emph{convergent $F$-isocrystal} on $\cX_0/\cO_k$ is a convergent isocrystal $\mathcal E$ on $\cX_0/\cO_k$ equipped with an isomorphism  $F^*\mathcal E \simto \mathcal E$ in $\mathrm{Isoc}(\cX_0/\cO_k)$. The category of convergent $F$-isocrystals on $\cX_0/\cO_k$ will be denoted in the following by $F\textrm{-Isoc}(\cX_0/\cO_k)$.

\begin{rk} \label{rk.CrystalVSIsoc}The category $F\textrm{-Isoc}(\cX_0/\cO_k)$ has as a full subcategory the isogeny category $F\textrm{-Cris}(\cX_0/\cO_k)\otimes\mathbb Q$ of $F$-crystals $\mathbb E$ on $\left(\mathcal{X}_0/\cO_k\right)_{\cris}$. To explain this, assume for simplicity that $\cX_0$ is the closed fiber of a smooth formal scheme $\cX$ over $\cO_k$. So $]\cX_0[_{\cX}=X$, the generic fiber of $\cX$. Let $(\cM,\nabla)$ be the $\cO_{\cX}$-module with integrable and quasi-nilpotent connection associated to the $F$-crystal $\mathbb E$ (Remark \ref{rk.Crystals}). Let $\mathbb E^{\an}$ denote the generic fiber of $\cM$, which is a coherent (hence locally free by \cite[2.3.2 (ii)]{Ber}) $\cO_{X_{\an}}$-module equipped with an integrable connection $\nabla^{\rm an}\colon \mathbb E^{\rm an}\longrightarrow \mathbb E^{\rm an}\otimes \Omega^{1}_{X_{\an}/k}$, which is nothing but the generic fiber of $\nabla$. Because of the $F$-crystal structure on $\mathbb E$, the connection $\nabla^{\rm an}$ is convergent (\cite[2.4.1]{Ber}). In this way we obtain an $F$-isocrystal $\mathbb E^{\mathrm{an}}$ on $\cX_0/\cO_k$, whence a natural functor
\begin{equation*}
(-)^{\rm an}\colon F\textrm{-Cris}(\cX_0/\cO_k) \otimes \mathbb Q\longrightarrow F\textrm{-Isoc}(\cX_0/\cO_k), \quad \mathbb E\mapsto \mathbb E^{\rm an}.
\end{equation*}
By \cite[2.4.2]{Ber}, this analytification functor is fully faithful, and for $\mathcal E$ a convergent $F$-isocrystal on $\cX_0/\cO_k$, there exists an integer $n\geq 0$ and an $F$-crystal $\mathbb E$ such that $\mathcal E\simto \mathbb E^{\rm  an}(n)$, where for $\mathcal F=(\mathcal F, \nabla, \varphi\colon F^{\ast}\mathcal F\simto \mathcal F)$ an $F$-isocrystal on $\cX_0/\cO_k$, $\mathcal F(n)$ denotes the Tate twist of $\mathcal F$, given by $(\mathcal F, \nabla, \frac{\varphi}{p^n}\colon F^{\ast}\mathcal F\simto \mathcal F)$  (\cite[2.3.8 (i)]{Ber}).
\end{rk}

Our next goal is to give a more explicit description of the Frobenius morphisms on convergent $F$-isocrystals on $\cX_0/\cO_k$. From now on, assume for simplicity that \emph{$\cX_0$ is the closed fiber of a smooth formal scheme $\cX$} and we identify convergent isocrystals on $\cX_0/\cO_k$ with their realizations on $\cX$. Let $X$ be the adic generic fiber of $\cX$. The proof of the following lemma is obvious.

\begin{lemma}\label{lem.FCrystal} Assume that the Frobenius $F\colon \cX_0\to \cX_0$  can be lifted to a morphism $\sigma \colon \cX\to \cX$ compatible with the Frobenius on $\cO_k$.  Still denote by $\sigma$ the endomorphism on $X$ induced by $\sigma$. Then there is an equivalence of categories between
\begin{enumerate}
\item the category $F\textrm{-}\mathrm{Isoc}(\cX_0/\cO_k)$ of convergent $F$-isocrystals on $\cX_0/\cO_k$; and
\item the category $\mathbf{Mod}_{\cO_{X}}^{\sigma,\nabla}$ of  $\cO_{X_{\an}}$-vector bundles $\mathcal E$ equipped with an integrable and convergent connection $\nabla$ and an $\cO_{X_{\an}}$-linear horizontal isomorphism $\varphi\colon \sigma^*\cE\to \cE$.
\end{enumerate}
\end{lemma}

%\begin{proof}
%This follows directly from the definitions and functoriality; see \cite[2.3.2 (iv)]{Ber}.
%\end{proof}

Consider two liftings of Froebnius $\sigma_i$ ($i=1,2$) on $\cX$. By the lemma above, for $i=1,2$, both categories $\mathbf{Mod}_{\cO_{X}}^{\sigma_i,\nabla}$ are naturally equivalent to the category of convergent $F$-iscrystals on $\cX_0/\cO_k$:
\[
\xymatrix{\mathbf{Mod}_{\cO_{X}}^{\sigma_1,\nabla} & F\textrm{-}\mathrm{Isoc}(\cX_0/\cO_k)\ar[r]^<<<<<{\sim}\ar[l]_<<<<<{\sim} & \mathbf{Mod}_{\cO_{X}}^{\sigma_2,\nabla}}.
\]
Therefore we deduce an equivalence of categories
\begin{equation}\label{eq.FunctorF}
F_{\sigma_1,\sigma_2}\colon \mathbf{Mod}_{\cO_{X}}^{\sigma_1,\nabla} \longrightarrow \mathbf{Mod}_{\cO_{X}}^{\sigma_2,\nabla}.
\end{equation}
When our formal scheme $\cX$ is small, we can explicitly describe this equivalence.  Assume there is an \'etale morphism $\cX\to \mathcal T^d=\Spf(\cO_k\{T_1^{\pm 1},\ldots, T_d^{\pm 1}\})$. So $\Omega_{X_{\an}/k}^{1}$ is a free $\cO_{X_{\an}}$-module with a basis given by $dT_i$ ($i=1,\ldots, d$). In the following, for $\nabla$ a connection on an $\cO_{X_{\an}}$-module $\mathcal E$, let $N_i$ be the endomorphism of $\cE$ (as an abelian sheaf) such that  $\nabla=\sum_{i=1}^{d}N_i\otimes dT_i$.

\begin{lemma}[\cite{Bri}, Proposition 7.2.3] \label{fsigma12}
Assume that $\cX=\Spf(A)$ is affine, admitting an \'etale morphism $\cX\to \mathcal T^d$ as above. Let $(\mathcal E,\nabla, \varphi_1)\in \mathbf{Mod}_{\cO_{X}}^{\sigma_1,\nabla}$, with $(\mathcal E,\nabla, \varphi_2)$ the corresponding object of $\mathbf{Mod}_{\cO_{X}}^{\sigma_2,\nabla}$ under the equivalence $F_{\sigma_1,\sigma_2}$. Then on $\cE(X)$ we have
\begin{equation*}
\varphi_2=\sum_{(n_1,\ldots, n_d)\in \NN^d}\left(\prod_{i=1}^d(\sigma_2(T_i)-\sigma_1(T_i))^{[n_i]}\right)\left(\varphi_1\circ \left(\prod_{i=1}^dN_i^{n_i}\right)\right).
\end{equation*}
Furthermore, $\varphi_1$ and $\varphi_2$ coincide on $\cE(X)^{\nabla=0}$.
\end{lemma}

More generally, i.e., without assuming the existence of Frobenius lifts to $\cX$,  for $(\cE,\nabla)$ an $\cO_{X_{\an}}$-module with integrable and convergent connection, \emph{a compatible system of Frobenii on $\cE$} consists of, for any open subset $\cU\subset \cX$ equipped with a lifting of Frobenius $\sigma_{\cU}$, a horizontal isomorphism $\varphi_{(\cU,\sigma_{\cU})}\colon \sigma_{\cU}^{\ast} \cE|_{\cU_k}\to \cE|_{\cU_k}$ satisfying the following condition: for $\cV\subset \cX$ another open subset equipped with a lifting of Frobenius $\sigma_{\cV}$, the functor
\[
F_{\sigma_{\cU},\sigma_{\cV}} \colon \mathbf{Mod}_{\cO_{\cU_k \bigcap \cV_k}}^{\sigma_{\cU},\nabla} \longrightarrow \mathbf{Mod}_{\cO_{\cU_k\bigcap \cV_k}}^{\sigma_{\cV},\nabla}
\]
sends $(\cE|_{\cU_k\bigcap \cV_k}, \nabla, \varphi_{(\cU,\sigma_{\cU})}|_{\mathcal U_k\bigcap \cV_k})$ to $(\cE|_{\mathcal U_k\bigcap \cV_k}, \nabla, \varphi_{(\cV,\sigma_{\cV})}|_{\mathcal U_k\bigcap \cV_k})$. We denote a compatible system of Frobenii on $\cE$ by the symbol $\varphi$, when no confusion arises.
Let $\mathbf{Mod}_{\cO_{X}}^{\sigma,\nabla}$ be the category of $\cO_{X_{\an}}$-vector bundles equipped with an integrable and convergent connection, and with a compatible system of Frobenii. The morphism in $\mathbf{Mod}_{\cO_{X}}^{\sigma,\nabla}$ are the morphisms of $\cO_{X_{\an}}$-modules which commute with the connections, and with the Frobenius morphisms on any open subset $\mathcal U\subset \cX$ equipped with a lifting of Frobenius.

\begin{rk} Let $\cE$ be a convergent isocrystal on $\cX_0/\cO_k$. To define a compatible system of Frobenii on $\cE$, we only need to give, for a cover $\cX=\bigcup_i \cU_i$ of $\cX$ by open subsets $\cU_i$ equipped with a lifting of Frobenius $\sigma_{i}$, a family of horizontal isomorphisms $\varphi_{i}\colon \sigma_i^{\ast}\cE|_{U_i} \stackrel{\sim}{\to} \cE|_{U_i}$ such that $\varphi_i|_{U_{i}\bigcap U_{j}}$ corresponds to $\varphi_j|_{U_{i}\bigcap U_{j}}$ under the functor $F_{\sigma_i,\sigma_j}\colon \mathbf{Mod}_{\cO_{U_{i}\bigcap U_{j}}}^{\sigma_i,\nabla}\to  \mathbf{Mod}_{\cO_{U_{i}\bigcap U_{j}}}^{\sigma_j,\nabla}$ (Here $U_{\bullet}:=\cU_{\bullet,k}$). Indeed, for $\cU$ any open subset equipped with a lifting of Frobenius $\sigma_{\cU}$, one can first use the functor $F_{\sigma_i,\sigma_{\cU}}$ of \eqref{eq.FunctorF} applied to $(\cE|_{U_i}, \nabla|_{U_i}, \varphi_i)|_{U_i\bigcap U}$ to obtain a horizontal isomorphism $\varphi_{\cU,i}\colon (\sigma_{\cU}^{\ast}(\cE|_{\cU}))|_{U_i\bigcap U} \to \cE|_{U_i\bigcap U}$. From the compatibility of the $\varphi_i$'s, we deduce $
\varphi_{\cU,i}|_{U\bigcap U_{i}\bigcap U_{j}}=\varphi_{\cU,j}|_{U\bigcap U_{i}\bigcap U_{j}}$.
Consequently we can glue the $\varphi_{\cU,i}$'s ($i\in I$) to get a horizontal isomorphism $\varphi_{\cU}\colon \sigma_{\cU}^{\ast}(\cE|_{U})\to \cE|_{U}$. One checks that these $\varphi_{\cU}$'s give the desired compatible system of Frobenii on $\cE$.
\end{rk}

Let $\cE$ be a convergent $F$-isocrystal on $\cX_0/\cO_k$. For $\cU\subset \cX$ an open subset equipped with a lifting of Frobenius $\sigma_{\cU}$, the restriction $\cE|_{\cU_k}$ gives rise to  a convergent $F$-isocrystal on $\cU_0/k$. Thus there exists a $\nabla$-horizontal isomorphism $\varphi_{(\cU, \sigma_{\cU})}\colon \sigma_{\cU}^{\ast}\cE|_{\cU_k}\to \cE|_{\cU_k}$. Varying $(\cU,\sigma_{\cU})$ we obtain a compatible system of Frobenii $\varphi$ on $\cE$. In this way, $(\cE,\nabla,\varphi)$ becomes an object of $\mathbf{Mod}_{\cO_{X}}^{\sigma, \nabla}$. Directly from the definition, we have the following

\begin{cor} The natural functor $F\textrm{-}\mathrm{Isoc}(\cX_0/\cO_k)\to \mathbf{Mod}_{\cO_{X}}^{\sigma,\nabla}$ is an equivalence of categories.
\end{cor}

 In the following, denote by $\mathbf{FMod}_{\cO_{X}}^{\sigma,\nabla}$ the category of  quadruples $(\cE, \nabla, \varphi, \Fil^{\bullet}(\cE))$ with $(\cE, \nabla, \varphi)\in \mathbf{Mod}_{\cO_{X}}^{\sigma,\nabla}$ and  a decreasing, separated and exhaustive filtration $\Fil^{\bullet}(\cE)$ on $\cE$ by locally free direct summands, such that $\nabla$ satisfies Griffiths transversality with respect to $\Fil^{\bullet}(\cE)$, i.e., $
\nabla (\Fil^{i}(\cE))\subset \Fil^{i-1}(\cE)\otimes_{\cO_{X_{\an}}}\Omega_{X_{\an}/k}^{1}$.
The morphisms are the morphisms in $\mathbf{Mod}_{\cO_{X}}^{\sigma,\nabla}$ which respect the filtrations. We call the objects in $\mathbf{FMod}_{\cO_{X}}^{\sigma,\nabla}$ \emph{filtered (convergent) $F$-isocrystals} on $\cX_0/\cO_k$. By analogy with the category $F\textrm{-}\mathrm{Isoc}(\cX_0/\cO_k)$ of $F$-isocrystals, we also denote the category of filtered $F$-isocrystals on $\cX_0/\cO_k$ by $FF\textrm{-}\mathrm{Isoc}(\cX_0/\cO_k)$.

\subsection{Lisse $\widehat{\ZZ}_p$-sheaves and filtered $F$-isocystals}

Let $\cX$ be a smooth formal scheme over $\cO_k$ with $X$ its adic generic fiber. Define $\widehat{\ZZ}_p:=\varprojlim \ZZ/p^n$ and $\widehat{\QQ}_p:=\widehat{\ZZ}_p[1/p]$ as sheaves on $X_{\proet}$. Recall that
a \emph{lisse $\Zp$-sheaf} on $X_{\et}$ is an inverse system of abelian sheaves $\LL_{\bullet}=(\LL_n)_{n\in \mathbb{N}}$ on $X_{\et}$ such that each $\LL_n$ is locally a constant sheaf associated to a finitely generated $\ZZ/p^n$-modules, and such that $\LL_{\bullet}$ is isomorphic in the pro-category to such an inverse system for which $\LL_{n+1}/p^n\simeq \LL_n$. A \emph{lisse $\widehat{\ZZ}_p$-sheaf} on $X_{\proet}$ is a sheaf of $\widehat{\ZZ}_p$-modules on $X_{\proet}$, which  is locally isomorphic to $\widehat{\ZZ}_p\otimes_{\Zp} M$ where $M$ is a finitely generated $\Zp$-module. By \cite[Proposition 8.2]{Sch} these two notions are equivalent via the functor $\nu^{-1}\colon X_{\et}^{\sim}\to X_{\proet}^{\sim}$. In the following, we use frequently the natural  morphism of topoi
\[
w\colon X_{\proet}^{\sim}\stackrel{\nu}{\ra} X_{\et}^{\sim}\rightarrow \cX_{\et}^{\sim}.
\]
Before defining crystalline sheaves, let us make the following observation.

\begin{rk}\label{rk.wE} (1) Let $\mathcal M$ be a crystal on $\cX_0/\cO_k$, viewed as a coherent $\cO_{\cX}$-module admitting an integrable connection. Then $w^{-1}\mathcal M$ is a coherent $\cO_X^{\ur+}$-module with an integrable connection $w^{-1}\mathcal M\to w^{-1}\mathcal{M}\otimes_{\cO_X^{\ur+}}\Omega_{X/k}^{1,\ur+}$. If furthermore $\cM$ is an $F$-crystal, then $w^{-1}\mathcal M$ inherits a system of Frobenii: for any open subset $\cU\subset \cX$ equipped with a lifting of Frobenius $\sigma_{\cU}$, there is naturally an endomorphism of $w^{-1}\cM|_{U}$ which is semilinear with respect to the Frobenius $w^{-1}\sigma_{\cU}$ on $\cO_X^{\ur +}|_{U}$ (here $U:=\cU_k$). Indeed, the Frobenius structure on $\cM$ gives a horizontal $\cO_{\cU}$-linear morphism $ \sigma_{\cU}^{\ast}\cM|_{\cU}\to \cM|_{\cU}$, or equivalently, a $\sigma_{\cU}$-semilinear morphism $\varphi_{\cU}\colon \cM|_{\cU}\to \cM|_{\cU}$ (as $\sigma_{\cU}$ is the identity map on the underlying topological space). So we obtain a natural endomorphism $w^{-1}\varphi_{\cU}$ of $w^{-1}\cM|_{U}$, which is $w^{-1}\sigma_{\cU}$-semilinear.

(2) Let $\cE$ be a convergent $F$-isocrystal on $\cX_0/\cO_k$. By Remark \ref{rk.CrystalVSIsoc}, there exists an $F$-crystal $\mathcal M$ on $\cX_0/\cO_k$ and $n\in \mathbb N$ such that $\cE\simeq \mathcal M^{\rm an}(n)$. By (1), $w^{-1}\cM$ is a coherent $\cO_{X}^{\ur+}$-module equipped with an integrable connection and a compatible system of Frobenii $\varphi$. Inverting $p$, we get an $\cO_{X}^{\ur}$-module $w^{-1}\cM[1/p]$ equipped with an integrable connection and a system of Frobenii $\varphi /p^n$, which does not depend on the choice of the formal model $\cM$ or the integer $n$. For this reason, abusing notation, let us denote $w^{-1}\cM[1/p]$ by $w^{-1}\cE$, which is equipped with an integrable connection and a system of Frobenii inherited from $\cE$. If furthermore $\cE$ has a descending filtration $\{\Fil^i \cE\}$ by locally direct summands, by Remark \ref{O[1/p]} (1), each $\Fil^i\cE$ has a coherent formal model $\cE_i^+$ on $\cX$. Then $\{w^{-1}\cE_i^{+}[1/p]\}$ gives a descending filtration by locally direct summands on $w^{-1}\cE$.
\end{rk}

\begin{defn}\label{associated}
 We say a lisse $\widehat{\ZZ}_p$-sheaf $\LL$  on $X_{\proet}$  is \emph{crystalline} if there exists a filtered $F$-isocrystal $\cE$ on $\cX_0/\cO_k$, together with an isomorphism of $\cO\BBcr$-modules
\begin{equation}\label{eq.associated}
w^{-1}\cE\otimes_{\cO_X^{\ur}} \cO\BBcr\simeq \LL\otimes_{\widehat{\ZZ}_p}\cO\BBcr
\end{equation}
which is compatible with connection, filtration and Frobenius. In this case, we say that the lisse $\widehat{\ZZ}_p$-sheaf $\LL$ and the filtered $F$-isocrystal $\cE$ are \emph{associated}.
\end{defn}

\begin{rk} \label{abuseem} The Frobenius compatibility of the isomorphism \eqref{eq.associated} means the following. Take any open subset $\mathcal U\subset \cX$ equipped with a lifting of Frobenius $\sigma \colon \mathcal U\to \mathcal U$. By the discussion in \S \ref{localfrob}, we know that $\cO\mathbb B_{\cris}|_{\mathcal U_k}$ is naturally endowed with a Frobenius $\varphi$. Meanwhile, as $\mathcal E$ is an $F$-isocrystal, by Remark \ref{rk.wE} $w^{-1}\cE|_{\cU_k}$ is endowed with a $w^{-1}\sigma$-semilinear Frobenius, still denoted by $\varphi$. Now the required Frobenius compatibility means that when restricted to any such $\mathcal U_k$, we have $\varphi\otimes \varphi=\mathrm{id}\otimes \varphi$ via the isomorphism \eqref{eq.associated}.
\end{rk}

\begin{defn} For $\LL$ a lisse $\widehat{\ZZ}_p$-sheaf and $i\in \mathbb Z$, set
\[
\mathbb D_{\cris}(\LL):=w_{\ast}(\LL\otimes_{\widehat{\ZZ}_p}\cO\BBcr), \quad \text{and} \quad \Fil^i\mathbb D_{\cris}(\LL):=w_{\ast}(\LL\otimes_{\widehat{\ZZ}_p}\Fil^i\cO\BBcr).
\]
All of them are $\cO_{\cX}[1/p]$-modules, and the $\Fil^i\mathbb D_{\cris}(\LL)$'s give a separated exhaustive decreasing filtration on $\mathbb D_{\cris}(\LL)$ (as the same holds for the filtration on $\cO\BBcr$; see Corollary \ref{GradedOfBcris}).
\end{defn}

Next we shall compare the notion of crystalline sheaves with other related notions considered in \cite[Chapitre 8]{Bri}, \cite{Fal} and \cite{Sch}. We begin with the following characterization of crystalline sheaves, which is more closely related to the classical definition of crystalline representations by Fontaine (see also \cite[Chapitre 8]{Bri}).

\begin{prop}\label{prop.CrysSheaf} Let $\LL$ be a lisse $\widehat{\Z}_p$-sheaf on $X_{\proet}$. Then $\LL$ is crystalline if and only if the following two conditions are verified:
\begin{enumerate}
\item the $\cO_{\cX}[1/p]$-modules $\mathbb D_{\cris}(\LL)$ and $\Fil^i\mathbb D_{\cris}(\LL)$, $i\in \mathbb Z$, are all coherent.
\item the adjunction morphism $w^{-1}\mathbb D_{\cris}(\LL)\otimes_{\cO_X^{\ur}} \cO\BBcr\to \LL\otimes_{\widehat{\ZZ}_p}\cO\BBcr$ is an isomorphism of $\cO\BBcr$-modules.
\end{enumerate}
\end{prop}

Before proving this proposition, let us express locally the sheaf $\mathbb D_{\cris}(\LL)=w_{\ast}(\LL\otimes \cO\BBcr)$ as the Galois invariants of some Galois module. Consider $\cU=\Spf(A)\subset \cX$ a connected affine open subset admitting an \'etale map $\cU\to \Spf(\cO_k\{T_1^{\pm 1},\ldots, T_d^{\pm 1}\})$. Write $U$ the generic fiber of $\cU$. As $\cU$ is smooth and connected, $A$ is an integral domain. Fix an algebraic closure $\Omega$ of $\mathrm{Frac}(A)$, and let $\overline{A}$ be the union of finite and normal $A$-algebras $B$ contained in $\Omega$ such that $B[1/p]$ is \'etale over $A[1/p]$. Write $G_U:=\mathrm{Gal}(\overline{A}[1/p]/A[1/p])$, which is nothing but the fundamental group of $U=\cU_k$. Let $U^{\mathrm{univ}}$ be the profinite \'etale cover of $U$ corresponding to $(\overline A[1/p], \overline A)$. One checks that $U^{\mathrm{univ}}$ is affinoid perfectoid (over the completion of $\overline k$). As $\LL$ is a lisse $\widehat{\ZZ}_p$-sheaf on $X$, its restriction to $U$ corresponds to a continuous $\Zp$-representation $V_{U}(\LL):=\LL(U^{\mathrm{univ}})$ of $G_U$. Write $\widehat{U^{\rm univ}}=\Spa(R,R^+)$, where $(R,R^+)$ is the $p$-adic completion of $(\overline A[1/p], \overline A)$.

\begin{lemma} \label{twodcris} Keep the notation above. Let $\LL$ be a lisse $\widehat{\ZZ}_p$-sheaf on $X$.
Then there exist natural isomorphisms of $A[1/p]$-modules
\[
%w_{\ast}(\LL\otimes \cO\BBcr)(\cU)=
\mathbb D_{\cris}(\LL)(\cU)\stackrel{\sim}{\longrightarrow} \left(V_{U}(\LL)\otimes_{\Zp}\cO\BBcr(R,R^+)\right)^{G_U}=:D_{\cris}(V_{U}(\LL))
\]
and, for any $r\in \ZZ$,
\[
\left(\Fil^r \mathbb D_{\cris}(\LL)\right)(\cU)\stackrel{\sim}{\longrightarrow} \left(V_{U}(\LL)\otimes_{\Zp}\Fil^r\cO\BBcr(R,R^+)\right)^{G_U}.
\]
Moreover, the $A[1/p]$-module $\mathbb D_{\cris}(\LL)(\cU)$ is projective of rank at most that of $V_{U}(\LL)\otimes \mathbb Q_p$ over $\mathbb Q_p$.
\end{lemma}
\begin{proof} As $\LL$ is a lisse $\widehat{\ZZ}_p$-sheaf, it becomes constant restricted  to $U^{\mathrm{univ}}$. In other words, we have $
\LL|_{U^{\mathrm{univ}}}\simeq V_{U}(\LL)\otimes_{\Zp}\widehat{\ZZ}_p|_{U^{\mathrm{univ}}}$.
For $i\geq 0$ an integer, denote by $U^{\mathrm{univ},i}$ the (i+1)-fold product of $U^{\mathrm{univ}}$ over $U$. Then $
U^{\mathrm{univ},i}\simeq U^{\mathrm{univ}}\times G_U^i$, and it is again an affinoid perfectoid.
We claim that there exists a natural identification
\[
H^0(U^{\mathrm{univ},i},\LL\otimes_{\widehat{\ZZ}_p}\cO\BBcr)=\mathrm{Map}_{\mathrm{cont}}\left(G_U^i, V_{U}(\LL)\otimes_{\Zp}\cO\AAcr(R,R^+)\right)[1/t],
\]
where for $T,T'$ two topological spaces, $\mathrm{Map}_{\rm cont}(T,T')$ denotes the set of continuous maps from $T$ to $T'$. To see this, write $\widehat{U^{\mathrm{univ},i}}=\Spa(R_i,R_i^+)$. Then, by Corollary \ref{cor.acyclicityOBcris}, $H^0(U^{\mathrm{univ},i}, \LL\otimes_{\widehat{\ZZ}_p}\cO\BBcr)=V_{U}(\LL)\otimes_{\Zp}\cO\BBcr(R_i,R_i^+)$, which is also
\[
\left(\varprojlim_n V_{U}(\LL)\otimes_{\Zp}\cO\AAcr(R_i,R_i^+)/p^n\right)[1/t].
\]
Since $V_U(\LL)$ is of finite type over $\Zp$, it suffices to show that, for all $n\in \mathbb N$, $\cO\AAcr(R_i,R_i^+)/p^n$ can be identified with
\[
\mathrm{Map}_{\rm cont}(G_U^i, \cO\AAcr(R,R^+)/p^n)=\varinjlim_{N} \mathrm{Map}_{\rm cont}(G_U^i/N,\cO\AAcr(R,R^+)/p^n)
\]
where $N$ runs through the set of open normal subgroups of $G_U^i$. Since both $\cO\AAcr(R,R^+)$ and $\cO\AAcr(R_i,R_i^+)$ are flat over $\Zp$, one reduces to the case where $n=1$, and thus to showing $R_i^{\flat +}/(\underline p^p)=\mathrm{Map}_{\rm cont}(G_U^i, R^{\flat +}/(\underline p^p))$ by the explicit descriptions of $\cO\AAcr(R,R^+)/p$ and $\cO\AAcr(R_i,R_i^+)/p$. As $R_i^{\flat+}$ and $R^{\flat +}$ are flat over $\cO_{\Cp}^{\flat}$, we finally reduces to showing $R_i^+/p=\mathrm{Map}_{\rm cont}(G_U^i,R^+/p)$. But this last assertion is clear, giving our claim.

Consider the following spectral sequence associated to the cover  $U^{\rm univ}\to U$:
\[
E_{1}^{i,j}=H^j(U^{\mathrm{univ},i},\LL\otimes_{\widehat{\ZZ}_p}\cO\BBcr) \Longrightarrow H^{i+j}(U,\LL\otimes_{\widehat{\ZZ}_p}\cO\BBcr).
\]
As $E_{1}^{i,j}= 0$ for $j\geq 1$ (Corollary \ref{cor.acyclicityOBcris}), we have $
E_{2}^{n,0}=E_{\infty}^{n,0}\simeq H^n(U,\LL\otimes_{\widehat{\ZZ}_p}\cO\BBcr)$.
Thus, by the discussion in the paragraph above, we deduce a natural isomorphism
\[
H^j(U,\LL\otimes_{\widehat{\ZZ}_p}\cO\BBcr)\stackrel{\sim}{\longrightarrow} H^j_{\mathrm{cont}}(G_U,V_{U}(\LL)\otimes_{\Zp}\cO\AAcr(R,R^+))[1/t]
\]
where the right-hand side is the continuous group cohomology. Taking $j=0$, we obtain our first assertion. The isomorphism concerning $\Fil^r\cO\BBcr$ can be proved in the same way. The last assertion follows from the first isomorphism and  \cite[Proposition 8.3.1]{Bri}, which gives the assertion for the right-hand side.
\end{proof}

\begin{cor}\label{easyfact} Let $\LL$ be a lisse $\widehat{\Z}_p$-sheaf on $X_{\proet}$ which satisfies the condition (1) of Proposition \ref{prop.CrysSheaf}. Let $\cU=\Spf(A)$ be a small connected affine open subset of $\cX$. Write  $U=\cU_k$. Then for any $V\in X_{\proet}/U$, we have
\[\mathbb D_{\cris}(\LL)(\cU)\otimes_{A[1/p]}\cO\BBcr\left(V\right)\stackrel{\sim}{\longrightarrow} (w^{-1}\mathbb D_{\cris}(\LL)\otimes_{\cO_X^{\ur}} \cO\BBcr)(V).\]
\end{cor}

\begin{proof}By Lemma \ref{twodcris}, the $A[1/p]$-module $\mathbb D_{\cris}(\LL)(\cU)$ is projective of finite type, hence it is a direct summand of a finite free $A[1/p]$-module.  As $\mathbb D_{\cris}(\LL)$ is coherent over $\cO_{\cX}[1/p]$ and as $\cU$ is affine, $\mathbb D_{\cris}(\LL)|_{\cU}$ is then a direct summand of a finite free $\cO_{\cX}[1/p]|_{\cU}$-module. The isomorphism in our corollary then follows, since we have similar isomorphism when $\mathbb D_{\cris}(\LL)|_{\cU}$ is replaced by a free $\cO_{\cX}[1/p]|_{\cU}$-module.
\end{proof}

\begin{cor}\label{CrysRepAff} Let $\LL$ be a lisse $\widehat{\ZZ}_p$-sheaf verifying the condition (1) of Proposition \ref{prop.CrysSheaf}. Then the condition (2) of Proposition \ref{prop.CrysSheaf} holds for $\LL$ if and only if for any small affine connected open subset $\cU=\Spf(A)\subset \cX$ (with $U:=\cU_k$), the $G_U$-representation $V_{U}(\LL)\otimes_{\Zp}\Qp$ is crystalline in the sense that the following natural morphism is an isomorphism (\cite[Chapitre 8]{Bri})
\[
D_{\cris}(V_U(\LL))\otimes_{A[1/p]} \cO\BBcr(R,R^+) \stackrel{\sim}{\longrightarrow} V_U(\LL)\otimes_{\Zp}\cO\BBcr(R,R^+),
\]
where $G_U, U^{\rm univ}, \widehat{U^{\rm univ}}=\Spa(R,R^+)$ are as in the paragraph before Lemma \ref{twodcris}.
\end{cor}

\begin{proof} If $\LL$ satisfies in addition the condition (2) of Proposition \ref{prop.CrysSheaf}, combining it with Corollary \ref{easyfact}, we find
\begin{eqnarray*}
\mathbb D_{\cris}(\LL)(\cU)\otimes_{A[1/p]}\cO\BBcr\left(U^{\rm univ}\right) &\stackrel{\sim}{\longrightarrow} & (w^{-1}\mathbb D_{\cris}(\LL)\otimes_{\cO_X^{\ur}}\cO\BBcr)(U^{\rm univ}) \\ & \stackrel{\sim}{\longrightarrow} & (\LL\otimes_{\widehat{\ZZ}_p}\cO\BBcr)(U^{\rm univ})\\ & =  & V_{U}(\LL)\otimes_{\Zp}\cO\BBcr(U^{\rm univ}).
\end{eqnarray*}
So, by Corollary \ref{cor.acyclicityOBcris} and Lemma \ref{twodcris}, the $G_U$-representation $V_{U}(\LL)\otimes \mathbb Q_p$ is crystalline.

Conversely, assume that for any small connected affine open subset $\cU=\Spf(A)$ of $\cX$, the $G_{U}$-representation $V_{U}(\LL)\otimes_{\Zp}\Qp$ is crystalline. Together with  Lemmas \ref{2obcris} and \ref{twodcris}, we get $
\mathbb{D}_{\cris}(\LL)(\cU)\otimes_{A[1/p]}\cO\BBcr(U^{\rm univ})\stackrel{\sim}{\rightarrow} V_{U}(\LL)\otimes_{\Zp}\cO\BBcr(U^{\rm univ})$ and the similar isomorphism after replacing $U^{\rm univ}$ by any $V\in X_{\proet}/U^{\rm univ}$. Using Corollary \ref{easyfact}, we deduce $
(w^{-1}\mathbb D_{\cris}(\LL)\otimes_{\cO_{X}^{\rm ur}}\cO\BBcr)(V) \stackrel{\sim}{\rightarrow} (\LL\otimes_{\widehat{\Z}_p}\cO\BBcr)(V)$ for any $V\in X_{\proet}/U^{\rm univ}$, \emph{i.e.,} $(w^{-1}\mathbb D_{\cris}(\LL)\otimes_{\cO_{X}^{\rm ur}}\cO\BBcr)|_{U^{\rm univ}}\stackrel{\sim}{\to} (\LL\otimes_{\widehat{\Z}_p}\cO\BBcr)|_{U^{\rm univ}}$. When the small opens $\cU$'s run through a cover of $\cX$, the $U^{\rm univ}$'s form a cover of $X$ in $X_{\proet}$. Therefore, $w^{-1}\mathbb D_{\cris}(\LL)\otimes \cO\BBcr\stackrel{\sim}{\to} \LL\otimes \cO\BBcr$, as desired.
\end{proof}

\begin{lemma}\label{IsoCrysOnDcris} Let $\LL$ be a lisse $\widehat{\Z}_p$-sheaf on $X$ satisfying the two conditions of Proposition \ref{prop.CrysSheaf}. Then (the analytification of) $\mathbb D_{\cris}(\LL)$ has a natural structure of filtered convergent $F$-isocrystal on $\cX_0/\cO_k$.
\end{lemma}

\begin{proof} First of all, the $\Fil^i\mathbb D_{\cris}(\LL)$'s ($i\in \mathbb Z$) endow a separated exhaustive decreasing filtration on $\mathbb D_{\cris}(\LL)$ by Corollary \ref{GradedOfBcris}, and the connection on $\mathbb D_{\cris}(\LL)=w_{\ast}(\LL\otimes \cO\BBcr)$ can be given by the composite of
\begin{eqnarray*}
w_{\ast}(\LL\otimes \cO\BBcr)&\stackrel{w_{\ast}(\mathrm{id}\otimes \nabla)}{\longrightarrow} & w_{\ast}(\LL\otimes \cO\BBcr\otimes_{\cO_{X}^{\ur}}\Omega_{X/k}^{1,\ur}) \\ & \stackrel{\sim}{\longrightarrow} &  w_{\ast}(\LL\otimes \cO\BBcr)\otimes_{\cO_{\cX}[1/p]}\Omega^1_{\cX/\cO_k}[1/p]
\end{eqnarray*}
where the last isomorphism is the projection formula. That the connection satisfies the Griffiths transversality with respect to the filtration $\Fil^{\bullet}\mathbb D_{\cris}$  follows from the analogous assertion for $\cO\BBcr$ (Proposition \ref{poincare}).

Now consider the special case where $\cX=\Spf(A)$ is affine connected admitting an \'etale map $\cX\to \Spf(\cO_k\{T_1^{\pm 1}, \ldots, T_d^{\pm 1}\})$, such that $\cX$ is equipped with a lifting of Frobenius $\sigma$.  As in the paragraph before Lemma \ref{twodcris}, let $X^{\rm univ}$ be the univeral profinite \'etale cover of $X$ (which is an affinoid perfectoid). Write $\widehat{X^{\rm univ}}=\Spa(R,R^+)$ and $G_X$ the fundamental group of $X$. As $\cX$ is affine, the category $\mathrm{Coh}(\cO_{\cX}[1/p])$ is equivalent to the category of finite type $A[1/p]$-modules. Under this equivalence, $\mathbb D_{\cris}(\LL)$ corresponds to $D_{\cris}(V_{X}(\LL)):=(V_X(\LL)\otimes \cO\BBcr(R,R^+))^{G_X}$, denoted by $D$ for simplicity. So $D$ is a projective $A[1/p]$-module of finite type (Lemma \ref{twodcris}) equipped with a connection $\nabla \colon D\to D\otimes \Omega^{1}_{A[1/p]/k}$. Under the same equivalence, $\Fil^i\mathbb D_{\cris}(\LL)$ corresponds to $
\Fil^i D:=(V_{X}(\LL)\otimes \Fil^i \cO\BBcr(R,R^+))^{G_X}$, by Lemma \ref{twodcris} again.
By the same proof as in  \cite[8.3.2]{Bri}, the graded quotient $\mathrm{gr}^i(D)$ is a projective module. In particular, $\Fil^i D\subset D$ is a direct summand. Therefore, each $\Fil^i\mathbb D_{\cris}(\LL)$ is a direct summand of $\mathbb D_{\cris}(\LL)$. Furthermore, since $\cX$ admits a lifting of Frobenius $\sigma$, we get from \S~\ref{localfrob} a $\sigma$-semilinear endomorphism $\varphi$ on $\cO\BBcr(X^{\rm univ})\simeq \cO\BBcr(R,R^+)$, whence a $\sigma$-semilinear endomorphism on $D$, still denoted by $\varphi$. Via Lemma  \ref{FrobHorizontal}, one checks that the Frobenius $\varphi$ on $D$ is horizontal with respect to its connection. Thus $\mathbb D_{\cris}(\LL)$ is endowed with a horizontal $\sigma$-semilinear morphism $\mathbb D_{\cris}(\LL)\to \mathbb D_{\cris}(\LL)$, always denoted by $\varphi$ in the following.

To finish the proof in the special case, one needs to show that $(\mathbb D_{\cris}(\LL),\nabla,\varphi)$ gives an $F$-isocrystal on $\cX_0/\cO_k$.  As $D$ is of finite type over $A[1/p]$, there is some $n\in \mathbb N$ such that $D=D^+[1/p]$ with $D^+:=(V_{X}(\LL)\otimes_{\Zp} t^{-n}\cO\AAcr(R,R^+))^{G_X}$.
The connection on $t^{-n}\cO\AAcr(R,R^+)$ induces a connection $\nabla^+ \colon D^+\to D^{+}\otimes_{A}\Omega^{1}_{A/\cO_k}$ on $D^+$, compatible with that of $D_{\cris}(V_{X}(\LL))$. Moreover, let $N_i$ be  the endomorphism of $D^+$ so that $\nabla^+=\sum_{i=1}^{d} N_i\otimes dT_i$. Then for any $a\in D^+$, $\underline N^{\underline m}(a)\in p\cdot D^+$ for all but finitely many $\underline m\in \mathbb N^d$ (as this holds for the connection on $t^{-n}\cO\AAcr$, seen in the proof of Lemma \ref{sigma12}). Similarly, the Frobenius on $\cO\BBcr(R,R^+)$ induces a map (note that the Frobenius on $\cO\BBcr(R,R^+)$ sends $t$ to $p\cdot t$)
\[
\varphi\colon D^+\longrightarrow (V_{X}(\LL)\otimes p^{-n} t^{-n}\cO\AAcr(R,R^+))^{G_X}.
\]
Thus $\psi:=p^n\varphi$ gives a well-defined $\sigma$-semilinear morphism on $D^+$. One checks that $\psi$ is horizontal with respect to the connection $\nabla^+$ on $D^+$ and it induces an $R^+$-linear isomorphism $\sigma^{\ast} D^+\stackrel{\sim}{\to} D^+$.
As a result, the triple $(D^+,\nabla, \psi)$ will define an $F$-crystal on $\cU_0/\cO_k$, once we know $D^+$ is of finite type over $A$. The required finiteness of $D^+$ is explained in \cite[Proposition 3.6]{AI}, and for the sake of completeness we recall briefly their proof here. As $D$ is projective of finite type (Lemma \ref{twodcris}), it is a direct summand of a finite free $A[1/p]$-module $T$. Let $T^+\subset T$ be a finite free $A$-submodule of $T$ such that $T^{+}[1/p]=T$. Then we have the inclusion $D\otimes_{A[1/p]}\cO\BBcr(R,R^+)\hookrightarrow T^+\otimes_{A} \cO\BBcr(R,R^+).$
As $V_{X}(\LL)$ is of finite type over $\Zp$ and $\cO\BBcr(R,R^+)=\cO\AAcr(R,R^+)[1/t]$, there exists $m\in \mathbb N$ such that the $\cO\AAcr(R,R^+)$-submodule $V_{X}(\LL)\otimes t^{-n}\cO\AAcr(R,R^+)$ of $V_X(\LL)\otimes \cO\BBcr(R,R^+)\simeq D\otimes \cO\BBcr(R,R^+)$ is contained in $T^{+}\otimes_{A}t^{-m}\cO\AAcr(R,R^+)$. By taking $G_{U}$-invariants and using the fact that $A$ is noetherian, we are reduced to showing that $A':=(t^{-m}\cO\AAcr(R,R^+))^{G_{X}}$ is of finite type over $A$. From the construction,  $A'$ is $p$-adically separated and $A\subset A' \subset A[1/p]=(\cO\BBcr(R,R^+))^{G_X}$.  As $A$ is normal, we deduce $p^NA'\subset A$ for some $N\in \mathbb N$. Thus $p^NA'$ and  hence $A'$ are of finite type over $A$. As a result, $(D^+,\nabla, \psi)$ defines an $F$-crystal $\mathcal D^+$ on $\cU_0/\cO_k$. As $D=D^+[1/p]$ and $\nabla=\nabla^+[1/p]$, the connection $\nabla$ on $\mathbb D_{\cris}(\LL)$ is convergent; this is standard and we refer to \cite[2.4.1]{Ber} for detail. Consequently, the triple $(\mathbb D_{\cris}(\LL), \nabla, \varphi )$ is an $F$-isocrystal on $\cX_0/\cO_k$, which is isomorphic to  $\mathcal D^{+,{\rm an}}(n)$. This finishes the proof in the special case.

In the general case, consider a covering $\cX=\bigcup_i \cU_i$ of $\cX$ by connected small affine open subsets such that each $\cU_i$ admits a lifting of Frobenius $\sigma_i$ and an \'etale morphism to some torus over $\cO_k$. By the special case, each $\Fil^i\mathbb D_{\cris}(\LL)\subset \mathbb D_{\cris}(\LL)$ is locally a direct summand, and the connection on $\mathbb D_{\cris}(\LL)$ is convergent (\cite[2.2.8]{Ber}). Furthermore, each $\mathbb D_{\cris}(\LL)|_{\cU_i}$ is equipped with a Frobenius $\varphi_i$, and over $\cU_i\cap \cU_j$, the two Frobenii $\varphi_i,\varphi_j$ on $\mathbb D_{\cris}(\LL)|_{\cU_{i}\bigcap \cU_{j}}$ are related by the formula in Lemma \ref{fsigma12} as it is the case for $\varphi_i,\varphi_j$ on $\cO\BBcr|_{U_i\bigcap U_j}$ (Lemma \ref{sigma12}). So these local Frobenii glue together to give a compatible system of Frobenii $\varphi$ on $\mathbb D_{\cris}(\LL)$ and the analytification of the quadruple $(\mathbb D_{\cris}(\LL),\Fil^{\bullet}\mathbb D_{\cris}(\LL), \nabla, \varphi)$ is a filtered $F$-isocrystal on $\cX_0/\cO_k$, as wanted.
\end{proof}

\begin{proof}[Proof of Proposition~\ref{prop.CrysSheaf}]  If a lisse $\widehat{\ZZ}_p$-sheaf $\LL$ on $X$ is associated to a filtered $F$-isocrystal $\cE$ on $X$, then we just have to show $\cE\simeq \DD_{\cris}(\LL)$. By assumption, we have $\mathbb L\otimes_{\widehat{\ZZ}_p}\cO\BBcr\simeq w^{-1}\cE\otimes_{\cO_X^{\ur}}\cO\BBcr$. Then
\[
w_*(\mathbb L\otimes_{\widehat{\ZZ}_p}\cO\BBcr)\simeq w_*(w^{-1}\cE\otimes_{\cO_{X}^{\ur}}\cO\BBcr)\simeq \cE\otimes_{\cO_{\cX_{\et}}[1/p]}w_*\cO\BBcr\simeq\cE
\] where the second isomorphism has used Remark \ref{abuseem}, and the last isomorphism is by the isomorphism  $w_*\cO\BBcr\simeq \cO_{\cX_{\et}}[1/p]$ from Corollary \ref{higherox}.

Conversely, let $\LL$ be a lisse $\widehat{\ZZ}_p$-sheaf verifying the two conditions of our proposition. By Lemma \ref{IsoCrysOnDcris}, $\mathbb D_{\cris}(\LL)$ is naturally a filtered $F$-isocrystal. To finish the proof, we need to show that the isomorphism in (2) is compatible with the extra structures. Only the compatibility with filtrations needs  verification. This is a local question, hence we shall assume $\cX=\Spf(A)$ is a small connected affine formal scheme. As $\Fil^i\mathbb D_{\cris}(\LL)$ is coherent over $\cO_{\cX}[1/p]$ and is a direct summand of $\mathbb D_{\cris}(\LL)$, the same proof as that of Corollary \ref{easyfact} gives a natural isomorphism
\[
\Fil^i\mathbb D_{\cris}(\LL)(\cX)\otimes_{A[1/p]} \Fil^j\cO\BBcr(V)\stackrel{\sim}{\longrightarrow} (w^{-1}\Fil^i \mathbb D_{\cris}(\LL)\otimes_{\cO_X^{\ur}} \Fil^j\cO\BBcr)(V)
\]
for any $V\in X_{\proet}$. Consequently, the isomorphism in Corollary \ref{easyfact} is strictly compatible with filtrations. Thus, we reduce to showing that, for an affinoid perfectoid $V\in X_{\proet}/X^{\rm univ} $, the isomorphism $D_{\cris}(V_X(\LL))\otimes_{A[1/p]} \cO\BBcr(V)\stackrel{\sim}{\to} V_{X}(\LL)\otimes \cO\BBcr(V)$ is strictly compatible with the filtrations in the sense that its inverse respects also the filtrations on both sides, or equivalently, the induced morphisms between the gradeds quotients are isomorphisms:
\begin{equation}\label{eq.gradediso}
\bigoplus_{i+j=n}\left(\mathrm{gr}^iD_{\cris}(V_X(\LL))\otimes_{A[1/p]} \mathrm{gr}^j\cO\BBcr(V)\right)\longrightarrow V_X(\LL)\otimes \mathrm{gr}^n\cO\BBcr(V).
\end{equation}
When $V=X^{\rm univ}$, this follows from \cite[8.4.3]{Bri}. For the general case, write $\widehat{X^{\rm univ}}=\Spa(R,R^+)$ and $\widehat{V}=\Spa(R_1,R_1^+)$. By \cite[Corollary 6.15]{Sch} and Corollary \ref{GradedOfBcris},
\[
\mathrm{gr}^j\cO\BBcr\simeq \xi^j\hat{\cO}_{X}[U_1/\xi,\ldots, U_d/\xi]\subset \mathrm{gr}^{\bullet }\cO\BBcr\simeq \hat{\cO}_X[\xi^{\pm 1},U_1,\ldots, U_d],
\]
where $\xi$ and all $U_i$ have degree $1$. So $\mathrm{gr}^j\cO\BBcr(X^{\rm univ})\simeq \xi^j R[U_1/\xi,\ldots, U_d/\xi]$ and  $\mathrm{gr}^j\cO\BBcr(V)\simeq \xi^j R_1[U_1/\xi,\ldots, U_d/\xi]$. As a result, the natural morphism $\mathrm{gr}^j\cO\BBcr(X^{\rm univ})\otimes_RR_1\stackrel{\sim}{\to}\mathrm{gr}^j\cO\BBcr(V) $ is an isomorphism. The required isomorphism \eqref{eq.gradediso} for general $V$ then follows from the special case for $X^{\rm univ}$.
\end{proof}

Let $\mathrm{Lis}^{\cris}_{\widehat{\ZZ}_p}(X)$ denote the category of lisse crystalline $\widehat{\ZZ}_p$-sheaves on $X$, and $\mathrm{Lis}^{\cris}_{\widehat{\QQ}_p}(X)$ the corresponding isogeny category. The functor
\[
\mathbb D_{\cris}\colon \mathrm{Lis}^{\cris}_{\widehat{\QQ}_p}(X)\longrightarrow FF\textrm{-Iso}(\cX_0/\cO_k), \quad  \LL\mapsto \mathbb D_{\cris}(\LL)
\]
allows us to relate $\mathrm{Lis}^{\cris}_{\widehat{\QQ}_p}(X)$ to the category $FF\textrm{-}\mathrm{Iso}(\cX_0/\cO_k)$  of filtered convergent $F$-isocrystals on $\cX_0/\cO_k$, thanks to Proposition \ref{prop.CrysSheaf}.
A filtered $F$-isocrystal $\cE$ on $\cX_0/\cO_k$ is called \emph{admissible} if it lies in the essential image of the functor above. The full subcategory of admissible filtered $F$-isocrystals on $\cX_0/\cO_k$ will be denoted by $FF\textrm{-Iso}(\cX_0/\cO_k)^{\textrm{adm}}$.

\begin{thm}
The functor $\DD_{\cris}$ above induces an equivalence of categories
\[
\mathbb D_{\cris}\colon \mathrm{Lis}^{\cris}_{\widehat{\QQ}_p}(X)\stackrel{\sim}{\longrightarrow} FF\textrm{-}\mathrm{Iso}(\cX_0/\cO_k)^{\mathrm{adm}}.
\]
A quasi-inverse of $\mathbb D_{\cris}$ is given by
\[
\VV_{\cris}: \cE\mapsto \Fil^0(w^{-1}\cE\otimes_{\cO_{X}^{\ur}}\cO\BBcr)^{\nabla=0, \varphi=1}
\]
where $\varphi$ denotes the compatible system of Frobenii on $\cE$ as before.
\end{thm}

\begin{proof} Observe first that, for $\cE$ a filtered convergent $F$-isocrystal, the local Frobenii on $\cE^{\nabla=0}$ glue to give a unique $\sigma$-semilinear morphism on $\mathcal E^{\nabla=0}$ (Lemma \ref{fsigma12}). In particular, the abelian sheaf $\mathbb V_{\cris}(\cE)$ is well-defined. Assume moreover $\cE$ is admissible, and let $\LL$ be a lisse $\widehat{\Z}_p$-sheaf such that $\cE\simeq \mathbb D_{\cris}(\LL)$. So $\LL$ and $\cE$ are associated by Proposition \ref{prop.CrysSheaf}. Hence $
\LL\otimes_{\widehat{\ZZ}_p}\cO\BBcr\simeq w^{-1}\cE\otimes_{\cO_X^{\ur}}\cO\BBcr$, and we find
\begin{eqnarray*}
\LL\otimes_{\widehat{\ZZ}_p}\widehat{\QQ}_p & \stackrel{\sim}{\longrightarrow} &  \LL\otimes_{\widehat{\ZZ}_p} \Fil^0(\cO\BBcr)^{\nabla=0, \varphi=1} \\ & \stackrel{\sim}{\longrightarrow} & \Fil^0(\LL\otimes_{\widehat{\ZZ}_p}\cO\BBcr)^{\nabla=0,\varphi=1} \\ & \stackrel{\sim}{\lra} & \Fil^0(w^{-1}\cE\otimes_{\cO_{X}^{\ur}}\cO\BBcr)^{\nabla=0,\varphi=1}
\\ & =& \mathbb V_{\cris}(\cE),
\end{eqnarray*}
where the first isomorphism following from the the fundamental exact sequence  (by Lemma \ref{vanish} and \cite[Corollary 6.2.19]{Bri})
\[
0\lra \Qp\lra \Fil^0\BBcr\stackrel{1-\varphi}{\lra}\BBcr\lra 0.
\]
In particular, $\mathbb V_{\cris}(\cE)$ is the associated $\widehat{\mathbb Q}_p$-sheaf of a lisse $\widehat{\Z}_p$-sheaf. Thus $\mathbb V_{\cris}(\cE)\in \mathrm{Lis}^{\rm cris}_{\widehat{\mathbb Q}_p}(X)$ and the functor $\mathbb V_{\cris}$ is well-defined. Furthermore, as we can recover the the lisse $\widehat{\Z}_p$-sheaf up to isogeny, it follows that $\mathbb D_{\cris}$ is fully faithful, and a quasi-inverse on its essential image is given by $\mathbb V_{\cris}$.
\end{proof}

\begin{rk} Using \cite[Theorem 8.5.2]{Bri}, one can show that the equivalence above is an equivalence of tannakian category.
\end{rk}

Next we compare Definition \ref{associated} with the ``associatedness" defined in \cite{Fal}. Let $\cE$ be a filtered convergent $F$-isocrystal $\cE$ on $\cX_0/\cO_k$, and $\cM$ an $F$-crystal on $\cX_0/\cO_k$ such that $\cM^{\an}=\cE(-n)$ for some $n\in \N$ (see Remark \ref{rk.CrystalVSIsoc} for the notations). Let $\cU=\Spf(A)$ be a small connected affine open subset of $\cX$, equipped with a lifting of Frobenius $\sigma$. Write $U=\Spa(A[1/p],A)$ the generic fiber of $\cU$. As before, let $\overline A$ be the union of all finite normal $A$-algebras (contained in some fixed algebraic closure of $\mathrm{Frac}(A)$) which are \'etale over $A[1/p]$. Let $G_U:=\mathrm{Gal}(\overline A[1/p]/A[1/p])$ and $(R,R^+)$ the $p$-adic completion of $(\overline A[1/p], \overline A)$. Then $(R,R^+)$ is an perfectoid affinoid algebra over $\mathbb C_p=\hat{\bk}$. So we can consider the period ring $\mathbb A_{\cris}(R,R^+)$. Moreover the composite of the following two natural morphisms
\begin{equation}\label{PDThickening}
\AAcr \left(R,R^+\right) \stackrel{\theta}{\longrightarrow} R^+ \stackrel{\textrm{can}}{\longrightarrow} R^+/pR^+,
\end{equation}
defines a $p$-adic PD-thickening of $\Spec(R^+/pR^+)$. Evaluate our $F$-crystal $\cM$ at it and write $\cM(\mathbb A_{\cris}(R,R^+))$ for the resulting $\mathbb A_{\cris}(R,R^+)$-module. As an element of $G_U$ defines a morphism of the PD-thickening \eqref{PDThickening} in the big crystalline site of $\cX_0/\cO_k$ and $\cM$ is a crystal, $\cM(\mathbb A(R,R^+))$ is endowed naturally with an action of $G_U$. Similarly, the Frobenius on the crystal $\cM$ gives a Frobenius $\psi$ on $\cM(\AAcr(R,R^+))$.
Set $\cE(\mathbb B_{\cris}(R,R^+)):=\cM(\mathbb A_{\cris}(R,R^+))[1/t]$, which is a $\mathbb B_{\cris}(R,R^+)$-module of finite type endowed with a Frobenius $\varphi=\psi/p^n$ and an action of $G_U$.

On the other hand, as $\cU$ is small, there exists a morphism $\alpha\colon A\to \mathbb A_{\cris}(R,R^+)$ of $\cO_k$-algebras, whose composite with $\theta:\AAcr(R,R^+)\to R^+$ is the inclusion $A\subset R^+$. For example, consider an \'etale morphism $\cU\to \Spf(\cO_k\{T_1^{\pm 1}, \ldots, T_d^{\pm 1}\})$. Let $(T_i^{1/p^{n}})$ be a compatible system of $p^n$-th roots of $T_i$ inside $\overline A\subset R^+$, and $T_i^{\flat}$ the corresponding element of $R^{\flat +}:=\varprojlim_{x\mapsto x^p} R^+/pR^+$. Then one can take $\alpha$ as the unique morphism of $\cO_k$-algebras $A\to \mathbb A_{\cris}(R,R^+)$ sending $T_i$ to $[T_i^{\flat}]$, such that its composite with the projection $\AAcr(R,R^+)\to R^+/pR^+$ is just the natural map $A\to R^+/pR^+$ (such a morphism exists as $A$ is \'etale over $\cO_k\{T_1^{\pm 1},\ldots, T_d^{\pm 1}\}$ and because of \eqref{PDThickening}; see the proof of Lemma \ref{algebra} for a similar situation). Now we fix  such a morphism $\alpha$. So we obtain a morphism of PD-thickenings from  $\cU_0\hookrightarrow \cU$ to the one defined by \eqref{PDThickening}. Consequently we get a natural isomorphism
$
\mathcal M(\mathbb A_{\cris}(R,R^+))\simeq \mathcal M(\cU)\otimes_{A,\alpha}\mathbb A_{\cris}(R,R^+)$, whence
\[
\mathcal E(\mathbb B_{\cris}(R,R^+))\simeq \mathcal E(\cU)\otimes_{A[1/p],\alpha}\mathbb B_{\cris}(R,R^+),
\]
here $\cE(\cU):=\cM(\cU)[1/p]$.
Using this isomorphism, we define the filtration on $\cE(\BBcr(R,R^+))$ as the tensor product of the filtration on $\cE(\cU)$ and that on $\BBcr(R,R^+)$.

\begin{rk}\label{rk.compfiltrations} It is well-known that the filtration on $\cE(\BBcr(R,R^+))$ does not depend on the choice of $\alpha$. More precisely, let $\alpha'$ be a second morphism $A\to \AAcr(R,R^+)$ of $\cO_k$-algebras whose composite with $\AAcr(R,R^+)\to R^+$ is the inclusion $A\subset R^+$. Fix an \'etale morphism $\cU\to \Spf(\cO_k\{T_1^{\pm 1},\ldots, T_d^{\pm 1}\})$. Denote $\beta=(\alpha, \alpha'):A\otimes_{\cO_k}A\to \AAcr(R,R^+)$ and by the same notation the corresponding map on schemes, and write $p_1, p_2: \Spec A\times \Spec A\ra \Spec A$ the two projections. We have a canonical isomorphism $(p_2\circ \beta)^*\cE\simto (p_1\circ \beta)^*\cE$, as $\cE$ is  a crystal.  In terms of  the connection $\nabla$ on $\cE$, this gives (cf. \cite[2.2.4]{Ber}) the following $\BBcr(R,R^+)$-linear isomorphism
\[
\eta\colon \cE(\cU)\otimes_{A[1/p],\alpha}\mathbb B_{\cris}(R,R^+) \longrightarrow \cE(\cU)\otimes_{A[1/p],\alpha' }\BBcr(R,R^+)
\]
sending $e\otimes 1$ to $\sum_{\underline n\in \mathbb N^d} \underline{N}^{\underline n}(e)\otimes (\alpha(\underline T)-\alpha'(\underline T))^{[\underline n]}$, with $\underline N$ the endomorphism of $\cE$ such that $\nabla=\underline N\otimes d\underline T$. Here we use the multi-index to simplify the notations, and note that $\alpha(T_i)-\alpha'(T_i)\in \Fil^1 \AAcr(R,R^+)$ hence the divided power $(\alpha(T_i)-\alpha'(T_i))^{[n_i]}$ is well-defined. Moreover, the series converge since the connection on $\cM$ is quasi-nilpotent. Now as the filtration on $\cE$ satisfies Griffiths transversality, the isomorphism $\eta$ is compatible with the tensor product filtrations on both sides. Since the inverse $\eta^{-1}$ can be described by a similar formula (one just switches $\alpha$ and $\alpha'$), it is also compatible with filtrations on both sides. Hence the isomorphism $\eta$ is strictly compatible with the filtrations, and the filtration on $\cE(\BBcr(R,R^+))$ does not depend on the choice of $\alpha$.
\end{rk}

Let $\LL$ be a lisse $\widehat{\Z}_p$-sheaf on $X$, and write as before $V_U(\LL)$ the $\Zp$-representation of $G_U$ corresponding to the lisse sheaf $\LL|_U$. Following \cite{Fal}, we say  a filtered convergent $F$-isocrystal \emph{$\cE$ on $\cX_0/\cO_k$ is associated to $\LL$ in the sense of Faltings} if, for all small open subset $\cU\subset \cX$, there is a functorial filtered isomorphism:
\begin{equation}\label{eq.Faltings}
\cE\left(\BBcr\left(R,R^+\right)\right) \stackrel{\sim}{\longrightarrow}  V_{U}(\LL)\otimes_{\Qp}\BBcr\left(R,R^+\right),
\end{equation}
which is compatible with $G_U$-action and Frobenius.

\begin{prop}\label{prop.faltings} If $\cE$ is associated to $\LL$ in the sense of Faltings then $\LL$ is crystalline (not necessarily associated to $\cE$) and there is an isomorphism $\mathbb D_{\cris}(\LL)\simeq \cE$ compatible with filtration and Frobenius.
Conversely,  if  $\LL$ is crystalline and if there is an isomorphism $\mathbb D_{\cris}(\LL)\simeq \cE$ of $\cO_{X^{\an}}$-modules compatible with filtration and Frobenius, then $\LL$ and $\cE$ are associated  in the sense of Faltings.
\end{prop}

Before giving the proof of Proposition \ref{prop.faltings}, we observe first the following commutative diagram in which  the left vertical morphisms are all PD-morphisms:
\[
\xymatrix{A\ar[r]\ar[d]_{\mathrm{can}} & A/pA \ar[d] \\ \cO\AAcr\left(R,R^+\right) \ar[r]^{\theta_A} & R^+/pR^+ \ar@{=}[d]\\ \AAcr\left(R,R^+\right)\ar[r]^{\theta} \ar[u]^{\mathrm{can}}& R^+/pR^+.}
\]
Therefore, we have isomorphisms
\begin{eqnarray*}
\cM(\cU)\otimes_{A}\cO\AAcr\left(R,R^+\right) & \stackrel{\sim}{\longrightarrow} & \cM\left(\cO\AAcr\left(R,R^+\right)\right) \\ & \stackrel{\sim}{\longleftarrow} &  \cM\left(\AAcr\left(R,R^+\right)\right)\otimes_{\AAcr\left(R,R^+\right)}\cO\AAcr\left(R,R^+\right),
\end{eqnarray*}
where the second term in the first row denotes the evaluation of the crystal $\cM$ at the PD-thickening defined by the PD-morphism $\theta_A$ in the commutative diagram above. Inverting $t$, we obtain a natural isomorphism
\begin{equation}\label{identification}
\cE(\cU)\otimes_{A[1/p]}\cO\BBcr\left(R,R^+\right)  \stackrel{\sim}{\longrightarrow}   \cE\left(\BBcr\left(R,R^+\right)\right)\otimes\cO\BBcr\left(R,R^+\right),
\end{equation}
where the last tensor product is taken over $\BBcr(R,R^+) $. This isomorphism is clearly compatible with Galois action and Frobenius. By a similar argument as in Remark \ref{rk.compfiltrations} one checks that \eqref{identification} is also strictly compatible with the filtrations. Furthermore, using the identification
\[
\AAcr\left(R,R^+\right)\{\langle u_1,\ldots, u_d\rangle \}\stackrel{\sim}{\longrightarrow} \cO\AAcr\left(R,R^+\right),\quad u_i\mapsto T_i\otimes 1-1\otimes [T_i^{\flat}]
\]
we obtain a section $s$ of the canonical map $\AAcr\left(R,R^+\right) \to \cO\AAcr\left(R,R^+\right) $:
\[
s\colon \cO\AAcr\left(R,R^+\right) \to \AAcr\left(R,R^+\right),\quad u_i\mapsto 0,
\]
which is again a PD-morphism. Composing it with the inclusion $A\subset \cO\AAcr(R,R^+)$, we get a morphism $\alpha_0\colon R^+\to \AAcr(R,R^+)$ whose composite with the projection $\AAcr(R,R^+)\to R^+$ is the inclusion $A\subset R^+$.

\begin{proof}[Proof of Proposition \ref{prop.faltings}]Now assume that $\cE$ is associated with $\LL$ in the sense of Faltings. Extending scalars to $\cO\BBcr(R,R^+) $ of the isomorphism \eqref{eq.Faltings} and using the identification \eqref{identification}, we obtain a functorial isomorphism, compatible with filtration, $G_U$-action, and Frobenius:
\[
V_{U}(\LL)\otimes_{\Zp}\cO\BBcr\left(R,R^+\right) \stackrel{\sim}{\longrightarrow}\cE(\cU)\otimes_{A[1/p]} \cO\BBcr\left(R,R^+\right).
\]
Therefore, $V_{U}(\LL)\otimes_{\Zp}\Qp$ is a crystalline $G_U$-representation (Corollary \ref{CrysRepAff}), and we get by Lemma  \ref{twodcris} an isomorphism $\cE(\cU)\simto\mathbb D_{\cris}(\LL)(\cU)$ compatible with filtrations and Frobenius. As such small open subsets $\cU$ form a basis for the Zariski topology of $\cX$, we find an isomorphism $\cE\simto \mathbb D_{\cris}(\LL)$ compatible with filtrations and Frobenius, and that $\LL$ is crystalline in the sense of Definition \ref{associated} (Corollary \ref{CrysRepAff}).

Conversely, assume $\LL$ is crystalline with $\mathbb D_{\cris}(\LL)\simeq \cE$ compatible with filtrations and Frobenius. As in the proof of Corollary \ref{CrysRepAff}, we have a functorial isomorphism
\[
\cE(\cU)\otimes_{A[1/p]}\cO\BBcr \left(R,R^+\right) \stackrel{\sim}{\longrightarrow}V_{U}(\LL)\otimes_{\Zp}\cO\BBcr \left(R,R^+\right)
\]
which is compatible with filtration, Galois action and Frobenius. Pulling it back via the section $\cO\BBcr\left(R,R^+\right) \to \BBcr\left(R,R^+\right)$ obtained from $s$ by inverting $p$, we obtain a functorial isomorphism
\[
\cE(\mathbb B_{\cris}(R,R^+))\simeq \cE(\cU)\otimes_{A[1/p],\alpha_0} \BBcr(R,R^+)\stackrel{\sim}{\longrightarrow} V_U(\LL)\otimes_{\Zp}\BBcr(R,R^+),
\]
which is again compatible with Galois action, Frobenius and filtrations. Therefore $\LL$ and $\cE$ are associated in the sense of Faltings.
\end{proof}

Finally we compare Definition \ref{associated} with its de Rham analogue considered in \cite{Sch}.

\begin{prop}\label{prop.crisdr} Let $\LL$ be a lisse $\widehat{\Z}_p$-sheaf on $X$ and $\cE$ a filtered convergent $F$-isocrystal on $\cX_0/\cO_k$. Assume that $\LL$ and $\cE$ are associated as defined in Definition \ref{associated}, then $\LL$ is de Rham in the sense of \cite[Definition 8.3]{Sch} . More precisely, if we view $\cE$ as a filtered module with integrable connection on $X$ (namely we forget the Frobenius), there exists a natural filtered isomorphism compatible with connections:
\[
\LL\otimes_{\widehat{\Z}_p}\cO\BB_{\rm dR}\lra \cE\otimes_{\cO_X} \cO\BB_{\rm dR}.
\]
\end{prop}

\begin{proof} Let $\cU=\Spf(R^+)\subset \cX$ be a connected affine open subset, and denote $U$ (resp. $U^{\rm univ}$) the generic fiber of $\cU$ (resp. the universal \'etale cover of $U$). Let $V$ be a affinoid perfectoid lying above $U^{\rm univ}$. As $\LL$ and $\cE$ are associated, there exits a filtered isomorphism compatible with connections and Frobenius
\[
\LL\otimes_{\widehat{\Z}_p}\cO\BBcr\stackrel{\sim}{\lra} w^{-1}\cE\otimes_{\cO_{X}^{\rm ur}} \cO\BBcr.
\]
Evaluate this map at $V\in X_{\proet}$ and use the fact that the $A[1/p]$-module $\cE(\cU)$ is projective, we deduce a filtered isomorphism compatible with all extra structures:
\[
V_U(\LL)\otimes_{\Zp}\cO\BBcr(V)\stackrel{\sim}{\lra} \cE(\cU)\otimes_{A[1/p]} \cO\BBcr(V).
\]
Taking tensor product $-\otimes_{\cO\BBcr(V)}\cO\BB_{\rm dR}(V)$ on both sides, we get a filtered isomorphism compatible with connection:
\[
V_U(\LL)\otimes_{\Zp}\cO\BB_{\rm dR}(V)\stackrel{\sim}{\lra} \cE(\cU)\otimes_{A[1/p]} \cO\BB_{\rm dR}(V).
\]
Again, as $\cE(\cU)$ is a projective $A[1/p]$-module and as $\cE$ is coherent, the isomorphism above can be rewritten as
\[
(\LL\otimes_{\widehat{\Z}_p}\cO\BB_{\rm dR})(V)\stackrel{\sim}{\lra} (\cE\otimes_{\cO_X}\cO\BB_{\rm dR})(V),
\]
which is clearly functorial in $\cU$ and in $V$. Varying $\cU$ and $V$, we deduce that $\LL$ is de Rham, giving  our proposition.
\end{proof}

\subsection{From pro-\'etale site to \'etale site} Let $\cX$ be a smooth formal scheme over $\cO_k$.
For $\cO=\cO_{\cX}, \cO_{\cX}[1/p], \cO_X^{\ur+},\cO_X^{\ur}$ and a sheaf of $\cO$-modules $\cF$ with connection, we denote the de Rham complex of $\cF$ as:
\[
DR(\cF)=(0\ra \cF
\stackrel{\nabla}{\lra}\cF\otimes_{\cO} \Omega^1\stackrel{\nabla}{\lra}\cdots).
\]
Let $\overline{w}$ be the composite of natural morphisms of topoi (here we use the same notation to denote the object in $X_{\proet}^{\sim}$ represented by $X_{\bk}\in X_{\proet}$):
\[
X_{\proet}^{\sim}\slash X_{\bk} \lra X_{\proet}^{\sim}\stackrel{w}{\lra} \mathcal X_{\et}^{\sim}.
\]
The following lemma is a global reformulation of the main results of \cite{AB}. As we shall prove a more general result later (Lemma \ref{quasirelative}), let us omit the proof here.

\begin{lemma}\label{quasi}
Let $\cX$ be smooth formal scheme over $\cO_k$. Then the natural morphism below is an isomorphism in the filtered derived category:
\[
\cO_{\cX}\widehat{\otimes}_{\cO_k}B_{\cris}\lra R\overline{w}_{*}(\cO\BBcr).
\]
Here $\cO_{\cX}\widehat{\otimes}_{\cO_k}B_{\cris}:=\left(\cO_{\cX}\widehat{\otimes}_{\cO_k}A_{\cris}\right)[1/t]$ with
\[
\cO_{\cX}\widehat{\otimes}_{\cO_k}A_{\cris}:=\varprojlim_{n\in \mathbb N} \cO_{\cX}\otimes_{\cO_k}A_{\cris}/p^n,
\]
and $\cO_{\cX}\widehat{\otimes}_{\cO_k}B_{\cris}$ is filtered by the subsheaves
\[
\cO\widehat{\otimes}_{\cO_k}\Fil^r B_{\cris}:=\varinjlim_{n\in \mathbb N} t^{-n}(\cO_{\cX}\widehat{\otimes}_{\cO_k} \Fil^{r+n}A_{\cris}), \quad r\in \mathbb Z,
\]
with $\cO_{\cX}\widehat{\otimes}_{\cO_k} \Fil^{r+n}A_{\cris}:=\varprojlim_n \cO_{\cX}\otimes_{\cO_k} \Fil^{r+n}A_{\cris}/p^n$.
\end{lemma}

\begin{cor}\label{quasicor} Let $\cX$ be a smooth formal scheme over $\cO_k$. Let $\LL$ be a crystalline lisse $\widehat{\Z}_p$-sheaf associated with a filtered convergent $F$-isocrystal $\cE$. Then there exists a natural quasi-isomorphism in the filtered derived category
\[
R\overline{w}_{\ast}(\LL\otimes_{\widehat{\Z}_p} \BBcr)\stackrel{\sim}{\lra} DR(\cE\otimes_{\cO_{\cX}}\cO_{\cX}\widehat{\otimes}_{\cO_k} B_{\cris}).
\]
If moreover $\cX$ is endowed with a lifting of Frobenius $\sigma$, then the isomorphism above is also compatible with the Frobenii deduced from $\sigma$ on both sides.
\end{cor}

\begin{proof}
Using the Poincar\'e lemma (Corollary \ref{poincare}), we get first a quasi-isomorphism which is strictly compatible with filtrations:
\[
\LL\otimes \BBcr\stackrel{\sim}{\lra} \LL\otimes DR(\cO\BBcr)=DR(\LL\otimes \cO\BBcr).
\]
As $\LL$ and $\cE$ are associated, there is a filtered isomorphism $\LL\otimes \cO\BBcr\simto w^{-1}\cE\otimes_{\cO_X^{\ur}} \cO\BBcr$ compatible with connection and Frobenius, from which we get the quasi-isomorphisms in the filtered derived category
\begin{equation}\label{eq.quasicor1}
\LL\otimes \BBcr\stackrel{\sim}{\lra}
DR(\LL\otimes \cO\BBcr)\stackrel{\sim}{\lra} DR(w^{-1}\cE\otimes \cO\BBcr).
\end{equation}
On the other hand, as $R^j\overline{w}_{\ast}\cO\BBcr=0$ for $j>0$ (Lemma \ref{quasi}), we obtain using projection formula that $R^j\overline{w}_{\ast}(w^{-1}\cE\otimes \cO\BBcr)=\cE\otimes R^j\overline w_{\ast}\cO\BBcr=0$ (note that $\cE$ is locally a direct factor of a finite free $\cO_{\cX}[1/p]$-module, hence one can apply projection formula here). In particular, each component of $DR(w^{-1}\cE\otimes \cO\BBcr)$ is $\overline{w}_{\ast}$-acyclic. Therefore,
\[
DR(\cE\otimes \overline{w}_{\ast}\cO\BBcr)\stackrel{\sim}{\lra}\overline w_{\ast}(DR(w^{-1}\cE\otimes \cO\BBcr))\stackrel{\sim}{\lra} R\overline{w}_{\ast} (DR(w^{-1}\cE\otimes \cO\BBcr)).
\]
Combining this with Lemma \ref{quasi}, we deduce the following quasi-isomorphisms in the filtered derived category
\begin{equation}\label{eq.quasicor2}
DR(\cE\otimes \cO_{\cX}\widehat{\otimes} B_{\cris})\stackrel{\sim}{\lra} DR(\cE\otimes\overline{w}_{\ast}\cO\BBcr )\stackrel{\sim}{\lra} R\overline{w}_{\ast} (DR(w^{-1}\cE\otimes \cO\BBcr)).
\end{equation}

The desired quasi-isomorphism follows from \eqref{eq.quasicor1} and \eqref{eq.quasicor2}. When moreover $\cX$ admits a lifting of Frobenius $\sigma$, one checks that both quasi-isomorphisms are compatible with Frobenius, hence the last part of our corollary.
\end{proof}

\begin{rk} Recall that $G_k$ denotes the absolute Galois group of $k$. Each element of $G_k$ defines a morphism of $U_{\bk}$ in the pro-\'etale site $X_{\proet}$ for any $\cU\in \cX_{\proet}$ with $U:=\cU_k$. Therefore, the object $R\overline{w}_{\ast}(\LL\otimes \BBcr)$ comes with  a natural Galois action of $G_k$. With this Galois action, one checks that the quasi-isomorphism in Corollary \ref{quasicor} is also Galois equivariant.
\end{rk}

Assume moreover that $\cX$ is proper over $\cO_k$. Let $\cE$ be a filtered convergent $F$-isocrystal on $\cX_0/\cO_k$, and $\cM$ an $F$-crystal on $\cX_0/\cO_k$  (viewed as a coherent $\cO_{\cX}$-module equipped with an integrable connection) such that $\mathcal E\simeq \mathcal M ^{\rm  an}(n)$ for some $n\in \mathbb N$ (Remark \ref{rk.CrystalVSIsoc}). The crystalline cohomology group $H^{i}_{\cris}(\cX_0/\cO_k,\cM)$ is an $\cO_k$-module of finite type endowed with a Frobenius $\psi$. In the following, the crystalline cohomology (or more appropriately, the rigid cohomology) of the convergent $F$-isocrystal $\cE$ is defined as
\[
H^{i}_{\cris}(\cX_0/\cO_k, \cE):=H^i_{\cris}(\cX_0/\cO_k, \cM)[1/p].
\]
It is a finite dimensional $k$-vector space equipped with the Frobenius $\psi/p^n$. Moreover, let $u=u_{\cX_0/\cO_k}$ be the morphism of topoi
\[
(\cX_0/\cO_k)_{\cris}^{\sim}\lra \cX_{\et}^{\sim}
\]
such that $u_{\ast}(\mathcal F)(\cU)=H^0_{\cris}(\cU_0/\cO_k,\mathcal F)$ for $\cU\in \cX_{\et}$. With the \'etale topology replaced by the Zariski topology, this is precisely the morphism $u_{\cX_0/\hat{S}}$ (with $\hat{S}=\Spf(\cO_k)$) considered in \cite[Theorem 7.23]{BO}. By \emph{loc.cit.}, there exists a natural quasi-isomorphism in the derived category
\begin{equation}\label{eq.derivedBO}
Ru_{\ast} \cM\stackrel{\sim}{\lra} DR(\cM),
\end{equation}
which induces an isomorphism $H^i_{\cris}(\cX_0/\cO_k,\cM)\simto H^i(\cX,DR(\cM))$. Thereby
\begin{equation}\label{eq.BO}
H^i_{\cris}(\cX_0/\cO_k,\cE)\stackrel{\sim}{\lra} H^i(\cX,DR(\cE)).
\end{equation}
On the other hand,  the de Rham complex $DR(\cE)$ of $\cE$ is filtered by its subcomplexes
\[
\Fil^r DR(\cE):=(\Fil^r \cE\stackrel{\nabla}{\lra} \Fil^{r-1}\cE\otimes \Omega_{X/k}^{1}\stackrel{\nabla}{\lra}\ldots).
\]
So the hypercohomology $ H^i(\cX,DR(\cE))$ has a descending filtration given by
\[
\Fil^r  H^i(\cX,DR(\cE)):=\mathrm{Im}\left(H^i(\cX,\Fil^r DR(\cE))\lra H^i(\cX,DR(\cE))\right).
\]
Consequently, through the isomorphism \eqref{eq.BO}, the $k$-space $H^i_{\cris}(\cX_0/\cO_k,\cE)$ is endowed naturally with a decreasing filtration.

\begin{thm} \label{main1} Assume that the smooth formal scheme $\cX$ is proper over $\cO_k$. Let $\mathcal E$ be a filtered convergent $F$-isocrystal on $\cX_0/\cO_k$ and $\LL$ a lisse $\widehat{\ZZ}_p$-sheaf on $X_{\proet}$. Assume that $\mathcal E$ and $\LL$ are associated. Then there is a natural filtered isomorphism
\begin{equation}\label{eq.iso}
H^i(X_{\bk,\proet},\LL\otimes_{\widehat{\ZZ}_p}\BB_{\cris}) \stackrel{\sim}{\longrightarrow} H_{\cris}^i(\cX_0/\cO_k, \cE)\otimes_{k} B_{\cris}
\end{equation}
 of $B_{\cris}$-modules, which is compatible with Frobenius and Galois action.
\end{thm}

\begin{proof} By Corollary \ref{quasicor}, we have the natural Galois equivariant quasi-isomorphism in the filtered derived category:
\[
R\Gamma(X_{\bk,\proet},\LL\otimes \BBcr)=R\Gamma(\cX,R\overline{w}_{\ast}(\LL\otimes \BBcr))\stackrel{\sim}{\lra}R\Gamma(\cX,DR(\cE\otimes \cO_{\cX}\widehat{\otimes} B_{\cris})).
\]
We claim that the natural morphism in the filtered derived category
\begin{equation}\label{eq:iso-for-the-cE}
R\Gamma(\cX,DR(\cE))\otimes_{\cO_k} A_{\cris}\lra R\Gamma(\cX_{\et}, DR(\cE\otimes_{\cO_{\cX}}\cO_{\cX}\widehat{\otimes}_{\cO_k}A_{\cris}))
\end{equation}
is an isomorphism. Let $\cM$ be an $F$-crystal on $\cX_0/\cO_k$ with $\cE=\cM^{\an}(n)$. Then, the similar natural morphism below is an isomorphism:
\begin{equation}\label{eq:iso-for-the-cM}
R\Gamma(\cX,DR(\cM))\otimes_{\cO_k} A_{\cris}\lra R\Gamma(\cX_{\et}, DR(\cM\otimes_{\cO_{\cX}}\cO_{\cX}\widehat{\otimes}_{\cO_k}A_{\cris})).
\end{equation}
Indeed, as $A_{\cris}$ is flat over $\cO_k$, $\cM\otimes_{\cO_{\cX}}\cO_{\cX}\widehat{\otimes}_{\cO_k}A_{\cris}\simeq \cM\widehat{\otimes}_{\cO_k}A_{\cris}$. So $DR(\cM\otimes_{\cO_{\cX}}\cO_{\cX}\widehat{\otimes}_{\cO_k}A_{\cris})=DR(\cM\widehat{\otimes}_{\cO_k}A_{\cris})$, and \eqref{eq:iso-for-the-cM} is an isomorphism by Lemma \ref{lem:flat-base-change}. Thus, to prove our claim, it suffices to check that \eqref{eq:iso-for-the-cE} induces quasi-isomorphisms on gradeds. Further filtering the de Rham complex by its naive filtration, we are reduced to checking the following isomorphism for $\mathcal A$ a coherent $\cO_{\cX}$-module:
\[
R\Gamma(\cX,\mathcal A)\otimes_{\cO_k} \cO_{\mathbb C_p} \stackrel{\sim}{\lra} R\Gamma(\cX,\mathcal A\otimes_{\cO_{\cX}}\widehat{\otimes}_{\cO_k}\cO_{\mathbb C_p})\simeq R\Gamma(\cX, \mathcal A\widehat{\otimes}_{\cO_k}\cO_{\mathbb C_p}),
\]
which holds because again $\cO_{\mathbb{C}_p}$ is flat over $\cO_k$ (Lemma \ref{lem:flat-base-change}). Consequently, inverting $t$ we obtain an isomorphism in the filtered derived category
\begin{equation*}
R\Gamma(\cX,DR(\cE))\otimes_{k} B_{\cris}\lra R\Gamma(\cX_{\et}, DR(\cE\otimes_{\cO_{\cX}} \cO_{\cX}\widehat{\otimes}_{\cO_k}B_{\cris})).
\end{equation*}
Thus we get a Galois equivariant quasi-isomorphism in the filtered derived category
\[
R\Gamma(X_{\bk,\proet},\LL\otimes \BBcr) \stackrel{\sim}{\lra} R\Gamma(\cX,DR(\cE))\otimes_k B_{\cris}.
\]
Combining it with \eqref{eq.BO}, we obtain the isomorphism \eqref{eq.iso} verifying the required properties except for the Frobenius compatibility.

To check the Frobenius compatibility, it suffices to check that the restriction to $H^i_{\cris}(\cX_0/\cO_k,\cE)\hookrightarrow H^i_{\cris}(\cX_0/\cO_k,\cE)\otimes_{k} B_{\cris}$ of the inverse of \eqref{eq.iso} is Frobenius-compatible. Let $\cM$ be an $F$-crystal on $\cX_0/\cO_k$ with $\cE=\cM^{\an}(n)$. Via the identification $H^i_{\cris}(\cX_0/\cO_k,\cE)=H^i_{\cris}(\cX_0/\cO_k,\cM)[1/p]$, the restriction map in question is induced from the following composed morphism at the level of derived category:
\begin{eqnarray*}
Ru_{\ast} \cM[1/p] \stackrel{\sim}{\lra} DR(\cM)[1/p]\stackrel{\sim}{\lra}  DR(\cE)\lra DR(\cE\otimes \cO_{\cX}\widehat{\otimes}_{k}B_{\cris})
\stackrel{\sim}{\lra} R\overline{w}_{\ast} (\LL\otimes \BBcr),
\end{eqnarray*}
where the first morphism is \eqref{eq.derivedBO}, and the last one is the inverse in the derived category of the quasi-isomorphism in Corollary \ref{quasicor}. Let us denote by $\theta$ the composite of these morphisms.
Let $\psi$ (resp. $\varphi$) be the induced Frobenius on $Ru_{\ast}\cM$ (resp. on $R\overline{w}_{\ast}(\LL\otimes \BBcr)$). One needs to check that $\varphi \circ \theta=\frac{1}{p^n}\theta\circ \psi$. This can be done locally on $\cX$. So let $\cU\subset \cX$ be a small open subset equipped with a lifting of Frobenius $\sigma$. Thus $\cM|_{\cU}$ (resp. $\cE|_{\cU}$) admits naturally a Frobenius, which we denote by $\psi_{\cU}$ (resp.  $\varphi_{\cU}$). Then all the morphisms above except the second one are Frobenius-compatible (see Corollary \ref{quasicor} for the last quasi-isomorphism). But by definition, under the identification $\cM[1/p]|_{\cU}\simeq \cE|_{\cU}$, the Frobenius $\varphi_{\cU}$ on $\cE$ corresponds exactly to $\psi_{\cU}/p^n$ on $\cM[1/p]$. This gives the desired equality $\varphi\circ \theta=\frac{1}{p^n}\theta\circ \psi$ on $\cU$, from which  the  Frobenius compatibility in \eqref{eq.iso} follows.
\end{proof}

\section{Primitive comparison on the pro-\'etale site}
\setcounter{equation}{0}
Let $\cX$ be a proper smooth formal scheme over $\cO_k$, with $X$ (resp. $\cX_0$) its generic (resp. closed) fiber. Let $\LL$ be a lisse $\widehat{\ZZ}_p$-sheaf on $X_{\proet}$.
In this section, we will construct a primitive  comparison isomorphism for any lisse $\widehat{\ZZ}_p$-sheaf $\LL$ on the pro-\'etale site $X_{\proet}$ (Theorem \ref{inout}). In particular, this primitive comparison isomorphism also holds for non-crystalline  lisse $\widehat{\ZZ}_p$-sheaves, which may lead to interesting arithmetic applications. On the other hand, in the case that $\LL$ is crystalline, such a result and Theorem \ref{main1} together give rise to the crystalline comparison isomorphism between \'etale cohomology and crystalline cohomology (Theorem \ref{thm.comp}).

We shall begin with some preparations. %The first lemma is well-known.

\begin{lemma}\label{jannsen}Let $(\mathcal F_n)_{n\in \mathbb N}$ be a projective system of abelian sheaves on a site $T$. Then for any object $Y\in T$ and any $i\in \mathbb Z$, there exists a natural exact sequence:
\[
0\lra R^1\varprojlim H^{i-1}(Y,\mathcal F_n)\lra H^i(Y,R\varprojlim \mathcal F_n)\lra \varprojlim H^i(Y,\mathcal F_n)\lra 0.
\]
\end{lemma}

\begin{proof} This is essentially \cite[(1.6) Proposition]{Jan}. Let $\mathrm{Sh}$ denote the category of abelian sheaves on $T$, and $\mathrm{Sh}^{\mathbb N}$ the category of projective systems of abelian sheaves indexed by $\mathbb N$. Let $\mathrm{Ab}$ denote the category of abelian groups. Consider the functor
\[
\tau \colon \mathrm{Sh}^{\mathbb N}\longrightarrow \mathrm{Ab}, \quad (\mathcal G_n) \mapsto \varprojlim \Gamma(Y,\mathcal G_n).
\]
Then $\tau$ is left exact, and we can consider its right derived functor $R\tau (\mathcal F_n)$. By \cite[(1.6) Proposition]{Jan}, we have a short exact sequence for each $i\in \mathbb Z$
\[
0\lra R^1\varprojlim H^{i-1}(Y,\mathcal F_n)\lra R^i\tau (\mathcal F_n)\lra \varprojlim H^i(Y,\mathcal F_n)\lra 0.
\]
One the other hand, write $\tau$ as the composite of the following two functors
\[
\xymatrix{\mathrm{Sh}^{\mathbb N}\ar[r]^{\varprojlim} &  \mathrm{Sh}\ar[r]^{\Gamma(Y,-)}\ar[r] & \mathrm{Ab}}.
\]
The functor $\varprojlim \colon \mathrm{Sh}^{\mathbb N}\to \mathrm{Sh}$ admits an exact left adjoint given by sending a sheaf to its associated constant projective system, so it sends injectives to injectives. Thus $R\Gamma(Y,R\varprojlim \mathcal F_n)\simeq R\tau (\mathcal F_n)$, and $H^i(Y,R\varprojlim \mathcal F_n)\simeq R^{i}\tau (\mathcal F_n)$ for each $i$. Together with the short exact sequence above, we obtain our lemma.
\end{proof}

\begin{lemma}\label{lem.finiteness} Let $\LL$ be a lisse $\widehat{\ZZ}_p$-sheaf on $X_{\bk,\proet}$, and $\LL_n:=\LL/p^n$ for $n\in \mathbb N$. Then, for $i\in \mathbb Z$, $H^i(X_{\bk,\proet},\LL)\stackrel{\sim}{\ra}\varprojlim_n H^i(X_{\bk,\proet},\LL_n)$. Moreover, $H^i(X_{\bk, \proet},\LL)$ is a $\Zp$-module of finite type, and $H^i(X_{\bk,\proet},\LL)=0$ whenever $i\notin [0,2 \dim(X)]$.
\end{lemma}
\begin{proof} All the cohomology groups below are computed in the pro-\'etale site, so we shall omit the subscript {\textquotedblleft pro\'et\textquotedblright} from the notations.

Thanks to \cite[Proposition 8.2]{Sch}, $R^j\varprojlim_n \LL_n=0$ for $j>0$. So $\LL\stackrel{\sim}{\ra}R\varprojlim_n \LL_n$. Furthermore, by \cite[Theorem 5.1]{Sch}, $H^{i-1}(X_{\bk},\LL_n)$ is finite for all $n\in \mathbb N$. So $R^1\varprojlim H^{i-1}(X_{\bk},\LL_n)=0$. Consequently, by Lemma \ref{jannsen}, the morphism
\begin{equation*}
H^i(X_{\bk},\LL)\lra \varprojlim H^i(X_{\bk},\LL_n)
\end{equation*}
is an isomorphism, giving the first part of our lemma. In particular, $H^i(X_{\bk},\LL)=0$ whenever $i\notin[0,2\dim(X)]$ according to \cite[Theorem 5.1]{Sch}.

Let $\LL_{\textrm{tor}}$ be the torsion subsheaf of $\LL$. The remaining part of our lemma follows from the corresponding statements for $\LL_{\textrm{tor}}$ and for $\LL/\LL_{\textrm{tor}}$. Therefore, we assume that $\LL$ is either  of torsion or locally free of finite rank. In the first case, we reduce to the finiteness statement of Scholze (\cite[Theorem 5.1]{Sch}). So it suffices to consider the case where $\LL$ is locally free. Then, we have the exact sequence
\[
0\longrightarrow  \LL \stackrel{p^n}{\longrightarrow} \LL\longrightarrow \LL_n\longrightarrow 0,
\]
inducing the following short exact sequence
\begin{equation}\label{eq.exact0}
0\longrightarrow H^i(X_{\bk},\LL)/p^n \longrightarrow H^i(X_{\bk},\LL_n)\longrightarrow H^{i+1}(X_{\bk},\LL)[p^n]\longrightarrow 0.
\end{equation}
By the first part of our lemma, $H^{i+1}(X_{\bk},\LL)\stackrel{\sim}{\to} \varprojlim H^{i+1}(X_{\bk},\LL_n)$ is a pro-$p$ abelian group, hence it does not contain any element infinitely divisible by $p$. Thus, $\varprojlim  (H^{i+1}(X_{\bk},\LL)[p^n])=0$ (the transition map is multiplication by $p$). From the exactness of \eqref{eq.exact0}, we deduce a canonical isomorphism
\[
\varprojlim \left(H^i(X_{\bk},\LL)/p^n\right)\simto \varprojlim H^{i}(X_{\bk},\LL_n).
\]
So $H^i(X_{\bk},\LL)\simto \varprojlim_n H^i(X_{\bk},\LL)/p^n$, and $H^i(X_{\bk},\LL)$ is $p$-adically complete. Thus it can be generated as a $\Zp$-module by a family of elements whose images in $H^i(X_{\bk},\LL)/p$ generate it as an $\Fp$-vector space. The latter is finite dimensional over $\Fp$: recall the inclusion $H^i(X_{\bk},\LL)/p\hookrightarrow H^i(X_{\bk},\LL_1)$ by \eqref{eq.exact0}. So the $\Zp$-module $H^i(X_{\bk},\LL)$ is of finite type, as desired.
\end{proof}

The primitive form of the crystalline comparison isomorphism on the pro-\'etale site is as follows.

\begin{thm} \label{inout}Let $\LL$ be a lisse $\widehat{\ZZ}_p$-sheaf on $X_{\proet}$. Then the natural morphism of $B_{\cris}$-modules below is a filtered isomorphism
\begin{equation}\label{eq.PrimitiveIso}
H^i(X_{\bk,\proet},\LL)\otimes_{\Zp}B_{\cris}\simto H^i(X_{\bk,\proet},\LL\otimes_{\widehat{\ZZ}_p}\mathbb B_{\cris}),
\end{equation}
compatible with Galois action and Frobenius. 
\end{thm}

\begin{proof}

In the following, all the cohomologies are computed on the pro-\'etale site, so we  omit the subscript {\textquotedblleft \textrm{pro\'et}\textquotedblright} from the notations.

If $\LL$ is of torsion, our theorem is obvious since both sides of \eqref{eq.PrimitiveIso} are trivial. Therefore, it suffices to consider the case where $\LL$ is locally a free lisse $\widehat{\ZZ}_p$-module. Let $\LL_n=\LL/p^n$, $n\in \mathbb N$. So, we have the short exact sequence \eqref{eq.exact0}, from which we deduce a short exact sequence of projective systems since $A_{\cris}$ is flat over $\Zp$
\begin{equation*}
\begin{array}{rcl}
0\longrightarrow \left(H^i(X_{\bk},\LL)\otimes A_{\cris}/p^n\right)_n & \longrightarrow & \left(H^i(X_{\bk},\LL_n)\otimes A_{\cris}\right)_n \\ & \longrightarrow & \left(\left(H^{i+1}(X_{\bk},\LL)[p^n]\right)\otimes A_{\cris}\right)_n\longrightarrow 0.
\end{array}
\end{equation*}
Because $H^{i+1}(X_{\bk},\LL)$ is a finite $\Zp$-modules, $\varprojlim_n (H^{i+1}(X_{\bk},\LL)[p^n]\otimes A_{\cris})=0$ (the transition map is multiplication by $p$). Thus, we get
\begin{equation}\label{eq.limit-tensor-Acris-iso}
H^i(X_{\bk},\LL)\otimes A_{\cris}\simeq \varprojlim_n H^i(X_{\bk},\LL)\otimes A_{\cris}/p^n\stackrel{\sim}{\lra} \varprojlim_n \left(H^i(X_{\bk},\LL_n)\otimes A_{\cris}\right).
\end{equation}
Here we have the first identification since $H^i(X_{\bk},\LL)\otimes A_{\cris}$ is $p$-adically complete thanks to the fact that $H^i(X_{\bk},\LL)$ is a finite $\Zp$-module.

Next, we claim that, for all $i\geq 0$, the canonical map of $A_{\cris}/p^n$-modules
\begin{equation}\label{eq.isomorphism-Acris-mod-pn}
H^i(X_{\bk}, \LL_n)\otimes_{\Zp}A_{\cris}\lra H^i(X_{\bk},\LL_n\otimes_{\widehat{\ZZ}_p}\mathbb A_{\cris})
\end{equation}
is an almost isomorphism. Since $A_{\cris}$ and $\mathbb A_{\cris}$ are flat respectively over $\Zp$ and $\widehat{\ZZ}_p$, by induction on $n$, it suffices to show that the natural map of $A_{\cris}/p$-modules
\[
H^i(X_{\bk},\mathbb K)\otimes_{\Zp}A_{\cris}\lra H^i(X_{\bk}, \mathbb K\otimes_{\widehat{\ZZ}_p}\mathbb A_{\cris})
\]
is an almost isomorphism, where $\mathbb K$ is an $\mathbb F_p$-local system on $X_{\proet}$. In this case, one can rewrite the morphism above as
\begin{equation}\label{eq.isomorphism-Acris-mod-p}
H^i(X_{\bk},\mathbb K)\otimes_{\mathbb F_p}A_{\cris}/p\lra H^i(X_{\bk}, \mathbb K\otimes_{\mathbb F_p}\mathbb A_{\cris}/p).
\end{equation}
Recall the following identification of $A_{\cris}/p$ (see \cite[Proposition 6.1.2]{Bri})
\[
A_{\cris}/p\stackrel{\sim}{\lra} (\cO_{\Cp}^{\flat}/(p^{\flat})^p)[\delta_i:i \in \mathbb N]/(\delta_i^p:i\in \mathbb N),
\]
with $\delta_i$ the image of $\xi^{[p^{i+1}]}$. Similarly, on $X_{\proet}/X_{\bk}$, we have
\[
\AAcr/p\stackrel{\sim}{\lra} (\cO_{X}^{\flat+}/(p^{\flat})^p)[\delta_i:i\in \mathbb N]/(\delta_i^p:i\in \mathbb N).
\]
As $X_{\bk}$ is qcqs, to show that \eqref{eq.isomorphism-Acris-mod-p} is an almost isomorphism, it suffices to check that it is the case for the map
\[
H^i(X_{\bk},\mathbb K)\otimes_{\mathbb F_p}\cO_{\Cp}^{\flat}/(p^{\flat})^p\lra H^i(X_{\bk}, \mathbb K\otimes_{\mathbb F_p}\cO_{X}^{\flat+}/(p^{\flat})^p).
\]
Using the $p^{\flat}$-adic filtration on $\cO_{\Cp}^{\flat}/(p^{\flat})^p$ and on $\cO_X^{\flat+}/(p^{\flat})^p$, one reduces further to showing that the natural map
\[
H^i(X_{\bk},\mathbb K)\otimes_{\mathbb F_p}\cO_{\Cp}^{\flat}/(p^{\flat})\lra H^i(X_{\bk}, \mathbb K\otimes_{\mathbb F_p}\cO_{X}^{\flat+}/(p^{\flat})),
\]
is an almost isomorphism. But this is proved in \cite[Theorem 5.1]{Sch} since $\cO_{\Cp}^{\flat}/p^{\flat}\simeq \cO_{\Cp}/p$ and $\cO_X^{\flat +}/p^{\flat}\simeq \cO_X^+/p$: recall that the almost-setting adopted here for $A_{\cris}/p$-modules is the same as the one used by Scholze. Consequently, the map \eqref{eq.isomorphism-Acris-mod-pn} is an almost isomorphism. Varying $n$ in \eqref{eq.isomorphism-Acris-mod-pn}, we obtain a morphism of projective systems of $A_{\cris}$-modules, with kernel and cokernel killed by $\mathcal I$. Passing to limits relative to $n$ and using \eqref{eq.limit-tensor-Acris-iso}, one deduces a natural morphism of $A_{\cris}$-modules
\[
H^i(X_{\bk},\LL)\otimes_{\Zp}A_{\cris}\simeq \varprojlim_n H^i(X_{\bk},\LL_n)\otimes_{\Zp}A_{\cris}\lra \varprojlim_n H^i(X_{\bk},\LL_n\otimes_{\widehat{\ZZ}_p}\AAcr),
\]
with kernel and cokernel killed by $\mathcal I^2$. Moreover, we have $\mathcal I\cdot R^1\varprojlim_n H^i(X_{\bk},\LL_n\otimes \AAcr)=0$ since $R^1\varprojlim \left(H^i(X_{\bk},\LL_n)\otimes A_{\cris}\right)=0$.

Then, we claim that the $A_{\cris}$-module $R^j\varprojlim (\LL_n\otimes_{\widehat{\ZZ}_p}\AAcr)$ is killed by $\mathcal I^2$ for $j>0$. The question being local on $X_{\proet}$, we may and do assume $\LL=M\otimes_{\Zp}\widehat{\ZZ}_p$ with $M$ a finitely generated free $\Zp$-module (recall that we have assumed that $\LL$ is locally free over $\widehat{\ZZ}_{p}$). So our claim in this case follows from Lemma \ref{lem.technical-I}. As a result, in the spectral sequence below
\[
E_2^{i,j}=H^i(X_{\bk},R^j\varprojlim (\LL_n\otimes \AAcr))\Longrightarrow H^{i+j}(X_{\bk}, R\varprojlim (\LL_n\otimes\AAcr)),
\]
we have $\mathcal I^2\cdot E_{2}^{i,j}=0$ for $j>0$ and $E_{\infty}^{i,0}=E_{i+1}^{i,0}$. Moreover, the natural surjection $E_{2}^{i,0}\ra E_{\infty}^{i,0}$ has kernel killed by $\mathcal I^{2i-2}$. It follows that the canonical map
\[
H^i(X_{\bk},\LL\otimes_{\widehat{\ZZ}_p}\AAcr) \lra H^i(X_{\bk},R\varprojlim (\LL_n\otimes_{\widehat{\ZZ}_p}\AAcr))
\]
has kernel killed by $\mathcal I^{2i-2}$ and cokernel killed by $\mathcal I^{2i}$. On the other hand, by Lemma \ref{jannsen}, the kernel of the canonical surjective morphism
\[
H^i(X_{\bk},R\varprojlim_n (\LL_n\otimes_{\widehat{\ZZ}_p}\AAcr))\lra \varprojlim_n H^i(X_{\bk},\LL_n\otimes_{\widehat{\ZZ}_p}\AAcr)
\]
is $R^1\varprojlim_n H^{i-1}(X_{\bk},\LL_n\otimes_{\widehat{\ZZ}_p}\AAcr)$, thus is killed by $\mathcal I$ by what we have shown above. Therefore, the kernel and cokernel of the composed map
\[
H^i(X_{\bk},\LL\otimes \AAcr)\lra H^i(X_{\bk}, R\varprojlim_n (\LL_n\otimes \AAcr))\lra \varprojlim_n H^i(X_{\bk},\LL_n\otimes \AAcr)
\]
are killed by $\mathcal I^{2i}$. Finally, from the commutative square
\[
\xymatrix{H^i(X_{\bk},\LL)\otimes A_{\cris}\ar[rr]^<<<<<<<<<<{\sim}_<<<<<<<<<<{\eqref{eq.limit-tensor-Acris-iso}}\ar[d]&  & \varprojlim \left(H^i(X_{\bk},\LL_n)\otimes A_{\cris}\right)\ar[d]^{\textrm{iso. up to }\mathcal I^2}_{\varprojlim_n \eqref{eq.isomorphism-Acris-mod-pn}}  \\ H^i(X_{\bk},\LL\otimes \AAcr)\ar[rr]^{\textrm{iso. up to }\mathcal I^{2i}} & & \varprojlim H^i(X_{\bk}, \LL_n\otimes\AAcr)},
\]
we deduce that the natural map below has kernel and cokernel killed by $\mathcal I^{2i+2}$, hence by $t^{2i+2}$:
\[
H^i(X_{\bk},\LL)\otimes_{\Zp} A_{\cris} \lra H^i(X_{\bk}, \LL\otimes_{\widehat{\ZZ}_p}\AAcr).
\]
Inverting $t$, we get the required isomorphism \eqref{eq.PrimitiveIso}.

We still need to check that \eqref{eq.PrimitiveIso} is compatible with the extra structures. Clearly only the strict compatibility with filtrations needs verification, and it suffices to check this on gradeds. So we reduce to showing that the natural morphism is an isomorphism:
\[
H^i(X_{\bk},\LL)\otimes_{\Zp} \mathbb C_p(j) \lra H^i(X_{\bk},\LL\otimes \widehat{\cO}_X(j)).
\]
Twisting, one reduces to $j=0$, which is given by the following lemma.
\end{proof}
\begin{lemma}\label{primitiveofstructuresheaf} Let $\LL$ be a lisse $\widehat{\ZZ}_p$-sheaf on $X_{\bk,\proet}$. Then the following natural morphism is an isomorphism:
\[
H^i(X_{\bk,\proet}, \LL)\otimes_{\Zp}\mathbb C_p \stackrel{\approx}{ \lra} H^i\left(X_{\bk,\proet}, \LL\otimes_{\widehat{\Z}_p} \widehat{\cO}_X\right),
\]
where $\widehat{\cO}_{X}$ is the completed structural sheaf of $X_{\bk,\proet}$ and $\mathbb C_p=\widehat{\bk}$.
\end{lemma}

\begin{proof} It suffices to show that the natural morphism of $\cO_{\Cp}$-modules
\[
H^i(X_{\bk,\proet}, \LL)\otimes_{\Zp}\cO_{\Cp}\lra H^i(X_{\bk,\proet}, \LL\otimes_{\widehat{\ZZ}_p}\widehat{\cO}_X^+)
\]
has kernel and cokernel annihilated by some power of $\mathcal I\cO_{\Cp}$. The proof is similar to that of the first part of Theorem \ref{inout}, so we omit the details here.
\end{proof}

Recall that the notion of lisse $\Zp$-sheaf on $X_{\et}$ and lisse $\widehat{\ZZ}_p$-sheaf on $X_{\proet}$ are equivalent. Combining Theorem \ref{main1} and Theorem \ref{inout}, we finally deduce the following crystalline comparison theorem:
\begin{thm}\label{thm.comp}Let $\cX$ be a proper smooth formal scheme over $\cO_k$, with $X$ (resp. $\cX_0$) its generic (resp. closed) fiber. Let $\LL$ be a lisse $\widehat{\ZZ}_p$-sheaf on $X_{\proet}$, associated to a filtered $F$-isocrystal $\mathcal E$ on $\cX_0/\cO_k$. Then there exists a functorial filtered isomorphism
\[
H^i(X_{\bk,\et},\LL)\otimes_{\Zp}B_{\cris}\simto H^i_{\cris}(\cX_0/\cO_k, \mathcal E)\otimes_{\cO_k}B_{\cris}
\]
of $B_{\cris}$-modules, compatible with Galois action and Frobenius.
\end{thm}

\section{Comparison isomorphism in the relative setting}\label{sec.relative}
\setcounter{equation}{0}

Let $f\colon \cX\to \cY$ be a  smooth morphism between two smooth formal schemes over $\Spf(\cO_k)$ of relative dimension $d\geq 0$. The induced morphism between the generic fibers will be denoted by $f_k\colon X\to Y$. We shall denote by $w_{\cX}$ (resp. $w_{\cY}$) the natural morphism of topoi $X_{\proet}^{\sim}\to \cX_{\et}^{\sim}$ (resp. $Y_{\proet}^{\sim}\to \cY_{\et}^{\sim}$). By abuse of notation, the morphism of topoi $X_{\proet}^{\sim}\to Y_{\proet}^{\sim}$ will be still denoted by $f_k$.

Let $\nabla_{X/Y}: \cO\mathbb A_{\cris,X}\ra \cO\mathbb A_{\cris,X}\otimes_{\cO_X^{\ur+}}\Omega_{X/Y}^{1,\ur+}$ be the natural relative derivation, where $\Omega_{X/Y}^{1,\ur+}:=w_{\cX}^*\Omega_{\cX/\cY}^1$.

\begin{prop}\label{relativepoincare}\begin{enumerate}
\item \emph(Relative Poincar\'e lemma\emph) The following sequence of pro-\'etale sheaves is exact and strict with respect to the filtration giving $\Omega^{i,\ur+}_{X/Y}$ degree $i$:
\[
\begin{array}{c}
0\lra \mathbb A_{\cris,X}\widehat{\otimes}_{f_k^{-1}\mathbb A_{\cris,Y}}f_k^{-1}\cO\mathbb A_{\cris,Y}\ra\cO\mathbb A_{\cris,X}\stackrel{\nabla_{X/Y}}{\lra}\cO\mathbb A_{\cris,X}\otimes_{\cO_X^{\ur+}}\Omega_{X/Y}^{1,\ur+} \\
\stackrel{\nabla_{X/Y}}{\lra}\cdots\stackrel{\nabla_{X/Y}}{\lra}\cO\mathbb A_{\cris,X}^+\otimes_{\cO_X^{\ur+}}\Omega^{d,\ur+}_{X/Y}\lra 0.
\end{array}
\]
In particular, the connection $\nabla_{X/Y}$ is integrable and satisfies Griffiths transversality with respect to the filtration, i.e., $
\nabla_{X/Y} (\Fil^i\cO\mathbb A_{\cris,X})\subset \Fil^{i-1}\cO\mathbb A_{\cris,X}\otimes \Omega_{X/Y}^{1,\ur+}$.

\item Suppose the Frobenius on $\cX_0$ (resp.  $\cY_0$) lifts to a Frobenius $\sigma_X$ (resp. $\sigma_Y$) on the formal scheme $\cX$ (resp. $\cY$) and they commute with $f$. Then the induced Frobenius  $\varphi_{X}$ on $\cO\mathbb A_{\cris,X}$ is horizontal with respect to $\nabla_{X/Y}$.
\end{enumerate}

\end{prop}
\begin{proof} The proof is routine (cf. Proposition \ref{iso}), so we omit the detail here.
\end{proof}

For the relative version of the crystalline comparison, we shall need the following primitive comparison in the relative setting.

\begin{prop}\label{part1relative} Let $f\colon \cX\to \cY$ be a proper smooth morphism between two smooth formal schemes over $\cO_k$.
 Let $\LL$ be a lisse $\widehat{\ZZ}_p$-sheaf on $X_{\proet}$. Suppose that  $R^if_{k*}\LL$ is a lisse $\widehat{\ZZ}_p$-sheaf for all $i\geq 0$. Then, the canonical morphism
\begin{equation}\label{eq.relativeprimitive}
(R^if_{k*}\LL)\otimes_{\widehat{\ZZ}_p}\BB_{\cris,Y}\lra R^if_{k*}(\LL\otimes_{\widehat{\ZZ}_p}\BB_{\cris,X})
\end{equation}
is a filtered isomorphism, compatible with Frobenius. Similarly, the natural morphism
\begin{equation}\label{eq.relativeprimitiveObcris}
(R^if_{k*}\LL)\otimes_{\widehat{\ZZ}_p}\cO\BB_{\cris,Y}\lra R^if_{k*}(\LL\otimes_{\widehat{\ZZ}_p}\mathbb A_{\cris,X}\widehat{\otimes}_{f_k^{-1}\mathbb A_{\cris,Y}}f_k^{-1}\cO\mathbb A_{\cris,Y}[1/t])
\end{equation}
is a filtered isomorphism, compatible with Frobenius and connections.
 \end{prop}

\begin{proof}
The proof is similar to that of Theorem \ref{inout}, so we shall only give a sketch here. As in the proof of Theorem \ref{inout}, to show our proposition, it suffices to consider the case where $\LL$ is locally free over $\widehat{\ZZ}_p$. Let $\LL_n=\LL/p^n$, $n \in \mathbb N$. Since $\LL$ has no $p$-torsion, we have short exact sequences of lisse $\widehat{\ZZ}_p$-sheaves on $Y$:
\[
0\lra R^if_{k*}(\LL)/p^n\lra R^if_{k*}(\LL_n)\lra R^{i+1}f_{k*}(\LL)[p^n]\lra 0,\quad n\in \mathbb N,
\]
inducing an exact sequence of projective systems as $\mathbb A_{\cris,Y}$ is flat over $\widehat{\ZZ}_p$:
\[
\begin{array}{rcl}
0\lra \left(R^if_{k*}(\LL)\otimes \mathbb A_{\cris, Y}/p^n\right)_n & \lra & \left(R^if_{k*}(\LL_n)\otimes \mathbb A_{\cris, Y}\right)_n \\ & \lra &  \left(R^{i+1}f_{k*}(\LL)[p^n]\otimes \mathbb A_{\cris, Y}\right)_n\lra 0.
\end{array}
\]
Because $R^{i+1}f_{k*}(\LL)$ is a lisse $\widehat{\ZZ}_p$-module, there exists $N\in \mathbb N$ such that $p^N$ kills $R^{i+1}f_{k*}(\LL)[p^n]$ for all $n$. Thus, the composed transition map in the last projective system of the sequence above is zero:
\[
R^{i+1}f_{k*}(\LL)[p^{n+N}]\otimes \mathbb A_{\cris,Y}\lra R^{i+1}f_{k*}(\LL)[p^n]\otimes \mathbb A_{\cris,Y},\quad x\mapsto p^Nx.
\]
Therefore, $R^j\varprojlim_n (R^{i+1}f_{k*}(\LL)[p^n]\otimes \mathbb A_{\cris,Y})=0$ for every $j\in \mathbb Z$, and thus
\[
R^j\varprojlim_n (R^if_{k*}(\LL)\otimes \mathbb A_{\cris, Y}/p^n)\stackrel{\sim}{\lra}R^j\varprojlim_n \left(R^if_{k*}(\LL_n)\otimes \mathbb A_{\cris,Y}\right)
\]
for all $j\in \mathbb Z$. In particular,
\begin{equation}\label{eq.relative-iso0}
\begin{array}{lcr}
R^if_{k*}(\LL)\otimes \mathbb A_{\cris,Y}& \simeq & \varprojlim_n (R^if_{k*}(\LL)\otimes \mathbb A_{\cris, Y}/p^n) \\ & \stackrel{\sim}{\lra} & \varprojlim_n \left(R^if_{k*}(\LL_n)\otimes \mathbb A_{\cris,Y}\right),
\end{array}
\end{equation}
and $R^j\varprojlim_n \left(R^if_{k*}(\LL_n)\otimes \mathbb A_{\cris,Y}\right)\simeq R^j\varprojlim_n \left(R^if_{k*}(\LL)\otimes \mathbb A_{\cris,Y}/p^n\right)$ is killed by $\mathcal I^{2}$ whenever $j>0$ as this is the case for $R^j\varprojlim_n (\mathbb A_{\cris,Y}/p^n)$ by Corollary \ref{cor.acyclicityBcris} and as $R^if_{k*}(\LL)$ is a lisse $\widehat{\ZZ}_p$-sheaf.

Next, with the help of \cite[Corollary 5.11]{Sch}, the same argument as in the proof of Theorem \ref{inout} yields almost isomorphisms
\begin{equation}\label{eq.relative-iso1}
R^if_{k*}(\LL_n)\otimes \mathbb A_{\cris,Y}\stackrel{\approx}{\lra} R^if_{k*}(\LL_n\otimes \mathbb A_{\cris,X}) , \quad n\in \mathbb N.
\end{equation}
So, the kernel and the cokernel of the natural map below are killed by $\mathcal I^2$:
\begin{equation}\label{eq.relative-iso2}
\varprojlim_n(R^if_{k*}(\LL_n)\otimes \mathbb A_{\cris,Y})\lra \varprojlim_n R^if_{k*}(\LL_n\otimes \mathbb A_{\cris,X}).
\end{equation}
Moreover, for $j>0$, from \eqref{eq.relative-iso1}, we find $\mathcal I^4\cdot R^j\varprojlim_n R^if_{k*}(\LL_n\otimes \mathbb A_{\cris,X})=0$ since $\mathcal I^2\cdot R^j\varprojlim_n (R^if_{k*}(\LL_n)\otimes \mathbb A_{\cris,Y})=0$ as observed at the end of the last paragraph. Therefore, by a standard argument using the spectral sequence
\[
E_2^{a,b}=R^a\varprojlim_n R^bf_{k*}(\LL_n\otimes \mathbb A_{\cris,X})\Longrightarrow \mathcal H^{a+b}(R\varprojlim_n Rf_{k*}(\LL_n\otimes \mathbb A_{\cris,X})),
\]
one checks that the kernel and the cokernel of the map
\[
\begin{array}{c}
\mathcal H^i(R\varprojlim_n Rf_{k*}(\LL_n\otimes \mathbb A_{\cris,X}))\simeq R^if_{k*}(R\varprojlim_n (\LL_n\otimes \mathbb A_{\cris,X}))  \\    \lra    E_{2}^{0,i}=\varprojlim_n R^if_{k*}(\LL_n\otimes \mathbb A_{\cris,X}) \end{array}
\]
are killed by some power of $\mathcal I$. On the other hand, as shown in the proof of Theorem \ref{inout}, $\mathcal I^2\cdot R^j\varprojlim_n (\LL_n\otimes \mathbb A_{\cris,X})=0$ if $j>0$. So the kernel and the cokernel of
\[
R^if_{k*}(\LL\otimes \mathbb A_{\cris,X})\lra R^if_{k*}(R\varprojlim_n(\LL_n\otimes \mathbb A_{\cris,X}))
\]
are killed by some power of $\mathcal I$. Consequently, the kernel and cokernel of the map
\begin{equation}\label{eq.relative-iso3}
R^if_{k*}(\LL\otimes \mathbb A_{\cris,X})\lra \varprojlim_n R^if_{k*}(\LL_n\otimes \mathbb A_{\cris,X})
\end{equation}
are killed by some power of $\mathcal I$. Combining the morphisms \eqref{eq.relative-iso0}, \eqref{eq.relative-iso2} and \eqref{eq.relative-iso3}, we deduce that the kernel and the cokernel of the map
\[
R^if_{k*}(\LL)\otimes \mathbb A_{\cris,Y}\lra R^if_{k*}(\LL\otimes \mathbb A_{\cris,X})
\]
are killed by some power of $\mathcal I$, thus also killed by some power of $t$. Inverting $t$, we obtain the desired isomorphism \eqref{eq.relativeprimitive}.

We need to verify the compatibility of the isomorphism \eqref{eq.relativeprimitive} with the extra structures. It clearly respects  Frobenius structures. To check the strict compatibility with respect to filtrations, by taking grading quotients, we just need to show that for each $r\in \mathbb N$, the following natural morphism
\[
R^if_{k*}\LL\otimes \widehat{\cO}_{Y}(r)\lra R^if_{k*}(\LL\otimes \widehat{\cO}_X(r))
\]
is an isomorphism: it is a local question, hence it suffices to show this after restricting the latter morphism to $Y_{\bk}$. As $\widehat{\cO}_X(r)|_{X_{\bk}}\simeq \widehat{\cO}_X|_{X_{\bk}}$ and $\widehat{\cO}_Y(r)|_{Y_{\bk}}\simeq \widehat{\cO}_Y|_{Y_{\bk}}$, we then reduce to the case where $r=0$. The proof of the latter statement is similar as above, so we omit the details here.

Finally, using Proposition \ref{iso}, a similar proof as above shows that the map \eqref{eq.relativeprimitiveObcris} is a filtered isomorphism compatible with Frobenius and connections.
\end{proof}

For a sheaf of $\cO_{\cX}$-modules $\cF$ with an $\cO_{\cY}$-linear connection $\nabla\colon \cF\to \cF\otimes\Omega^1_{\cX/\cY}$, we denote the de Rham complex of $\cF$ as:
\[
DR_{X/Y}(\cF):=(\ldots \lra 0\lra \cF
\stackrel{\nabla}{\lra}\cF\otimes_{\cO_{\cY}} \Omega^1_{\cX/\cY}\stackrel{\nabla}{\lra}\ldots).
\]
The same rule applies if we consider an $\cO_{X}^{\rm un}$-module endowed with an $\cO_Y^{\rm un}$-linear connection, etc.

In the lemma below, assume $\cY=\Spf(A)$ is affine and is \'etale over a torus $\mathcal S=\Spf(\cO_k\{S_1^{\pm 1},\ldots, S_{\delta}^{\pm 1}\})$. For each $1\leq j\leq \delta$, let $(S_j^{1/p^n})_{n\in \mathbb N}$ be a compatible family of $p$-power roots of $S_j$. As in Proposition \ref{iso}, set
\[
\widetilde{Y}:=\left(Y\times_{\mathcal S_k}\Spa\left(k\{S_1^{\pm 1/p^{n}},\ldots, S_{\delta}^{\pm 1/p^{n}}\},\cO_k\{S_1^{\pm 1/p^{n}},\ldots, S_{\delta}^{\pm 1/p^{n}}\}\right)\right)_{n\in \mathbb N}\in Y_{\proet}.
\]

\begin{lemma}\label{quasirelative} Let $V\in Y_{\proet}$ be an affinoid perfectoid which is pro-\'etale over $\widetilde{Y}_{\bk}$, with $\widehat{V}=\Spa(R,R^+)$. Let $w_V$ be the composite of natural morphisms of topoi
\[
w_V\colon X_{\proet}^{\sim}\slash X_{V} \lra X_{\proet}^{\sim}\stackrel{w}{\lra} \mathcal X_{\et}^{\sim}.
\]
\begin{enumerate}
\item[(1)] For any $j>0$, we have $R^jw_{V\ast}\cO\BBcr=0$, and the natural morphism
\[
\cO_{\cX}\widehat{\otimes}_A\cO\BB_{\cris,Y}(V) \lra w_{V*}(\cO\BB_{\cris,X})
\]
is an isomorphism.

\item[(2)] For any $r\in \mathbb Z$ and any $j>0$, we have $R^jw_{V\ast}(\Fil^r\cO\BBcr)=0$. Moreover, the natural morphism
\[
\cO_{\cX}\widehat{\otimes}_{A}\Fil^r\cO\BB_{\cris,Y}(V)\to w_{V\ast}(\Fil^r\cO\BB_{\cris,X})
\]
is an isomorphism.
\end{enumerate}
Here  $\cO_{\cX}\widehat{\otimes}_A\cO\BB_{\cris, Y}(V):=\left(\cO_{\cX}\widehat{\otimes}_A\cO\mathbb{A}_{\cris,Y}(V)\right)[1/t]$ with
\[
\cO_{\cX}\widehat{\otimes}_A\cO\mathbb{A}_{\cris,Y}(V):=\varprojlim \left(\cO_{\cX}\otimes_A\cO\mathbb{A}_{\cris,Y}(V)/p^n\right),
\]
and
\[
\cO_{\cX}\widehat{\otimes}_{A}\Fil^r\cO\BB_{\cris,Y}(V):=\varinjlim_{n\in \mathbb N} t^{-n} \left(\cO_{\cX}\widehat{\otimes}_{A} \Fil^{r+n} \cO\mathbb{A}_{\cris,Y}(V)\right).
\]
In particular, if we filter $\cO_{\cX}\widehat{\otimes}_{A}\cO\BB_{\cris,Y}(V)$ using $\{\cO_{\cX}\widehat{\otimes}_A\Fil^r\cO\BB_{\cris,Y}(V)\}_{r\in \mathbb Z}$, the natural morphism
\[
\cO_{\cX}\widehat{\otimes}_A\cO\BB_{\cris,Y}(V)\lra Rw_{V*}(\cO\BB_{\cris,X})
\]
is an isomorphism in the filtered derived category.
\end{lemma}

\begin{proof} Recall first that, for $\cF$ a pro-\'etale sheaf on $X$ and for $j\geq 0$, $R^jw_{V*}\cF$ is the associated sheaf on $\cX_{\et}$ of the presheaf sending $\cU\in \cX_{\et}$ to $
H^i(\cU_V,\cF)$, where $\cU_V:=\cU_k\times_X X_V$. Take $\mathcal U=\Spf(B) \in \cX_{\et}$ to be affine such that the composition $\cU\to\cX\ra \cY$ can be factored as
\[
\cU\lra \mathcal T\lra \cY,
\]
where the first morphism is \'etale and $\mathcal T:=\Spf(A\{T_1^{\pm 1},\ldots, T_d^{\pm 1}\})$. Write $\mathcal{T}_V=\mathcal{T}_k\times_Y V$. Then $\mathcal T_V=\Spa(S,S^+)$ with $S^+=R\{T_1^{\pm 1},\ldots, T_d^{\pm 1}\}$ and $S=S^+[1/p]$. Write  $\mathcal U_V=\Spa(\widetilde{S},\widetilde S^+)$. For each $1\leq i\leq d$, let $(T_i^{1/p^{n}})_{n\in\mathbb N}$ be a compatible family of $p$-power roots of $T_i$. Set
\[
S_{\infty}^+:=R^+\{T_1^{\pm 1/p^{\infty}},\ldots, T_d^{\pm 1/p^{\infty}}\}, \quad \widetilde{S}_{\infty}^+:=B\widehat{\otimes}_{A\{T_1^{\pm 1},\ldots, T_d^{\pm 1}\}}S_{\infty}^+,
\]
$S_{\infty}:=S_{\infty}^+[1/p]$ and $\widetilde{S}_{\infty}:=\widetilde{S}_{\infty}^+[1/p]$. Then $(S_{\infty},S_{\infty}^+)$ and $(\widetilde{S}_{\infty},\widetilde{S}_{\infty}^+)$ are affinoid perfectoid algebras over $(\widehat{\bk},\cO_{\widehat{\bk}})$. Let $\widetilde{\mathcal U_V}\in X_{\proet}$ (resp. $\widetilde{\mathcal T_V}\in \mathcal{T}_{k\ \!\proet}$) be the affinoid perfectoid corresponding to $(\widetilde{S}_{\infty},\widetilde{S}_{\infty}^+)$ (resp. to $(S_{\infty},S_{\infty}^+)$). We have the following commutative diagram of ringed spaces
\[
\xymatrix{\widehat{\widetilde{\mathcal U_V}}=\Spa(\widetilde{S}_{\infty},\widetilde{S}_{\infty}^+)\ar[r]^{\Gamma} \ar[d]& \widehat{\mathcal U_V}=\Spa(\widetilde{S},\widetilde{S}^+)\ar[d]\ar[r] & \cU=\Spf(B)\ar[d]\\ \widehat{\widetilde{\mathcal{T}_V} }=\Spa(S_{\infty},S_{\infty}^+)\ar[r] & \widehat{\mathcal T_V}=\Spa(S,S^+)\ar[r]\ar[d] & \mathcal T\ar[d]\\ & \widehat{V}=\Spa(R,R^+)\ar[r] & \mathcal Y=\Spf(A).}
\]
The morphism $\widetilde{\cU_V}\to \cU_V$ is a profinite Galois cover, with Galois group $\Gamma\simeq \Zp(1)^d$. For $q\in \mathbb N$, let $\widetilde{\cU_V}^{q}$ be the $(q+1)$-fold fiber product of $\widetilde{\cU_{V}}$ over $\cU_{V}$. So $\widetilde{\cU_{V}}^{q}\simeq \widetilde{\cU_{V}}\times \Gamma^{q}$ is an affinoid perfectoid.

(1) As in the proof of Lemma \ref{twodcris}, there is a natural isomorphism of $B_{\cris}$-modules
\[
H^q(\Gamma, \cO\mathbb A_{\cris}(\widetilde{S}_{\infty},\widetilde{S}_{\infty}^+))[1/t]\stackrel{\sim}{\lra}H^q(\cU_V, \cO\mathbb B_{\cris, X}),
\]
where the first group is the continuous group cohomology and $\cO\mathbb A_{\cris}(\widetilde{S}_{\infty},\widetilde{S}_{\infty}^+)$ is endowed with the $p$-adic topology.  So, by Theorem \ref{withoutfil}, $H^q(\cU_V,\cO\BB_{\cris,X})=0$ whenever $q>0$, and there is a natural  isomorphism
\begin{equation}\label{eq:Galois-proet-is-the-same}
B\widehat{\otimes}_A\cO\BBcr(R,R^+) \stackrel{\sim}{\lra}H^0(\cU_V,\cO\BB_{\cris,X}).
\end{equation}
On the other hand, $V$ being affinoid perfectoid with $\widehat{V}=\Spa(R,R^+)$, the maps
\[
\cO\mathbb A_{\cris}(R,R^+)/p^n\lra \cO\mathbb A_{\cris, Y}(V)/p^n, \quad n\in \mathbb N,
\]
and thus the maps
\[
B\otimes_A \cO\mathbb A_{\cris}(R,R^+)/p^n\lra B\otimes_A \cO\mathbb A_{\cris,Y}(V)/p^n, \quad n\in \mathbb N,
\]
are almost isomorphisms (Lemma \ref{2obcris}). Passing to projective limits, it follows that the kernel and the cokernel of the induced map
\[
B\widehat{\otimes}_A \cO\mathbb A_{\cris}(R,R^+)\lra B\widehat{\otimes}_A \cO\mathbb A_{\cris,Y}(V)
\]
are killed by $\mathcal I^2$, hence also by $t^2$. Inverting $t$, we deduce $B\widehat{\otimes}_A \cO\mathbb B_{\cris}(R,R^+)\stackrel{\sim}{\ra}B\widehat{\otimes}_A \cO\mathbb B_{\cris,Y}(V)$, and an isomorphism from \eqref{eq:Galois-proet-is-the-same}
\[
\cO_{\cX}(\cU)\widehat{\otimes}_A \cO\mathbb B_{\cris,Y}(V)=B\widehat{\otimes}_A\cO\mathbb B_{\cris,Y}(V)\stackrel{\sim}{\lra} H^0(\cU_V, \cO\mathbb B_{\cris,X}).
\]

To conclude the proof of (1), it remains to check that the canonical morphism 
\begin{equation*}
\cO_{\cX}(\cU)\widehat{\otimes}_A\cO\mathbb{B}_{\cris, Y}(V)\lra \left(\cO_{\cX}\widehat{\otimes}_A\cO\mathbb{B}_{\cris,Y}(V)\right)(\cU)
\end{equation*}
is an isomorphism. In fact, we have
\begin{eqnarray*}
\cO_{\cX}(\cU)\widehat{\otimes}_A \cO\mathbb{A}_{\cris,Y}(V) & = & \varprojlim_n \cO_{\cX}(\cU)\otimes_A \cO\mathbb{A}_{\cris,Y}(V)/p^n \\ & \stackrel{\sim}{\lra} & \varprojlim_{n}\left(\left(\cO_{\cX}\otimes_A \cO\mathbb{A}_{\cris,Y}(V)/p^n\right)(\cU) \right)\\ & \stackrel{\sim}{\lra} & \left(\cO_{\cX}\widehat{\otimes}_A \cO\mathbb{A}_{\cris,Y}(V)\right)(\cU).
\end{eqnarray*}
Therefore, as $\cU$ is quasi-compact and quasi-separated, we find
\begin{eqnarray*}
\cO_{\cX}(\cU)\widehat{\otimes}_A \cO\mathbb{B}_{\cris,Y}(V) & = & \left(\cO_{\cX}(\cU)\widehat{\otimes}_A \cO\mathbb{A}_{\cris,Y}(V)\right)[1/t] \\ & \stackrel{\sim}{\lra} & \left((\cO_{\cX}\widehat{\otimes}_A \cO\mathbb{A}_{\cris,Y}(V))(\cU)\right)[1/t] \\ & \stackrel{\sim}{\lra} & \left( \cO_{\cX}\widehat{\otimes}_A \cO\mathbb{A}_{\cris,Y}(V)[1/t]\right)(\cU) \\ & = & \left(\cO_{\cX}\widehat{\otimes}_A\cO\BB_{\cris,Y}(V)\right)(\cU),
\end{eqnarray*}
as required.

(2) We shall only prove (2) when $r=0$: the general case can be deduced by twisting. As in (1), there exists a natural isomorphism
\begin{equation*}%\label{eq:Hq-Fil-OAcris-galois-proet}
\varinjlim_{s\geq 0}H^q(\Gamma, \Fil^s\cO\mathbb A_{\cris}(\widetilde{S}_{\infty},\widetilde{S}_{\infty}^+))\lra H^q(\cU_{V},\Fil^0\cO\mathbb B_{\cris,X}).
\end{equation*}
By definition, the first group is $H^q(\Gamma, \Fil^0\cO\mathbb B_{\cris}(\widetilde{S}_{\infty},\widetilde{S}_{\infty}^+))$ computed in \S~\ref{acy}. So, according to Proposition \ref{prop.higher-cohomology-of-Fil} and Corollary \ref{cor.invariant-of-Fil}, $H^q(\cU_V, \Fil^0\cO\mathbb B_{\cris,X})=0$ if $q>0$, and we have an isomorphism
\[
B\widehat{\otimes}\Fil^0\cO\mathbb B_{\cris}(R,R^+):=\varinjlim_{s\geq 0}B\widehat{\otimes}\Fil^s\cO\mathbb A_{\cris}(R,R^+)\stackrel{\sim}{\lra} H^0(\cU_V, \Fil^0\cO\mathbb B_{\cris,X}).
\]
To go further, one has to identify $B\widehat{\otimes}\Fil^0\cO\mathbb B_{\cris}(R,R^+)$ with $B\widehat{\otimes}\Fil^0\cO\mathbb B_{\cris}(V):=\varinjlim_{s
\geq 0}B\widehat{\otimes}\Fil^s\cO\mathbb A_{\cris}(V)$. As $V$ is an affinoid perfectoid pro-\'etale over $\widetilde{Y}_{\bk}$, by Lemma \ref{lem.technical-I} and Lemma \ref{2obcris}, the kernels and cokernels of the natural maps
\[
\Fil^s\cO\mathbb A_{\cris}(R,R^+)\lra \Fil^s\cO\mathbb A_{\cris,Y}(V), \quad s\in \mathbb N,
\]
are killed by $\mathcal I^2$, and $\mathcal I^3\cdot H^1(V,\Fil^s\cO\mathbb A_{\cris,Y})=0$. In particular, the kernels and cokernels of
\[
\mathrm{gr}^s\cO\mathbb A_{\cris}(R,R^+)\lra \mathrm{gr}^s\cO\mathbb A_{\cris,Y}(V), \quad s\in \mathbb N,
\]
are killed by some power of $\mathcal I(A_{\cris}/\ker(\theta))=\mathcal I\cO_{\mathbb C_p}$, thus by $p^{1/p^N}$ for every $N\in \mathbb N$. Moreover, we have the following commutative diagram with exact rows
\[
\xymatrix{0\ar[r] & B\widehat{\otimes}_A \Fil^{s+1}\cO\mathbb{A}_{\cris}(R,R^+)\ar[r]\ar[d] & B\widehat{\otimes}_A \Fil^s\cO\mathbb{A}_{\cris}(R,R^+)\ar[r]\ar[d] & B\widehat{\otimes}_A \mathrm{gr}^s\cO\mathbb{A}_{\cris}(R,R^+)\ar[r]\ar[d]^{\textrm{iso. up to }p\textrm{-torsion}}&  0 \\0\ar[r] & B\widehat{\otimes}_A \Fil^{s+1}\cO\mathbb{A}_{\cris,Y}(V)\ar[r] & B\widehat{\otimes}_A \Fil^s\cO\mathbb{A}_{\cris,Y}(V)\ar[r] & B\widehat{\otimes}_A \mathrm{gr}^s\cO\mathbb{A}_{\cris,Y}(V).& }
\]
By the observations above, the kernels and cokernels of the first two vertical maps are killed by some power of $\mathcal I$, thus by some power of $t$, and the last vertical map becomes an isomorphism after inverting $p$. So, by a similar argument as in the proof of Lemma \ref{cor.acyclicityBcris} (2), we find $B\widehat{\otimes}_A \Fil^0\cO\mathbb B_{\cris}(R,R^+)\stackrel{\sim}{\ra}B\widehat{\otimes}_A\Fil^0\cO\mathbb B_{\cris,Y}(V)$. We get finally a natural isomorphism
\[
B\widehat{\otimes}_A\Fil^0\cO\mathbb B_{\cris,Y}(V)\stackrel{\sim}{\lra} H^0(\cU_V, \Fil^0\cO\mathbb B_{\cris,X}).
\]
The remaining part of (2) can be done similarly as in the last part of the proof of (1), so we omit the details here.
\end{proof}

From now on, assume $f\colon \cX\to \cY$ is a proper smooth morphism (between smooth formal schemes) over $\cO_k$. Its closed fiber gives rise to a morphism between the crystalline topoi,
\[
f_{\cris}\colon (\cX_0/\cO_k)_{\cris}^{\sim}\lra (\cY_0/\cO_k)_{\cris}^{\sim}.
\]
Let $\cE$ be a filtered convergent $F$-isocrystal on $\cX_0/\cO_k$, and $\mathcal M$ an $F$-crystal on $\cX_0/\cO_k$ such that $\cE\simeq \mathcal M^{\rm an}(n)$ for some $n\in \mathbb N$ (cf. Remark \ref{rk.CrystalVSIsoc}). Then $\cM$ can be viewed naturally as a coherent $\cO_{\cX}$-module endowed with an integrable and quasi-nilpotent $\cO_k$-linear connection $\cM\to \cM\otimes \Omega^1_{\cX/\cO_k}$.

In the following we consider the higher direct image $R^{i}f_{\cris*}\mathcal M$ of the crystal $\mathcal M$. One can determine the value of this abelian sheaf on $\cY_0/\cO_k$ at the $p$-adic PD-thickening $\cY_0\hookrightarrow \cY$ in terms of the relative de Rham complex $DR_{X/Y}(\mathcal M)$ of $\mathcal M$. To state this, take $\mathcal{V}=\Spf(A)$ an affine open subset of $\cY$, and put $\cX_A:=f^{-1}(\mathcal V)$. We consider $A$ as a PD-ring with the canonical divided power structure on $(p)\subset A$. In particular, we can consider the crystalline site $(\cX_{A,0}/A)_{\cris}$ of $\cX_{A,0}:=\cX\times_{\cY}\mathcal V_0$ relative to $A$. By \cite[Lemme 3.2.2]{Ber}, the latter can be identified naturally to the open subsite of $(\cX_0/\cO_k)_{\cris}$ whose objects are objets $(U,T)$ of $(\cX_0/\cO_k)_{\cris}$ such that $f(U)\subset \mathcal V_0$ and such that there exists a morphism $\alpha\colon T\to \mathcal V_n:=\mathcal V\otimes_A A/p^{n+1}$ for some $n\in \mathbb N$, making the square below commute
\[
\xymatrix{U\ar@{^(->}[r]\ar[d]_{\rm can} & T\ar[d]^{\alpha} \\ \mathcal{V}_0\ar@{^(->}[r] & \mathcal{V}_n.}
\]
Using \cite[Corollaire 3.2.3]{Ber} and a limit argument, one finds
\[
R^if_{\cris*}(\cM)(\mathcal{V}_0,\mathcal{V})\stackrel{\sim}{\lra} H^i_{\cris}(\cX_{A,0}/A,\cM)
\]
where we denote again by $\cM$ the restriction of $\cM$ to $(\cX_{A,0}/A)_{\cris}$. Let
\[
u=u_{\cX_{A,0}/A}: (\cX_{A,0}/A)_{\cris}^{\sim}\lra \cX_{A~ \! \et}^{\sim}
\]
be the morphism of topoi such that $u_{\ast}(\mathcal F)(\cU)=H^0_{\cris}(\cU_0/A,\mathcal F)$ for $\cU\in \cX_{A~\!\et}$. By \cite[Theorem 7.23]{BO}, there is a natural quasi-isomorphism $Ru_{\ast} \cM\stackrel{\sim}{\ra} DR_{X/Y}(\cM)$ in the derived category, inducing an isomorphism
\[
H^i_{\cris}(\cX_{A,0}/A,\cM)\simto H^i(\cX_{A},DR_{X/Y}(\cM)).
\]
Passing to associated sheaves, we deduce $
R^if_{\cris *}(\cM)_{\cY}\stackrel{\sim}{\ra}R^if_{*}(DR_{X/Y}(\cM))$. On the other hand, as $f\colon \cX\to \cY$ is proper and smooth, $R^if_{*}(DR_{X/Y}(\cE))$, viewed as a coherent sheaf on the adic space $Y$, is the $i$-th relative convergent cohomology of $\cE$ with respect to the morphism $f_0\colon \cX_0\to \cY_0$. Thus, by \cite[Th\'eor\`eme 5]{Ber86} (see also \cite[Theorem 4.1.4]{Tsu}), if we invert $p$, the $\cO_{\cY}[1/p]$-module $R^if_{*}(DR_{X/Y}(\cE))\simeq R^if_{*}(DR_{X/Y}(\cM))[1/p]$, together with the Gauss-Manin connection and the natural Frobenius structure inherited from $R^if_{\cris*}(\cM)_{\cY}\simeq R^if_{*}(DR_{X/Y}(\cM))$, is a convergent $F$-isocrystal on $\cY_0/\cO_k$, denoted by $R^if_{\cris *}(\cE)$ in the following (this is an abuse of notation, a more appropriate notation should be $R^if_{0~ \! \mathrm{conv}*}(\cE)$). Using the filtration on $\cE$, one sees that $R^if_{\cris *}(\cE)$ has naturally a filtration, and it is well-known that this filtration satisfies Griffiths transversality with respect to the Gauss-Manin connection.

\begin{prop} \label{main1relative} Let $\cX\to \cY$ be a proper smooth morphism of smooth $p$-adic formal schemes over $\cO_k$. Let $\mathcal E$ be a filtered convergent $F$-isocrystal on $\cX_0/\cO_k$ and $\LL$ a lisse $\widehat{\ZZ}_p$-sheaf on $X_{\proet}$. Assume that $\mathcal E$ and $\LL$ are associated.
\begin{enumerate}
\item For every $i\in \mathbb Z$, $R^if_{\cris *}\cE$ is a filtered convergent $F$-isocrystal on $\cY_0/\cO_k$, with gradeds $R^if_*(\mathrm{gr}^{r}DR_{X/Y}(\cE))$, $r\in \mathbb Z$.
\item There is a natural filtered isomorphism of $\cO\BB_{\cris,Y}$-modules
\begin{equation}\label{eq.isorelative}
R^if_{k*}(\LL\otimes_{\widehat{\ZZ}_p}\mathbb A_{\cris,X}\widehat{\otimes}_{f_k^{-1}\mathbb A_{\cris,Y}}f_k^{-1}\cO\mathbb A_{\cris,Y}[1/t]) \stackrel{\sim}{\longrightarrow} w_{\cY}^{-1}(R^if_{\cris *}( \cE))\otimes\cO\BB_{\cris,Y}
\end{equation}
which is compatible with Frobenius and connection.
\end{enumerate}
\end{prop}

\begin{proof} (1) We have observed above that $R^if_{\cris*}\cE$ is naturally a convergent $F$-isocrystal. To complete the proof of (1), it suffices to check that the filtration on $R^if_{\cris*}\cE$ is given by locally direct summands. By Proposition \ref{prop.crisdr}, the lisse $\widehat{\Z}_p$-sheaf $\LL$ is de Rham with associated filtered $\cO_X$-module with integrable connection $\cE$. Therefore the Hodge-to-de Rham spectral sequence
\[
E_1^{i,j}=R^{i+j}f_{*}\left(\mathrm{gr}^{i}\left(DR_{X/Y}(\cE)\right)\right)\Longrightarrow R^{i+j}f_{*}(DR_{X/Y}(\cE))
\]
degenerates at $E_1$. Moreover, $E_{1}^{i,j}$, the relative Hodge cohomology of $\cE$ in \cite[Theorem 8.8]{Sch}, is a locally free $\cO_Y$-module of finite rank for all $i,j$ by \emph{loc. cit}. Thereby the filtration on $R^if_{*}(DR_{X/Y}(\cE))=R^if_{\cris*}(\cE)$, which is the same as the one induced by the spectral sequence above, is given by locally direct summands, with gradeds $R^{i}f_{*}(\mathrm{gr}^{r}(DR_{X/Y}(\cE)))$, $r\in \mathbb Z$.

(2) Using Proposition \ref{relativepoincare} (1) and the fact that $\LL$ and $\cE$ are associated, we have the following filtered isomorphisms compatible with connection:
\begin{equation}\label{eq.isosrelatifs}
\begin{array}{cl}
& R^if_{k*}(\LL\otimes \mathbb A_{\cris,X}\widehat{\otimes}f_k^{-1}\cO\mathbb A_{\cris,Y}[1/t]) \\ \stackrel{\sim}{\lra} & R^if_{k*}(\LL\otimes DR_{X/Y}(\cO\BB_{\cris,X})) \\ \stackrel{\sim}{\lra} & R^if_{k*}(DR_{X/Y}(\LL\otimes \cO\BB_{\cris,X})) \\ \stackrel{\sim}{\lra} & R^if_{k*}(DR_{X/Y}(w_{\cX}^{-1}\cE\otimes \cO\BB_{\cris,X})). %\\ \stackrel{\sim}{\lra} & R^if_{k*}(w_{\cX}^{-1}DR_{X/Y}(\cE)\otimes\cO\BB_{\cris,X}).
\end{array}
\end{equation}
On the other hand, the morphism below given by adjunction respects the connections on both sides:
\begin{equation}\label{eq.toshowitisaniso}
w_{\cY}^{-1}R^if_{*}(DR_{X/Y}(\cE))\otimes \cO\BB_{\cris,Y}\lra R^if_{k*}(DR_{X/Y}(w_{\cX}^{-1}\cE\otimes\cO\BB_{\cris,X})).
\end{equation}
We claim that \eqref{eq.toshowitisaniso} is a filtered isomorphism. This is a local question, we may and do assume that $\cY=\Spf(A)$ is affine and is \'etale over some torus over $\cO_k$. Let $V\in Y_{\proet}$ be an affinoid perfectoid pro-\'etale over $\widetilde{Y}_{\bk}$. As $R^if_{*}(DR_{X/Y}(\cE))=R^if_{\cris*}(\cE)$ is a locally free $\cO_{\cY}[1/p]$-module on $\cY$, we have
\[
\left(w_{\cY}^{-1}R^if_{*}(DR_{X/Y}(\cE))\otimes \cO\BB_{\cris,Y}\right)(V)\simeq H^i(\cX,DR_{X/Y}(\cE))\otimes_A\cO\BB_{\cris,Y}(V).
\]
So we only need to check that the natural morphism below is a filtered isomorphism
\[
H^i(\cX,DR_{X/Y}(\cE))\otimes_A\cO\BB_{\cris,Y}(V)\lra H^i(X_V,DR_{X/Y}(w_{\cX}^{-1}\cE\otimes\cO\BB_{\cris,X})).
\]
By Lemma \ref{quasirelative}, one has further identifications strictly compatible with filtrations:
\begin{eqnarray*}
H^i(X_V,DR_{X/Y}(w_{\cX}^{-1}\cE\otimes \cO\BB_{\cris,X}))&\simeq & H^i(\cX, Rw_{V*}(DR_{X/Y}(w_{\cX}^{-1}\cE\otimes \cO\BB_{\cris,X})))\\ & \simeq &  H^i(\cX, DR_{X/Y}(\cE\otimes \cO_{\cX}\widehat{\otimes}_A\cO\BB_{\cris,Y}(V)).
\end{eqnarray*}
Write $\widehat{V}=\Spa(R,R^+)$. So $\cO\BB_{\cris}(R,R^+)\stackrel{\sim}{\ra}\cO\BB_{\cris,Y}(V)$ by Corollary \ref{cor.acyclicityOBcris}. Thus, to prove our claim, it suffices to show that the canonical morphism
\[
H^i(\cX,DR_{X/Y}(\cE))\otimes_A\cO\BB_{\cris}^+(R,R^+)\lra H^i(\cX, DR_{X/Y}(\cE\otimes \cO_{\cX}\widehat{\otimes}_A\cO\BB_{\cris}^+(R,R^+)))
\]
is a filtered isomorphism. One only needs to check this on the gradeds. Since the gradeds of $\cO\mathbb B^+_{\cris}(R,R^+)$ are finite free $R$-modules, we are reduced to showing that for every $r\in \mathbb Z$, the natural map
\begin{equation}\label{eq:technical-base-change-OBcris}
H^i(\cX,\mathrm{gr}^rDR_{X/Y}(\cE))\otimes_A R\lra H^i(\cX, \mathrm{gr}^r(DR_{X/Y}(\cE))\otimes \cO_{\cX}\widehat{\otimes}_AR)
\end{equation}
is an isomorphism. This follows from Proposition \ref{prop:technical-base-change} by taking $B=R^+$. More precisely, let $\cF$ be a bounded complex of coherent sheaves on $\cX$, such that $\cF[1/p]=\mathrm{gr}^rDR_{X/Y}(\cE)$. So, for each term $\cF^i$ of $\cF$, $\cF^i[1/p]$ is locally a direct factor of a finite free $\cO_{\cX}[1/p]$-module. In particular, the kernel and cokernel of the natural map $\cF\otimes \cO_{\cX}\widehat{\otimes}_A R^+\ra \cF\widehat{\otimes}_{A}R^+$ is killed by some bounded power of $p$. So
\begin{equation}\label{eq:iso-for-cF-tensor-cO}
H^i(\cX,\cF\otimes \cO_{\cX}\widehat{\otimes}_A R^+)[1/p]\stackrel{\sim}{\lra} H^i(\cX,\cF\widehat{\otimes}_{A}R^+)[1/p].
\end{equation}
Moreover, by (1), $R^if_{*}\mathrm{gr}^rDR_{X/Y}(\cE)$ is a graded piece of the filtered isocrystal $R^if_{\cris*}(\cE)=R^if_{*}(DR_{X/Y}(\cE))$. As $\cY=\Spf(A)$ is affine, taking global sections, we see that $H^i(\cX,DR_{X/Y}(\cE))$ is projective over $A[1/p]$, and $H^i(\cX,\mathrm{gr}^rDR_{X/Y}(\cE))$ is locally a direct factor of $H^{i}(\cX,DR_{X/Y}(\cE))$. Therefore, $H^i(\cX,\mathrm{gr}^r(DR_{X/Y}(\cE)))=H^i(\cX,\cF)[1/p]$ is flat over $A[1/p]$. By Proposition \ref{prop:technical-base-change}, the kernel and the cokernel of the map
\[
H^i(\cX,\cF)\otimes_A R^+\lra H^i(\cX,\cF\widehat{\otimes}_A R^+)
\]
are killed by some power of $p$. Inverting $p$ and combining \eqref{eq:iso-for-cF-tensor-cO}, we obtain that  \eqref{eq:technical-base-change-OBcris} is an isomorphism, completing the proof of our claim.

Composing the isomorphisms in \eqref{eq.isosrelatifs} with the inverse of \eqref{eq.toshowitisaniso}, we get the desired filtered isomorphism \eqref{eq.isorelative} that is compatible with connections on both sides. It remains to check the Frobenius compatibility of \eqref{eq.isorelative}. For this, we may and do assume again that $\cY=\Spf(A)$ is affine and is \'etale over some torus over $\cO_k$, and let $V\in Y_{\proet}$ some affinoid perfectoid pro-\'etale over $\widetilde{Y}$. In particular, $A$ admits a lifting of the Frobenius on $\cY_0$, denoted by $\sigma$.  Let $\mathcal M$ be an $F$-crystal on $\cX_0/\cO_k$ such that $\cE=\mathcal M^{\rm an}(n)$ for some $n\in \mathbb N$ (Remark \ref{rk.CrystalVSIsoc}).  Then the crystalline cohomology $H^i_{\cris}(\cX_0/A,\mathcal M)$ is endowed with a Frobenius which is $\sigma$-semilinear. We just need to check the Frobenius compatibility of composition of the maps below (here the last one is induced by the inverse of \eqref{eq.isorelative}):
\[
\begin{array}{c}
H_{\cris}^i(\cX_0/A,\mathcal M)\lra H_{\cris}^i(\cX_0/A,\mathcal M)[1/p]\stackrel{\sim}{\lra} H^i(\cX,DR_{X/Y}(\cE))\lra \\ \left(w_{\cY}^{-1}(R^if_{\cris*}(\cE))\otimes \cO\BB_{\cris,Y}\right)(V)  \lra H^i(X_V, \LL\otimes \mathbb A_{\cris,X}\widehat{\otimes}f_k^{-1}\cO\mathbb A_{\cris,Y}[1/t]),
\end{array}
\]
which can be done in the same way as in the proof of Theorem \ref{main1}.
\end{proof}

The relative crystalline comparison theorem then can be stated as follows:

\begin{thm}\label{thm.relativecomp} Let $\LL$ be a crystalline lisse  $\widehat{\ZZ}_p$-sheaf on $X$ associated to a filtered $F$-isocrystal  $\cE$ on $\cX_0/\cO_k$. Assume that, for any $i\in \mathbb Z$, $R^if_{k*}\LL$ is a lisse  $\widehat{\ZZ}_p$-sheaf on $Y$. Then $R^if_{k*}\LL$ is crystalline and  is associated to the filtered convergent $F$-isocrystal $R^if_{\cris\ast}\cE$.
\end{thm}

\begin{proof} We have seen in Proposition \ref{main1relative} (1) that $R^if_{\cris *}\cE$ is a filtered convergent $F$-isocrystal on $\cY_0/\cO_k$. To complete the proof, we need to find filtered isomorphisms that are compatible with Frobenius and connections:
\[
R^if_{k*}(\LL)\otimes \cO\BB_{\cris,Y}\stackrel{\sim}{\lra} w_{\cY}^{-1}R^if_{\cris*}(\cE)\otimes \cO\BB_{\cris,Y}, \quad i\in \mathbb Z.
\]
For this, it suffices to combine Proposition \ref{part1relative} and Proposition \ref{main1relative} (2).
\end{proof}

\section{Appendix: geometric acyclicity of $\cO\BBcr$}\label{acy}

In this section, we extend the main results of \cite{AB} to the setting of perfectoids. The generalization is rather straightforward. Although one may see here certain difference from the arguments in \cite{AB}, the strategy and technique are entirely theirs.

Let $f\colon \cX=\Spf(B)\to \cY=\Spf(A)$ be a smooth morphism of smooth affine formal schemes over $\cO_k$. Write $X$ (resp. $Y$) for the generic fiber of $\cX$ (resp. of $\cY$). By abuse of notation the morphism $X\to Y$ induced from $f$ is still denoted by $f$.

Assume that $\cY$ is \'etale over the torus $\mathcal S:=\Spf(\cO_k\{S_1^{\pm 1}, \ldots, S_{\delta}^{\pm 1}\})$ defined over $\cO_k$ and that the morphism $f\colon \cX\to \cY$ can factor as
\[
\cX \stackrel{\textrm{\'etale}}{\lra} \mathcal{T}\lra \cY,
\]
where  $\mathcal T=\Spf(C)$ is a torus over $\cY$ and  the first morphism $\cX\to \mathcal T$ is \'etale.

Write $C=A\{T_1^{\pm 1},\ldots, T_d^{\pm 1}\}$. For each $1\leq i\leq d$ (resp. each $1\leq j\leq \delta$), let $\{T_i^{1/p^n}\}_{n\in \mathbb N}$ (resp. $\{S_j^{1/p^n}\}_{n\in \mathbb N}$) be a compatible family of $p$-power roots of $T_i$ (resp. of $S_j$). As in Proposition \ref{iso}, we denote by $\widetilde Y$ the following fiber product over the generic fiber $\mathcal S_k$ of $\mathcal S$:
\[
\widetilde Y=Y\times_{\mathcal S_k}\Spa(k\{S_1^{\pm 1/p^{\infty}},\ldots, S_{\delta}^{\pm 1/p^{\infty}}\},\cO_k\{S_1^{\pm 1/p^{\infty}},\ldots, S_{\delta}^{\pm 1/p^{\infty}}\}).
\]
Let $V\in Y_{\proet}$ be an affinoid perfectoid over $\widetilde{Y}_{\bk}$ with $\widehat{V}=\Spa(R,R^+)$. Let $T_V=\Spa(S,S^+)$ be the base change $\cT_k\times_Y V$ and $X_V=\Spa(\widetilde{S},\widetilde{S}^+)$ the base change $X\times_Y V$. Thus $S^+=R^+\{T_1^{\pm 1}, \ldots, T_d^{\pm 1}\}$ and $S=S^+[1/p]$. Set
\[
S_{\infty}^+=R^+\left\{T_1^{\pm 1/p^{\infty}},\cdots, T_d^{\pm 1/p^{\infty}}\right\}, \quad \widetilde{S}_{\infty}^{+}:=B\widehat{\otimes}_C S_{\infty}^+,
\]
$S_{\infty}:=S_{\infty}^+[1/p]$ and $\widetilde{S}_{\infty}:=\widetilde{S}_{\infty}^+[1/p]$. Then $(S_{\infty},S_{\infty}^+)$ and $(\widetilde{S}_{\infty},\widetilde{S}_{\infty}^+)$ are affinoid perfectoids and
\[
S^{\flat+}_{\infty}=R^{\flat+}\left\{(T_1^{\flat})^{\pm 1/p^{\infty}},\cdots, (T_d^{\flat})^{\pm 1/p^{\infty}}\right\},
\]
where $T_i^{\flat}:=(T_i, T_i^{1/p}, T_{i}^{1/p^2}, \ldots )\in S_{\infty}^{\flat +}$. The inclusions $S^+\subset S_{\infty}^+$ and $\widetilde S^+\subset \widetilde{S}_{\infty}^+$ define two profinite Galois covers. Their Galois groups are the same, denoted by $\Gamma$, which is a profinite group isomorphic to $\mathbb Z_p(1)^d$. One can summarize these notations in the following commutative diagramme
\[
\xymatrix{\widetilde{S}_{\infty}^+ & \widetilde{S}^+\ar[l]_{\Gamma} & B\ar[l] \\ S_{\infty}^+\ar[u] & S^+\ar[u]\ar[l]_{\Gamma} & C\ar[u]_{\textrm{\'etale}}\ar[l] \\ R^+\{T_1^{\pm 1/p^{\infty}},\ldots, T_d^{\pm 1/p^{\infty}}\}\ar@{=}[u] & R^+\{T_1^{\pm 1},\ldots, T_d^{\pm 1}\}\ar@{=}[u]\ar[l]_{\Gamma} & A\{T_1^{\pm 1},\ldots, T_d^{\pm 1}\}\ar@{=}[u]\ar[l] \\ & R^+\ar[u]& A\ar[u]\ar[l]}
\]
The group $\Gamma$ acts naturally on the period ring $\cO\BBcr(\widetilde{S}_{\infty},\widetilde{S}_{\infty}^+)$ and on its filtration $\Fil^r\cO\BBcr(\widetilde{S}_{\infty},\widetilde{S}_{\infty}^+)$. The aim of this appendix is to compute the group cohomology
\[
H^q\left(\Gamma, \cO\BBcr\left(\widetilde{S}_{\infty},\widetilde{S}_{\infty}^+\right)\right):=H^q_{\rm cont}\left(\Gamma, \cO\AAcr\left(\widetilde{S}_{\infty},\widetilde{S}^+_{\infty}\right)\right)[1/t]
\]
and
\[
H^q\left(\Gamma, \Fil^r\cO\BBcr\left(\widetilde{S}_{\infty},\widetilde{S}_{\infty}^+\right)\right):=\varinjlim_{n\geq |r|} H_{\rm cont}^q\left(\Gamma, \frac{1}{t^n}\Fil^{r+n}\cO\AAcr\left(\widetilde{S}_{\infty},\widetilde{S}_{\infty}^+\right)\right)
\]
for $q,r\in \mathbb Z$.

In the following, we will omit systematically the subscript {\textquotedblleft cont\textquotedblright} whenever there is no confusion arising. Moreover, we shall use multi-indice to simplify the notation:  for example, for $\underline{a}=(a_1,\ldots, a_d)\in \mathbb Z[1/p]^d$, $T^{\underline{a}}:=T_1^{a_1}\cdot T_2^{a_2}\cdots T_d^{a_d}$.

\subsection{Cohomology of $\cO\BBcr$}

We will first compute $
H^q(\Gamma, \cO\AAcr(S_{\infty},S_{\infty}^+)/p^n)$ up to $(1-[\epsilon])^{\infty}$-torsion for all $q,n\in \mathbb N$ (Corollary \ref{cor.cohomologyofobcris1}). From these computations, we deduce from the results about the cohomology groups $H^q(\Gamma, \cO\BBcr(\widetilde{S}_{\infty},\widetilde{S}_{\infty}^+))$, $q\in \mathbb Z$ (Theorem \ref{withoutfil}).

\begin{lemma}\label{lem.isotechnique}
For $n\in \Z_{\geq 1}$, there are natural isomorphisms
\[
\AAcr(R,R^+)/p^n\otimes_{W(R^{\flat+})/p^n}W(S_{\infty}^{\flat+})/p^n\stackrel{\sim}{\lra} \AAcr(S_{\infty},S_{\infty}^+)/p^n
\]
and
\[
\left(\AAcr(S_{\infty},S_{\infty}^+)/p^n \otimes \cO\AAcr(R,R^+)/p^n\right)\langle u_1,\ldots, u_d\rangle \stackrel{\sim}{\lra} \cO\AAcr(S_{\infty},S_{\infty}^+)/p^n,
\]
sending $u_i$ to $T_i-[T_i^{\flat}]$. Here the tensor product in the second isomorphism above is taken over $\AAcr(R,R^+)/p^n$. Moreover, the natural morphisms
\[
\AAcr(R,R^+)/p^n\to \AAcr(S_{\infty},S_{\infty}^+)/p^n, \quad \cO\AAcr(R,R^+)/p^n\to \cO\AAcr(S_{\infty},S_{\infty}^+)/p^n
\]
are both injective.

\end{lemma}

\begin{proof}
 Recall $\xi=[p^{\flat}]-p$. We know that $\AAcr(S_{\infty},S_{\infty}^+)$ is the $p$-adic completion of
\[
\AAcr^{0}(S_{\infty},S_{\infty}^+):=W(S_{\infty}^{\flat+})\left[\frac{\xi^m}{m!}|m=0,1,\cdots\right]=\frac{W(S_{\infty}^{\flat+})[X_0,X_1,\cdots]}{(m! X_m-\xi^m: m\in \mathbb Z_{\geq 0})}.
\]
Note that we have the same expression with $R$ in place of $S_{\infty}$. We then have
\begin{eqnarray*}
\AAcr^{0}(S_{\infty},S_{\infty}^+) &\stackrel{\sim}{\longleftarrow} & W(S_{\infty}^{\flat+})\otimes_{W(R^{\flat+})}\frac{W(R^{\flat+})[X_0,X_1,\cdots]}{(m!X_m-\xi^m|m\in \Z_{\geq 0})}\\ & =& W(S_{\infty}^{\flat+})\otimes_{W(R^{\flat+})}\AAcr^{0}(R,R^+).
\end{eqnarray*}
The first isomorphism follows.

Secondly, as $V$ lies above $\widetilde Y_{\bk}$, by Proposition \ref{iso} we have
\[
\AAcr(R,R^+)\{\langle w_1,\ldots, w_{\delta}\rangle \} \stackrel{\sim}{\lra}\cO\AAcr(R,R^+), \quad w_j\mapsto S_j-[S_j^{\flat}]
\]
where $S_j^{\flat}:=(S_j,S_j^{1/p},S_j^{1/p^2},\ldots)\in R^{\flat+}$. Similarly, we have
\[
\AAcr(S_{\infty},S_{\infty}^+)\{\langle u_1,\ldots, u_d,w_1,\ldots, w_{\delta}\rangle\} \stackrel{\sim}{\lra} \cO\AAcr(S_{\infty},S_{\infty}^+), \quad u_i \mapsto T_i-[T_i^{\flat}], \ w_j\mapsto S_j-[S_j^{\flat}].
\]
Thus (the isomorphisms below are all the natural ones)
\begin{eqnarray*}
\frac{\cO\AAcr(S_{\infty},S_{\infty}^+)}{p^n} & \stackrel{\sim}{\longleftarrow} & \left(\frac{\AAcr(S_{\infty},S_{\infty}^+)}{p^n}\right)\langle u_1,\ldots, u_d,w_1,\ldots, w_{\delta}\rangle \\ & \stackrel{\sim}{\longleftarrow} & \frac{\AAcr (S_{\infty},S_{\infty}^+)}{p^n} \otimes_{\frac{\AAcr(R,R^+)}{p^n}} \left(\frac{\AAcr(R,R^+)}{p^n}\langle u_1,\ldots, u_d,w_1,\ldots, w_{\delta}\rangle\right) \\ & \stackrel{\sim}{\lra}& \frac{\AAcr(S_{\infty},S_{\infty}^+)}{p^n}\otimes_{\frac{\AAcr(R,R^+)}{p^n}}\left(\frac{\cO\AAcr(R,R^+)}{p^n}\langle u_1,\ldots, u_d\rangle\right) \\
& \stackrel{\sim}{\lra} & \left(\frac{\AAcr(S_{\infty},S_{\infty}^+)}{p^n}\otimes_{\frac{\AAcr(R,R^+)}{p^n}}\frac{\cO\AAcr(R,R^+)}{p^n}\right)\langle u_1,\ldots, u_d\rangle.
 \end{eqnarray*}
So our second isomorphism is obtained.

Next we prove that the natural morphism $\AAcr(R,R^+)/p^n\to \AAcr(S_{\infty},S_{\infty}^+)/p^n$ is injective.  When $n=1$, we are reduced to showing the injectivity of
 \[
\frac{(R^{\flat+}/(p^{\flat})^p)[X_1,X_2,\ldots ]}{(X_1^p,X_2^p,\ldots)} \lra \frac{(S^{\flat+}/(p^{\flat})^p)[X_1,X_2,\ldots ]}{(X_1^p,X_2^p,\ldots)},
\]
or equivalently the injectivity of
\[
R^{\flat+}/(p^{\flat})^p\to S^{\flat+}/(p^{\flat})^p=\left(R^{\flat+}/(p^{\flat})^p\right)\left[(T_1^{\flat})^{\pm 1/p^{\infty}},\ldots (T_d^{\flat})^{\pm 1/p^{\infty}}\right],
\]
which is clear. The general case follows easily since $\AAcr(S_{\infty},S_{\infty}^+)$ is $p$-torsion free. One deduces also the injectivity of $\cO\AAcr(R,^+)/p^n\to \cO\AAcr(S_{\infty},S_{\infty}^+)/p^n$ by using the natural isomorphisms
$
\cO\AAcr(R,R^+)/p^n\simeq (\AAcr(R,R^+)/p^n)\langle w_1,\ldots, w_{\delta}\rangle$, and $\cO\AAcr(S_{\infty},S_{\infty}^+)/p^n\simeq (\AAcr(S_{\infty},S_{\infty}^+)/p^n)\langle u_1,\ldots, u_d,w_1,\ldots, w_{\delta}\rangle$. This concludes the proof of our lemma.
\end{proof}

\begin{prop} \label{prop.freeness}
$ \AAcr(S_{\infty},S_{\infty}^+)/p^n$ is free over $\AAcr(R,R^+)/p^n$ with a basis given by $\{[T^{\flat}]^{\underline a}|\underline{a}\in \Z[1/p]^d\}$.
\end{prop}

 \begin{proof}
By Lemma \ref{lem.isotechnique}, $\AAcr(S_{\infty},S_{\infty}^+)/p^n$ is generated over $\AAcr(R,R^+)/p^n$ by elements of the form $[x]$ with $x\in S_{\infty}^{\flat+}=R^{\flat+}\{(T_1^{\flat})^{\pm 1/p^{\infty}},\ldots, (T_d^{\flat})^{\pm 1/p^{\infty}}\}$. Write $B_n\subset \AAcr(S_{\infty},S_{\infty}^+)/p^n$ for the $\AAcr(R,R^+)/p^n$-submodule generated by elements of the form $[x]$ with $x\in \mathsf{S}:=R^{\flat+}\left[(T_1^{\flat})^{\pm 1/p^{\infty}},\ldots, (T_d^{\flat})^{\pm 1/p^{\infty}}\right]\subset S_{\infty}^{\flat +}$. We claim that $B_n=\AAcr(S_{\infty},S_{\infty}^+)/p^n$.

Since $S_{\infty}^{\flat +}$ is the $p^{\flat}$-adic completion of $\mathsf{S}$, for each $x\in S^{\flat+}$ we can write $x=y_0+p^{\flat}x'$ with $x'\in \mathsf{S}$. Iteration yields
\[
x=y_0+p^{\flat}y_1+\cdots +(p^{\flat})^{p-1}y_{p-1}+(p^{\flat})^px''
\]
with $y_i\in \mathsf{S}$ and $x''\in S_{\infty}^{\flat +}$. Then in $W(S_{\infty}^{\flat +})$:
 \begin{eqnarray*}
[x]&\equiv & [y_0]+[p^{\flat}][y_1]+\cdots +[(p^{\flat})^{p-1}][y_{p-1}]+[(p^{\flat})^p][x'']  \quad \textrm{mod} \quad pW(S_{\infty}^{\flat+})\\
&\equiv & [y_0]+\xi [y_1]+\cdots +\xi^{p-1}[y_{p-1}]+\xi^p[x'']  \quad  \textrm{mod} \quad pW(S_{\infty}^{\flat+}).
\end{eqnarray*}
As $\xi\in \mathbb A_{\cris}(S_{\infty},S_{\infty}^+)$ has divided power, $\xi^p=p!\cdot \xi^{[p]}\in p\AAcr(S_{\infty},S_{\infty}^+)$. So we obtain  in $\mathbb A_{\cris}(S_{\infty},S_{\infty}^+)$
\[
[x]\equiv [y_0]+\xi [y_1]+\cdots +\xi^{p-1}[y_{p-1}] \quad  \textrm{mod} \quad p\AAcr(S_{\infty},S_{\infty}^+).
\]
 For any $\alpha\in \AAcr(S_{\infty},S_{\infty}^+)/p^n=\AAcr(R,R^+)/p^n\otimes_{W(R^{\flat+})/p^n}W(S_{\infty}^{\flat+})/p^n$, we may write
 \[
 \alpha=\sum_{i=0}^m\lambda_i[x_i]+p\alpha', \quad x_i\in S_{\infty}^{\flat+},\lambda_i\in \AAcr(R,R^+)/p^n, \alpha'\in \AAcr(S_{\infty},S_{\infty}^+)/p^n.
 \]
 The observation above tells us that one can write
 \[
 \alpha=\beta_0+p\alpha'',\quad \beta_0\in B_n, \alpha''\in \AAcr(S_{\infty},S_{\infty}^+)/p^n.
 \]
By iteration again, we find
\[
\alpha=\beta_0+p\beta_1+\cdots p^{n-1}\beta_{n-1}+p^n\tilde{\alpha}, \quad \beta_0,\cdots, \beta_{n-1}\in B_n, \tilde{\alpha}\in \AAcr(S_{\infty},S_{\infty}^+)/p^n.
\]
Thus
\[
\alpha=\beta_0+p\beta_1+\cdots p^{n-1}\beta_{n-1}\in B_n\subset \AAcr(S_{\infty},S_{\infty}^+)/p^n.
 \]
This shows the claim, i.e.  $\AAcr(S_{\infty},S_{\infty}^+)/p^n$ is generated over $\AAcr(R,R^+)/p^n$ by the elements of the form $[x]$ with $x\in \mathsf S=R^{\flat +}[(T_1^{\flat})^{\pm 1/p^{\infty}},\ldots, (T_d^{\flat})^{\pm 1/p^{\infty}}]\subset S_{\infty}^{\flat+}$. Furthermore, as for any $x,y\in S_{\infty}^{\flat +}$
\[
[x+y]\equiv [x]+[y] \quad \mathrm{mod} \quad pW(S_{\infty}^{\flat +}),
\]
a similar argument shows that $\AAcr(S_{\infty},S_{\infty}^+)/p^n$ is generated over $\AAcr(R,R^+)/p^n$ by the family of elements  $\{[T^{\flat}]^{\underline a}|\underline{a}\in \Z[1/p]^d\}$.

It remains to show  the freeness of the family $\{[T^{\flat}]^{\underline a}|\underline{a}\in \Z[1/p]^d\}$ over $\AAcr(R,R^+)/p^n$. For this, suppose there exist $\lambda_1,\cdots, \lambda_m\in \AAcr(R,R^+)$ and distinct elements $\underline{a}_1,\cdots,\underline{a}_m\in \Z[1/p]^d$ such that
 \[
 \sum_{i=1}^m\lambda_i[T^{\flat}]^{\underline{a}_i}\in p^n\AAcr(S_{\infty},S_{\infty}^+).
\]
One needs to prove $\lambda_i\in p^n \mathbb A_{\cris}(R,R^+)$ for each $i$. Modulo $p$ we find
\[
\sum_{i=1}^m\overline{\lambda_i} \cdot (T^{\flat})^{\underline{a}_i}=0 \quad \textrm{in }\AAcr(S_{\infty},S_{\infty}^+)/p,
\]
with $\overline{\lambda_i}\in \AAcr(R,R^+)/p$ the reduction modulo $p$ of $\lambda_i$.
On the other hand, the family  of elements $\{(T^{\flat})^{\underline{a}}:\underline{a}\in \mathbb Z[1/p]^d\}$ in
\begin{eqnarray*}
\AAcr(S_{\infty},S_{\infty}^+)/p & \simeq  & \frac{S^{\flat+}/((p^{\flat})^p)[\delta_2,\delta_3,\ldots]}{(\delta_2^p,\delta_3^p,\ldots )} \\ & \simeq & \frac{R^{\flat +}/((p^{\flat})^p)[(T_1^{\flat})^{\pm1/p^{\infty}},\ldots, (T_d^{\flat})^{\pm 1/p^{\infty}}, \delta_2, \delta_3,\ldots ]}{(\delta_2^p,\delta_3^p,\ldots )}
\end{eqnarray*}
is free over $\AAcr(R,R^+)/p\simeq \frac{R^{\flat+}/((p^{\flat})^p)[\delta_2,\delta_3,\ldots]}{(\delta_2^p,\delta_3^p,\ldots )}$. Therefore, $\overline{\lambda_i}=0$, or equivalently, $\lambda_i=p\lambda_i'$ for some $\lambda_i'\in  \AAcr(R,R^+)$. In particular,
 \[
 \sum_{i=1}^{m} \lambda_i[T^{\flat}]^{\underline{a}_i}=
p\cdot \left (\sum_{i=1}^m\lambda_i'[T^{\flat}]^{\underline{a}_i}\right)\in p^n\AAcr(S_{\infty},S_{\infty}^+).
\]
 But $\AAcr(S_{\infty},S_{\infty}^+)$ is $p$-torsion free, which implies that
\[
\sum_{i=1}^m\lambda_i'[T^{\flat}]^{\underline{a}_i}\in p^{n-1}\AAcr(S_{\infty},S_{\infty}^+).
\]
This way, we may find $\lambda_i=p^n\tilde{\lambda}_i$ for some $\tilde{\lambda}_i\in  \AAcr(R,R^+)$, which concludes the proof of the freeness.
 \end{proof}

Recall that $\Gamma$ is the Galois group of the profinite cover $(S_{\infty},S_{\infty}^+)$ of $(S,S^+)$. Let $\epsilon=(\epsilon^{(0)},\epsilon^{(1)},\ldots)\in \cO_{\widehat{\bk}}^{\flat}$ be a compatible system of $p$-power roots of unity such that $\epsilon^{(0)}=1$ and that $\epsilon^{(1)}\neq 1$. Let $\{\gamma_1,\ldots, \gamma_d\}$ be a family of generators such that for each $1\leq i\leq d$, $\gamma_i$ acts trivially on the variables $T_j$ for any index $j$ different from $i$ and that $\gamma_i(T_i^{\flat})=\epsilon T_i^{\flat}$.

\begin{lemma} Let $1\leq i\leq d$ be an integer. Then one has $\gamma_i([T_i^{\flat}]^{p^n})=[T_i^{\flat}]^{p^n}$ in $\cO\AAcr(S_{\infty},S_{\infty}^+)/p^n$.
\end{lemma}

\begin{proof} By definition, $\gamma_i([T_i^{\flat}]^{p^n})=[\epsilon]^{p^n}[T_i^{\flat}]^{p^n}$ in $\cO\AAcr(S_{\infty},S_{\infty}^+)$. So our lemma follows from the fact that $[\epsilon]^{p^n}-1=\exp(p^n t)-1=\sum_{r\geq 1}p^{nr} t^{[r]}\in p^nA_{\cris}$.
\end{proof}

Let $A_n$ be the $\cO\AAcr(R,R^+)$-subalgebra of $\cO\AAcr(S_{\infty},S_{\infty}^+)/p^n$ generated by $[T_i^{\flat}]^{\pm p^n}$ for $1\leq i\leq d$. The previous lemma shows that $\Gamma$ acts trivially on $A_n$. Furthermore, by the second isomorphism of Lemma \ref{lem.isotechnique} and by Proposition \ref{prop.freeness}, we have
\[
\frac{\cO\AAcr(S_{\infty},S_{\infty}^+)}{p^n} \stackrel{\sim}{\lra} \left(\bigoplus_{\underline{a}\in \mathbb Z[1/p]^d\cap [0,p^n)^d} A_{n}[T^{\flat}]^{\underline a}\right) \langle u_1,\ldots, u_d\rangle,\quad T_i-[T_i^{\flat}]\mapsto u_i.
\]
Transport the Galois action of $\Gamma$ on $\cO\mathbb A_{\cris}(S_{\infty},S_{\infty}^+)/p^n$ to the right-hand side of this isomorphism. It follows that
\[
\gamma_i(u_i)=u_i+(1-[\epsilon])[T_i^{\flat}].
\]
Therefore,
\[
\gamma_i(u_i^{[n]})=u_i^{[n]}+\sum_{j=1}^{n}[T_i^{\flat}]^j(1-[\epsilon])^{[j]}u_i^{[n-j]}.
\]
For other index $j\neq i$, $\gamma_i(u_j)=u_j$ and hence $\gamma_i(u_j^{[n]})=u_j^{[n]}$ for any $n$. Set
\[
\mathbf{X}_{n}:=\bigoplus_{\underline{a}\in (\mathbb Z[1/p]\cap [0,p^n))^d \setminus \mathbb Z^d} A_n[T^{\flat}]^{\underline a}, \quad \textrm{and}\quad \mathbf A_n:=\bigoplus_{\underline{a}\in \mathbb Z^d\cap [0,p^n)^d}A_n[T^{\flat}]^{\underline{a}}.
\]
Then we have the following decomposition, which respects  the $\Gamma$-actions:
 \[
 \cO\AAcr(S_{\infty},S_{\infty}^+)/p^n= \mathbf{X}_n\langle u_1,\ldots, u_d\rangle\oplus \mathbf A_n\langle u_1,\ldots, u_d\rangle.
 \]

The following result can be proven similarly  as \cite[Proposition 16]{AB}.
\begin{prop} \label{hqxnu}For any $q\in \mathbb N$, $H^{q}(\Gamma,\mathbf X_n\langle u_1,\ldots, u_d\rangle)$ is killed by $(1-[\epsilon]^{1/p})^2$.
\end{prop}

The computation of $H^q(\Gamma, \mathbf A_n\langle u_1,\ldots, u_d\rangle )$ is more subtle. Note that we have the following decomposition
\[
\mathbf A_n\langle u_1,\ldots, u_d\rangle =\bigotimes_{i=1}^d (\cO\AAcr(R,R^+)/p^n)\left[[T_i^{\flat}]^{\pm 1}\right]\langle u_i\rangle,
\]
where the tensor products above are taken over $\cO\AAcr(R,R^+)/p^n$.

We shall first treat the case where $d=1$. We set $T:=T_1$, $u:=u_1$ and $\gamma:=\gamma_1$. Let $\mathcal A_n^{(m)}$ be the $A_n$-submodule of $\cO\AAcr(S_{\infty},S_{\infty}^+)/p^n$ generated by the $u^{[m+a]}/[T^{\flat}]^a$'s with $m+a\geq 0$ and $0\leq a<p^n$. Then
\[
\mathbf{A}_n\langle u\rangle =A_n\left[[T^{\flat}]\right]\langle u\rangle =\sum_{m> -p^n} \mathcal A_n^{(m)}.
\]
Consider the following complex:
\begin{equation}\label{eq.An<U>}
A_n \left[[T^{\flat}]\right]\langle u\rangle \stackrel{\gamma_-1}{\longrightarrow} A_n\left[[T^{\flat}]\right]\langle u\rangle,
\end{equation}
which computes $H^q(\Gamma, \mathbf A_n\langle u\rangle )=H^q\left(\Gamma,A_n\left[[T^{\flat}]\right]\langle u\rangle \right)$.

Again, the following lemma can be proven in a completely analogous way as is done in the proof of \cite[Proposition 20]{AB}.
\begin{lemma} The cokernel of \eqref{eq.An<U>}, and hence $H^q\left(\Gamma, \mathbf A_n\langle u\rangle\right)$ for any $q>0$, are killed by $1-[\epsilon]$.
\end{lemma}

One still needs to compute $H^0\left(\Gamma, \mathbf A_n\langle u\rangle \right)$. Note first that we have the following isomorphism
\[
\left(\cO\AAcr(R,R^+)/p^n\right)\left[T^{\pm 1}\right]\langle u\rangle \simto \mathbf{A}_n\langle u\rangle=(\cO\AAcr(R,R^+)/p^n) \left[[T^{\flat}]\right]\langle u\rangle
\]
sending $T$ to $u+[T^{\flat}]$. Endow an action of $\Gamma$ on $(\cO\AAcr(R,R^+)/p^n)[T^{\pm 1}]\langle u\rangle$ via the isomorphism above. So $H^0(\Gamma,\mathbf A_n\langle u\rangle )$ is naturally isomorphic to the kernel of the morphism
\begin{equation*}
\gamma-1\colon (\cO\AAcr (R,R^+)/p^n)[T^{\pm 1}]\langle u\rangle \longrightarrow (\cO\AAcr(R,R^+)/p^n)[T^{\pm 1}]\langle u\rangle,
\end{equation*}
and there is a natural injection
\[
C\otimes_{A} \cO\AAcr (R,R^+)/p^n=(\cO\AAcr(R,R^+)/p^n)[T^{\pm 1}] \hra H^0(\Gamma,\mathbf A_n\langle u\rangle).
\]

The proof of \cite[Lemme 29]{AB} applies to this case. We have

\begin{lemma} The cokernel of the last map is killed by $1-[\epsilon]$.
\end{lemma}

Now we are ready to prove

\begin{prop} \label{hqanu}
For any $d>0$, $n>0$ and $q>0$, $H^q\left(\Gamma, \mathbf A_n\langle u_1,\ldots, u_d\rangle \right)$ is killed by $(1-[\epsilon])^{2d-1}$. Moreover, the natural morphism
\begin{equation}\label{eq.injmor}
C\otimes_A\cO\AAcr(R,R^+)/p^n \lra H^0(\Gamma, \mathbf A_n\langle u_1,\ldots, u_d\rangle ), \quad T_i\mapsto u_i+[T_i^{\flat}]
\end{equation}
is injective with cokernel killed by $(1-[\epsilon])^{2d-1}$.
\end{prop}
\begin{proof} Recall that we have the  decomposition
\[
\mathbf A_n\langle u_1,\ldots, u_d\rangle =\bigotimes_{i=1}^d (\cO\AAcr(R,R^+)/p^n)\left[[T_i^{\flat}]^{\pm 1}\right]\langle u_i\rangle.
\]
We shall proceed by induction on $d$. The case $d=1$ comes from the previous two lemmas. For integer $d>1$, one uses Hochschild-Serre spectral sequence
\[
E_{2}^{i,j}=H^i(\Gamma/\Gamma_1,H^j(\Gamma_1, \mathbf{A}_n\langle u_1,\ldots, u_d\rangle))\Longrightarrow H^{i+j}(\Gamma, \mathbf A_n\langle u_1,\ldots, u_d\rangle ).
\]
Using the decomposition above, the group $H^j(\Gamma_1, \mathbf{A}_n\langle u_1,\ldots, u_d\rangle)$ is isomorphic to
\begin{eqnarray*}
H^j\left(\Gamma_1,\cO\AAcr(R,R^+)/p^n\left[[T_1^{\flat}]^{\pm 1}\right]\langle u_1\rangle \right)\otimes \left(\otimes_{i=2}^{d}(\cO\AAcr(R,R^+)/p^n)\left[[T_i^{\flat}]^{\pm 1}\right]\langle u_i\rangle \right). \end{eqnarray*}
So by the calculation done for the case $d=1$, we find that, up to $(1-[\epsilon])$-torsion, $H^j(\Gamma_1, A_n\langle u_1,\ldots, u_d\rangle)$ is zero when $j>0$, and is equal to
\[
(\cO\AAcr(R,R+)/p^n)[T_1^{\pm 1}]\otimes \left(\otimes_{i=2}^{d}(\cO\AAcr(R,R^+)/p^n)\left[[T_i^{\flat}]^{\pm 1} \right] \langle u_i\rangle\right)
\]
when $j=0$.
Thus, up to $(1-[\epsilon])$-torsion, $E_{2}^{i,j}=0$ when $j>0$ and $E_{2}^{i,0}$ is equal to
\[
(\cO\AAcr(R,R+)/p^n)[T_1^{\pm 1}]\otimes H^i\left(\Gamma/\Gamma_1,\left(\otimes_{i=2}^{d}(\cO\AAcr(R,R^+)/p^n)\left[[T_i^{\flat}]^{\pm 1}\right]\langle u_i\rangle \right)\right).
\]
Using the induction hypothesis, we get that, up to $(1-[\epsilon])^{2(d-1)-1+1}$-torsion, $E_{2}^{i,0}=0$ when $i>0$ and
\[
E_{2}^{0,0}=(\cO\AAcr(R,R^+)/p^n)[T_1^{\pm 1},\ldots, T_d^{\pm 1}]=C\otimes_A \cO\AAcr(R,R^+)/p^n.
\]
As $E_2^{i,j}=0$ for $j>1$, we have short exact sequence
\[
0\lra E_{\infty}^{q,0}\lra H^q(\Gamma, \mathbf A_n\langle u_1,\ldots, u_d\rangle)\lra E_{\infty}^{q-1,1}\lra 0.
\]
By what we have shown above, $E_{\infty}^{q-1,1}$ is killed by $(1-[\epsilon])$ (as this is already the case for $E_2^{q-1,1}$), and $E_{\infty}^{q,0}$ is killed by $(1-[\epsilon])^{2d-2}$ for $q>0$ (as this is the case for $E_2^{q,0}$), thus $H^q(\Gamma, \mathbf A_n\langle u_1,\ldots, u_d\rangle)$ is killed by $(1-[\epsilon])^{2d-1}$.
For $q=0$, $H^0(\Gamma, \mathbf A_n\langle u_1,\ldots, u_d\rangle)\simeq E_{\infty}^{0,0}=E_2^{0,0}$. So the cokernel of the natural injection
\[
C\otimes_A\cO\AAcr(R,R^+)/p^n\lra H^0(\Gamma, \mathbf A_n\langle u_1,\ldots, u_d\rangle)
\]
is killed by $(1-[\epsilon])^{2d-2}$, hence by $(1-[\epsilon])^{2d-1}$.
\end{proof}

\begin{rk} With more efforts, one may prove that $H^q(\Gamma, \mathbf{A}_n\langle u_1,\ldots, u_d\rangle)$ is killed by $(1-[\epsilon])^d$ for $q>0$ (\cite[Proposition 21]{AB}), and that the cokernel of the morphism \eqref{eq.injmor} is killed by $(1-[\epsilon])^2$ (\cite[Proposition 30]{AB}).
\end{rk}

\begin{cor}\label{cor.cohomologyofobcris1} For any $n\geq 0$ and any $q>0$, $H^q(\Gamma, \cO\AAcr(S_{\infty},S_{\infty}^+)/p^n)$ is killed by $(1-[\epsilon])^{2d}$. Moreover, the natural morphism
\[
C\widehat{\otimes}_A\cO\AAcr(R,R^+)/p^n\lra  H^0(\Gamma, \cO\AAcr(S_{\infty},S_{\infty}^+)/p^n)
\]
is injective, with cokernel killed by $(1-[\epsilon])^{2d}$.
\end{cor}

Recall that we want to compute $H^q(\Gamma, \cO\BBcr(\widetilde{S}_{\infty},\widetilde{S}_{\infty}^+))$. For this, one needs

\begin{lemma} Keep the notations above and assume that the morphism $f\colon \cX\to \cY$ is \'etale; thus $V=\Spa(R,R^+)$ and $X\times_YV=\Spa(\widetilde{S},\widetilde{S}^+)=\Spa(\widetilde{S}_{\infty},\widetilde{S}_{\infty}^+)$.
The natural morphism
\[
B\widehat{\otimes}_A\cO\AAcr(R,R^+)\lra \cO\AAcr(\widetilde{S}_{\infty},\widetilde{S}_{\infty}^+)
\]
is an isomorphism.
\end{lemma}

\begin{proof}By Lemma \ref{2obcris} we are reduced to showing that the natural map (here $w_j=S_j-[S_j^{\flat}]$)
\[
B\widehat{\otimes}_A \mathbb A_{\cris}(R,R^+)\{\langle w_1,\ldots, w_{\delta}\rangle \} \lra \mathbb A_{\cris}(\widetilde{S},\widetilde{S}^+)\{\langle w_1,\ldots, w_{\delta}\rangle\}
\]
is an isomorphism. Since both sides of the previous maps are $p$-adically complete and without $p$-torsion, we just need to check that its reduction modulo $p$
\[
B/p\otimes_{A/p}(\mathbb A_{\cris}(R,R^+)/p)\langle w_1,\ldots, w_{\delta}\rangle \lra (\mathbb A_{\cris}(\widetilde{S},\widetilde{S}^+)/p)\langle w_1,\ldots, w_{\delta}\rangle
\]
is an isomorphism. Note that we have the following expression
\[
(\mathbb{A}_{\cris}(R,R^+)/p)\langle w_1,\ldots, w_{\delta}\rangle \simeq \frac{(R^{\flat +}/((p^{\flat})^p))[\delta_m,w_i, Z_{im}]_{1\leq i\leq \delta, m\in \mathbb N}}{(\delta_m^p,w_i^m, Z_{im}^p)_{1\leq i\leq \delta,m\in \mathbb N}},
\]
and the similar expression for $(\mathbb A_{\cris}(\widetilde{S},\widetilde{S}^+)/p)\langle w_1,\ldots, w_{\delta}\rangle$, where $\delta_m$ is the image of $\gamma^{m+1}(\xi)$ with $\gamma: x\mapsto x^p/p$.
One sees that both sides of the morphism above  are $p^{\flat}$-adically complete (in fact $p^{\flat}$ is nilpotent). Moreover $R^{\flat+}$ has no $p^{\flat}$-torsion. So we are reduced to showing that the morphism
\[
B/p\otimes_{A/p} \frac{(R^{\flat +}/p^{\flat})[\delta_m,w_i,Z_{im}]_{1\leq i\leq \delta,m\in \mathbb N}}{(\delta_m^p,w_i^p, Z_{im}^p)_{1\leq i\leq \delta,m\in \mathbb N}}\lra \frac{(\widetilde{S}^{\flat +}/p^{\flat})[\delta_m,w_i,Z_{im}]_{1\leq i\leq \delta,m\in \mathbb N}}{(\delta_m^p,w_i^p, Z_{im}^p)_{1\leq i\leq \delta,m\in \mathbb N}}
\]
is an isomorphism. But $R^{\flat +}/p^{\flat}\simeq R^{+}/p$ and $\widetilde{S}^{\flat +}/p^{\flat}\simeq \widetilde{S}^+/p$, so we just need to show that the following morphism is an isomorphism:
\[
\alpha:B/p\otimes_{A/p} \frac{(R^{+}/p)[\delta_m,w_i,Z_{im}]_{1\leq i\leq \delta,m\in \mathbb N}}{(\delta_m^p,w_i^p, Z_{im}^p)_{1\leq i\leq \delta,m\in \mathbb N}}\lra \frac{(\widetilde{S}^{+}/p)[\delta_m,w_i,Z_{im}]_{1\leq i\leq \delta,m\in \mathbb N}}{(\delta_m^p,w_i^p, Z_{im}^p)_{1\leq i\leq \delta,m\in \mathbb N}}.
\]
To see this, we consider the following diagram
\[
\xymatrix{&& \frac{\widetilde{S}^+/p[\delta_m,w_i,Z_{im}]_{1\leq i\leq \delta,m\in \mathbb N}}{(\delta_m^p,w_i^p, Z_{im}^p)_{1\leq i\leq \delta,m\in \mathbb N}} \\ B/p\ar@/^1pc/[urr]\ar[r] & B/p\otimes_{A/p} \frac{R^+/p[\delta_m,w_i,Z_{im}]_{1\leq i\leq \delta,m\in \mathbb N}}{(\delta_m^p,w_i^p, Z_{im}^p)_{1\leq i\leq \delta,m\in \mathbb N}}\ar[ur]^{\alpha}  & \\ A/p\ar[u]^{\textrm{\'etale}}\ar[r] & \frac{R^+/p[\delta_m,w_i,Z_{im}]_{1\leq i\leq \delta,m\in \mathbb N}}{(\delta_m^p,w_i^p, Z_{im}^p)_{1\leq i\leq \delta,m\in \mathbb N}} \ar@/_1pc/[uur]_{\textrm{\'etale}} \ar[u]^{\textrm{\'etale}}& }.
\]
It follows that $\alpha$ is \'etale. To see that $\alpha$ is an isomorphism, we just need to show that this is the case after modulo some nilpotent ideals of both sides of $\alpha$. Hence we are reduced to showing that the natural morphism
\[
B/p\otimes_{A/p} R^+/p \lra \widetilde{S}^+/p
\]
is an isomorphism, which is clear from the definition.
\end{proof}

Apply the previous lemma to the \'etale morphism $f\colon \cX\to \mathcal T$, we find a canonical $\Gamma$-equivariant isomorphism
\[
B\widehat{\otimes}_C\cO\AAcr(S_{\infty},S_{\infty}^+)\stackrel{\sim}{\lra} \cO\AAcr(\widetilde{S}_{\infty},\widetilde{S}_{\infty}^+).
\]
In particular, we find
\[
H^q(\Gamma, B\widehat{\otimes}_C\cO\AAcr(S_{\infty},S_{\infty}^+))\stackrel{\sim}{\lra} H^q(\Gamma, \cO\AAcr(\widetilde{S}_{\infty},\widetilde{S}_{\infty}^+)).
\]
Now consider the following spectral sequence
\[
E_{2}^{i,j}=R^i\varprojlim_n H^j(\Gamma, B\otimes_C\cO\AAcr(S_{\infty},S_{\infty}^+)/p^n) \Longrightarrow H^{i+j}(\Gamma, B\widehat{\otimes}_C\cO\AAcr(S_{\infty},S_{\infty}^+))
\]
which induces a short exact sequence for each $i$:
\begin{equation}\label{eq.sesforRlim}
\begin{array}{c}
0\lra R^1\varprojlim_n H^{i-1}(\Gamma, B\otimes_C\cO\AAcr(S_{\infty},S_{\infty})/p^n)\lra H^{i}(\Gamma, B\widehat{\otimes}_C\cO\AAcr(S_{\infty},S_{\infty}^+)) \\ \lra \varprojlim_n H^{i}(\Gamma, B\otimes_C\cO\AAcr(S_{\infty},S_{\infty})/p^n) \lra 0.
\end{array}
\end{equation}
As $B$ is flat over $C$, it can be written as a filtered limit of finite free $C$-modules by Lazard's theorem (\cite[Th\'eor\`eme 1.2]{Laz}), and as $\Gamma$ acts trivially on $B$, the following natural morphism is an isomorphism for each $i$:
\[
B\otimes_CH^{i}(\Gamma, \cO\AAcr(S_{\infty},S_{\infty})/p^n) \stackrel{\sim}{\lra} H^{i}(\Gamma, B\otimes_C\cO\AAcr(S_{\infty},S_{\infty})/p^n).
\]
Therefore, for $i\geq 1$, $H^{i}(\Gamma, B\otimes_C\cO\AAcr(S_{\infty},S_{\infty}^+)/p^n)$ is killed by $(1-[\epsilon])^{2d}$ by Corollary \ref{cor.cohomologyofobcris1}. Moreover, by the same corollary, we know that the following morphism is injective with cokernel killed by $(1-[\epsilon])^{2d}$:
\[
C\otimes_A \cO\AAcr(R,R^+)/p^n\lra H^0(\Gamma, \cO\AAcr(S_{\infty},S_{\infty}^+)/p^n).
\]
Thus the same holds if we apply the functor $B\otimes_C-$ to the morphism above
\[
B\otimes_A \cO\AAcr(R,R^+)/p^n\lra B\otimes_CH^0(\Gamma, \cO\AAcr(S_{\infty},S_{\infty}^+)/p^n).
\]
Passing to limits we obtain an injective morphism
\[
B\widehat{\otimes}_A \cO\AAcr(R,R^+)\lra \varprojlim_n \left(B\otimes_CH^0(\Gamma, \cO\AAcr(S_{\infty},S_{\infty}^+)/p^n)\right),
\]
whose cokernel is killed by $(1-[\epsilon])^{2d}$, and  that
\[
R^1\varprojlim_n \left(B\otimes_C H^0(\Gamma, \cO\AAcr(S_{\infty},S_{\infty}^+)/p^n)\right)
\]
is killed by $(1-[\epsilon])^{2d}$. As a result, using the short exact sequence \eqref{eq.sesforRlim}, we deduce that for $i\geq 1$, $H^i(\Gamma, B\widehat{\otimes}_C\cO\AAcr(S_{\infty},S_{\infty}^+))\simeq H^i(\Gamma, \cO\AAcr(\widetilde{S}_{\infty},\widetilde{S}_{\infty}^+))$ is killed by $(1-[\epsilon])^{4d}$, and that the canonical morphism
\[
B\widehat{\otimes}_A\cO\AAcr(R,R^+)\lra H^0(\Gamma, B\widehat{\otimes}_C\cO\AAcr(S_{\infty},S_{\infty}^+))\simeq H^0(\Gamma, \cO\AAcr(\widetilde{S}_{\infty},\widetilde{S}_{\infty}^+))
\]
is injective with cokernel killed by $(1-[\epsilon])^{2d}$. One can summarize the calculations above as follows:

\begin{prop}\label{cor.cohomologyofobcris} (i) For any $n\geq 0$ and $q>0$, $H^q(\Gamma, \cO\AAcr(\widetilde{S}_{\infty},\widetilde{S}_{\infty}^+)/p^n)$ is killed by $(1-[\epsilon])^{2d}$, and the natural morphism
\[
B\otimes_A\cO\AAcr(R,R^+)/p^n\lra H^0(\Gamma, \cO\AAcr(\widetilde{S}_{\infty},\widetilde{S}_{\infty}^+))/p^n
\]
is injective with cokernel killed by $(1-[\epsilon])^{2d}$.

(ii) For any  $q>0$, $H^q(\Gamma, \cO\AAcr(S_{\infty},S_{\infty}^+))$ is killed by $(1-[\epsilon])^{4d}$ and the natural morphism
\[
B\widehat{\otimes}_A\cO\AAcr(R,R^+)\to H^0(\Gamma, \cO\AAcr(\widetilde{S}_{\infty},\widetilde{S}_{\infty}^+))
\]
is injective, with cokernel killed by $(1-[\epsilon])^{2d}$.
\end{prop}

\begin{thm}\label{withoutfil} Keep the notations above. Then $H^q(\Gamma,\cO\BBcr(\widetilde{S}_{\infty},\widetilde{S}_{\infty}^+))=0$ for $q\geq 1$, and the natural morphism
\[
B\widehat{\otimes}_A\cO\BBcr(R,R^+)\lra H^0(\Gamma,\cO\BBcr(\widetilde{S}_{\infty},\widetilde{S}_{\infty}^+))
\]
is an isomorphism.
\end{thm}
\begin{proof} By the previous proposition, we just need to remark that to invert $1-[\epsilon]$ one just needs to invert $t$, as
\[
t=\log([\epsilon])=-\sum_{n\geq 1}(n-1)!\cdot (1-[\epsilon])^{[n]}=-(1-[\epsilon])\sum_{n\geq 1}(n-1)!\cdot \frac{(1-[\epsilon])^{[n]}}{1-[\epsilon]}.
\]
Here, by \cite[Lemme 18]{AB}, $\frac{(1-[\epsilon])^{[n]}}{1-[\epsilon]}\in \ker(A_{\cris}\stackrel{\theta}{\to} \widehat{\cO_{\bk}} )$, hence the last summation above converges in $A_{\cris}$.
\end{proof}

\subsection{Cohomology of $\Fil^r\cO\BBcr$}

With Proposition \ref{cor.cohomologyofobcris} in hand, the following results on the cohomology of $\Fil^r\cO\BBcr$ can be shown in exactly the same way as is done in \cite[\S 5]{AB}. We thus only state these results, and refer to \cite{AB} for the detailed proofs. %with only very brief indications on the strategy of the proofs.

%In the following, we sometimes write $\Fil^r\cO\AAcr =\Fil^r\cO\AAcr(S_{\infty},S_{\infty}^+))$ for simplicity.

\begin{lemma}[\cite{AB} Proposition 32]\label{lem.killedbylargepower} Let $q\in \mathbb N_{>0}$, and $n\in \mathbb Z_{\geq 4d+r}$. The $A_{\cris}$-module $H^q(\Gamma, \Fil^r\cO\AAcr(S_{\infty},S_{\infty}^+))$ is killed by $t^n$.
\end{lemma}

\begin{proof}  Let $\mathrm{gr}^r\cO\AAcr:=\Fil^r\cO\AAcr/\Fil^{r+1}\cO\AAcr$. As $\theta(1-[\epsilon])=0$, $\mathrm{gr}^r\cO\AAcr$ is killed by $1-[\epsilon]$. So using the tautological short exact sequence below
\[
0\lra \Fil^{r+1}\cO\AAcr\lra \Fil^r\cO\AAcr\lra \mathrm{gr}^r\cO\AAcr\lra 0
\]
and by induction on the integer $r\geq 0$, one shows that $H^q(\Gamma, \Fil^r\cO\AAcr)$ is killed by $(1-[\epsilon])^{4d+r}$; the $r=0$ case being Proposition \ref{cor.cohomologyofobcris}(ii). So the multiplication-by-$t^n$ with $n\geq 4d+r$ is zero for $H^q(\Gamma,\cO\AAcr)$.
\end{proof}

 Then, as in \cite[Proposition 34]{AB}, we have

\begin{prop}\label{prop.higher-cohomology-of-Fil} $H^q(\Gamma, \Fil^r\cO\BBcr(\widetilde{S}_{\infty},\widetilde{S}_{\infty}^+))=0$ for any $q>0$.
\end{prop}

\if false
\begin{proof} By \cite[Lemme 33]{AB} , the cohomology group $H^q(\Gamma, \Fil^r\cO\BBcr(\widetilde{S}_{\infty},\widetilde{S}_{\infty}^+))$ is a vector space over $\Qp$. Hence to show the desired annihilation, we just need to show that for any $x\in H^q(\Gamma, t^{-n}\Fil^{r+n}\AAcr)$, there exists $m\gg n$ such that the image of $x$ in $H^q(\Gamma, t^{-m}\Fil^{r+m}\cO\AAcr)$ is $p$-torsion. In view of Lemma \ref{lem.killedbylargepower}, we are reduced to showing that the kernel of the map
\[
H^q(\Gamma, \Fil^{r+1}\cO\AAcr)\to H^q(\Gamma,\Fil^r\cO\AAcr)
\]
is of $p$-torsion, or equivalently, the cokernel of the map
\[
H^{q-1}(\Gamma, \Fil^r\cO\AAcr)\to H^{q-1}(\Gamma, \mathrm{gr}^r\cO\AAcr)
\]
is of $p$-torsion for any $q\geq 1$. This assertion is verified in \cite[Lemmes 35, 36]{AB}.
\end{proof}
\fi

It remains to compute the $\Gamma$-invariants of $\Fil^r\cO\BBcr(S_{\infty},S_{\infty})$. We shall first show $H^0(\Gamma,\Fil^r\cO\AAcr(S_{\infty},S_{\infty}^{+}))=B\widehat{\otimes}_A \Fil^r\cO\AAcr(R,R^+)$ %. Recall that $\tilde{\xi}:=\frac{[\epsilon]-1}{[\epsilon]^{1/p}-1}$. Put
%\[
%\eta:=\frac{([\epsilon]-1)^{p-1}}{p}\in A_{\cris}[p^{-1}].
%\]
%For $x\in \ker(\theta\colon \cO\AAcr(\widetilde{S}_{\infty},\widetilde{S}_{\infty}^+)\to \widetilde{S}_{\infty}^+)$, set $\gamma(x):=x^p/p$. One checks that $\gamma(x)\in \cO\AAcr$ and $\theta(\gamma(x))=0$.
in the way of proof of \cite[Proposition 41]{AB}.%, we obtain
\begin{prop} \label{prop:invariant-Fil-OAcris}For each $r\in \mathbb N$, the natural injective map
\[
\iota_r\colon B\widehat{\otimes}_A\Fil^r\cO\AAcr(R,R^+)\lra H^0(\Gamma, \Fil^r\cO\AAcr(S_{\infty},S_{\infty}^+))
\]
is an isomorphism.
\end{prop}

\if false
\begin{proof} We shall begin with the case where $r=0$. By Proposition \ref{cor.cohomologyofobcris} (ii), the natural morphism $B\widehat{\otimes}_A\cO\AAcr(R,R^+)\to H^0(\Gamma, \cO\AAcr)$ is injective with cokernel killed by $(1-[\epsilon])^{2d}$. It remains to show that the latter map is also surjective. Let $x\in H^0(\Gamma, \cO\AAcr)$. So $(1-[\epsilon])^{2d}x\in B\widehat{\otimes}_A\cO\AAcr(R,R^+)$. In particular,
\[
(1-[\epsilon])^{2d(p-1)}x\in \left(B\widehat{\otimes}_A\cO\AAcr(R,R^+)\right)\bigcap (1-[\epsilon])^{2d(p-1)}\cO\AAcr.
\]
Recall from \cite[Corollaire 40]{AB}\label{cor.intersectioninobcris} that,  inside $\cO\AAcr$, we have
\[
(B\widehat{\otimes}_A \cO\AAcr(R,R^+))\bigcap ([\epsilon]-1)^{p-1}\cO\AAcr=([\epsilon]-1)^{p-1}(B\widehat{\otimes}_A\cO\AAcr(R,R^+)).
\]
Thus the last intersection is $(1-[\epsilon])^{2d(p-1)}(B\widehat{\otimes}_A\cO\AAcr(R,R^+))$. In particular, $x\in B\widehat{\otimes}_A\cO\AAcr(R,R^+)$: note that $1-[\epsilon]$ is a regular element. This concludes the proof for $r=0$. Assume now $r>0$ and that the statement holds for $\Fil^{r-1}\cO\AAcr$. Then we have the following commutative diagram with exact rows
\[
\xymatrix{0\ar[r] & B\widehat{\otimes}_A \Fil^r\cO\AAcr(R,R^+)\ar[r] \ar[d]^{\iota_r}& B\widehat{\otimes}_A \Fil^{r-1}\cO\AAcr(R,R^+)\ar[r]\ar[d]^{\iota_{r-1}} & B\widehat{\otimes}_A \mathrm{gr}^{r-1}\cO\AAcr(R,R^+)\ar[r]\ar@{^(->}[d] & 0 \\ 0\ar[r] & H^0(\Gamma, \Fil^{r}\cO\AAcr)\ar[r] & H^0(\Gamma, \Fil^{r-1}\cO\AAcr)\ar[r] & H^0(\Gamma, \mathrm{gr}^{r-1}\cO\AAcr). & }
\]
One checks that the last vertical morphism is injective as $B\widehat{\otimes}_A R^+=\widetilde{S}^+\subset \widetilde{S}_{\infty}^+$. So by snake lemma,  that $\iota_{r-1}$ is an isomorphism implies that $\iota_r$ is also an isomorphism. This finishes the induction and hence the proof of our proposition.
\end{proof}
\fi

\begin{cor}\label{cor.invariant-of-Fil} The natural morphism
\[
B\widehat{\otimes}_A \Fil^r\cO\BBcr(R,R^+) \lra H^0(\Gamma, \Fil^r\cO\BBcr(S_{\infty},S_{\infty}^+))
\]
is an isomorphism, where
\[
B\widehat{\otimes}_A\Fil^r\cO\BBcr(R,R^+):=\varinjlim_{n\geq 0} B\widehat{\otimes}_A \Fil^{r+n}\cO\AAcr(R,R^+)
\]
with transition maps given by multiplication by $t$.
\end{cor}

\section{Appendix: base change for cohomology of formal schemes}
The aim of this appendix is to establish a base change result for cohomology of formal schemes (Proposition \ref{prop:technical-base-change}) that is used in the proof of relative comparison theorem. %Before doing this, let us first observe the following two lemmas, which are well-known (at least when $B$ is moreover a noetherian $A$-algebra). We include the proofs here since we could not find a reference.

\begin{lemma}\label{lem:elementary-commutative-algebra} Let $A$ be a noetherian ring that is $p$-adically complete, i.e., the natural morphism $A\ra \varprojlim A/p^n$ is an isomorphism. Let $B$ be a flat $A$-module, and $F$ a finite $A$-module. Then, if $B$ is $p$-adically complete, so is $F\otimes_A B$.
%\begin{enumerate}
%\item Let $F$ be a finite $A$-module. Then, if $B$ is $p$-adically complete, so is $F\otimes_A B$.
%\item The $p$-adic completion $\hat{B}:=\varprojlim_n B/p^n$ of $B$ is flat over $A$.
%\end{enumerate}
\end{lemma}

\begin{proof} Let $B_n:=B/p^n$. So $B\stackrel{\sim}{\ra}\varprojlim_nB_n$. Since $B$ is $A$-flat, $B_n$ is flat over $A/p^n$. Moreover, the canonical map $B_{n+1}\ra B_n$ is clearly surjective. Therefore, by \cite[\href{https://stacks.math.columbia.edu/tag/0912}{Lemma 0912}]{Stacks}, for $F$ a finite $A$-module, we have $F\otimes B\simeq \varprojlim_n (F\otimes B_n)\simeq \varprojlim_n(F\otimes B)/p^n$. In other words, $F\otimes B$ is $p$-adically complete. 

\if false 
For $M$ an $A$-module, let us denote by $M[p^n]$ the $A$-submodule of elements killed by $p^n$. As $B$ is flat over $A$, we have $B[p^n]=A[p^n]\otimes_A B$, and $(F\otimes_A B)[p^n]=F[p^n]\otimes_A B$. Since $A$ is noetherian and $F$ is of finite type over $A$, there exists an integer $a\geq 0$ such that $A[p^n]=A[p^a]$, $B[p^n]=B[p^a]$, $F[p^n]=F[p^a]$ and $(F\otimes_A B)[p^n]=(F\otimes_A B)[p^a]$ for any $n\geq a$. On the other hand, write $F$ as the cokernel of an $A$-linear map $A^s\ra A^t$, from which we deduce an exact sequence
\[
B^s\stackrel{\alpha}{\lra} B^t \stackrel{\beta}{\lra} F\otimes_A B\lra 0.
\]
Let $K=\ker(\alpha)$. We claim that the induced topology on $K\subset B^s$ is the same as the $p$-adic topology. Clearly $p^nK\subset K\cap (p^nB^s)$. Conversely, let $x\in K\cap (p^nB^s)$. Let $y\in B^s$ with $x=p^n y$. Then $0=\alpha(x)=p^n\alpha(y)$. So $\alpha(y)\in B^t[p^n]=B^t[p^a]$. Thus $p^ay\in \ker(\alpha)=K$ and $x=p^ny=p^{n-a}(p^ay)\in p^{n-a}K$. Therefore,
\[
p^nK\subset K\cap (p^nB^s)\subset p^{n-a}K, \quad \forall \ n\geq a,
\]
giving our claim. By consequence, as $K$ is complete for the induced topology because $B^t$ is $p$-adically complete thus Hausdorff with respect to the $p$-adic topology (so $K\subset B^s$ is a closed subspace), $K$ is $p$-adically complete. Let $C=\mathrm{Im}(\alpha)\subset B^t$. So, by \cite[Corollary 10.3]{AM}, we deduce from the short exact sequence
\[
0\lra K\lra B^s\lra C\lra 0,
\]
a second short exact sequence
\[
0\lra \varprojlim K/p^n \lra \varprojlim B^s/p^n \lra \varprojlim C/p^n\lra 0,
\]
whence a commutative diagram with exact rows
\[
\xymatrix{0\ar[r] & \varprojlim K/p^n \ar[r] &  \varprojlim B^s/p^n \ar[r] &  \varprojlim C/p^n\ar[r] & 0 \\ 0\ar[r] & K\ar[r]\ar[u]^{\simeq} & B^s\ar[r]\ar[u]^{\simeq} & C\ar[r]\ar[u] & 0 .}
\]
Since $K\stackrel{\sim}{\ra}\varprojlim K/p^n$ and $B^s\stackrel{\sim}{\ra}\varprojlim B^s/p^n$, we find $C\stackrel{\sim}{\ra} \varprojlim C/p^n$. In other words, $C$ is $p$-adically complete.
Similarly, the subspace topology on $C\subset B^t$ is the same as the $p$-adic topology on $C$, so from the exact sequence
\[
0\lra C\lra B^t\lra F\otimes_ AB\lra 0,
\]
and the fact that $C$ and $B^t$ are $p$-adically complete, we deduce that $F\otimes_A B$ is $p$-adically complete, as claimed by our lemma.\fi
\if false 
(2) Let $I\subset A$ be an ideal. Consider the short exact sequence
\[
0\lra I\otimes_A B\lra B\lra (A/I)\otimes_A B \lra 0.
\]
A similar argument as in (1) shows that the sequence below is exact 
\[
0\lra \varprojlim (I\otimes_A B/p^n)\lra \varprojlim(B/p^n)\lra \varprojlim((A/I)\otimes_A B/p^n) \lra 0.
\]
As $\hat{B}/p^n=B/p^n$, $\hat{B}$ is $p$-adically complete and $\varprojlim (I\otimes_A B/p^n)=\varprojlim(I\otimes_A \hat{B}/p^n)$ is the $p$-adic completion of $I\otimes_A \hat{B}$. By (1), the latter is already $p$-adically complete. So the exact sequence above can be written as
\[
0\lra I\otimes_A \hat{B}\lra \hat{B}\lra (A/I)\otimes_A \hat{B}\lra 0.
\]
In particular, the natural map $I\otimes_A \hat B\ra \hat B$ is injective, and $\hat{B}$ is flat over $A$.\fi
\end{proof}

\begin{lemma}\label{lem:flat-base-change} Let $A$ be a $p$-adically complete notherian ring, and $B$ a flat $A$-module that is $p$-adically complete.  Let $\cX\to \Spf(A)$ be a proper morphism of $p$-adic formal schemes. Let $\mathcal F$ be a coherent sheaf on $\cX$. Then, for every $i\in \mathbb Z$, the natural morphism
\[
H^i(\cX,\mathcal F)\otimes_A B\longrightarrow H^i(\cX,\mathcal F\widehat{\otimes}_A B)
\]
is an isomorphism.

\end{lemma}

\begin{proof} If $\mathcal F$ is a coherent sheaf annihilated by some power of $p$, $\mathcal F\widehat{\otimes}_A B=\mathcal F\otimes_A B$. Then, our proposition follows from the standard flat base change result. Indeed, by a theorem of Lazard (\cite[Th\'eor\`eme 1.2]{Laz}), $B$ can be written as a filtered inductive limit of finite free $A$-modules. As $\cX$ is quasi-compact and quasi-separated, one only needs to prove our assertion when $B$ is finite and free over $A$. But in this case our assertion is obvious. In general, let $\mathcal G\subset \mathcal F$ be the subsheaf formed by elements killed by some power of $p$. As $\mathcal F$ is coherent and as $\cX$ is quasi-compact, $\mathcal G$ is killed by a large power of $p$. Therefore $\mathcal G$ and the quotient $\mathcal F/\mathcal G$ are coherent sheaves on $\cX$. Since $\cF/\cG$ is $p$-torsion free, the same holds for $(\cF/\cG)\otimes_A B$ as $B$ is flat over $A$. Using the tautological exact sequence
\begin{equation}\label{eq:ses-proj-systems}
0\lra \mathcal G\otimes_AB\lra \mathcal F\otimes_A B\lra \mathcal (F/\mathcal G)\otimes_A B\lra 0,
\end{equation}
we get a short exact sequence of projective systems:
\[
0\lra \left(\mathcal G\otimes_A B/p^{n}\right)_{n\geq 0}\longrightarrow \left(\mathcal F\otimes_A B/p^{n}\right)_{n\geq 0}\longrightarrow \left((\mathcal F/\mathcal G)\otimes_A B/p^{n}\right)_{n\geq 0}\lra 0.
\]
Because $p^{n}\mathcal G=0$ and thus $p^{n}(\mathcal G\otimes_A B)=0$ for large $n$, we find
\[
\mathcal G\otimes_A B/p^{n+1}\stackrel{\sim}{\lra} \mathcal G\otimes_A B/p^{n}, \quad \textrm{for }n\gg 0.
\]
So $R^1\varprojlim_n \left((\mathcal G\otimes_A B)/p^{n}\right)=0$.
Passing to projective limits in \eqref{eq:ses-proj-systems}, we obtain a short exact sequence
\[
0\lra \mathcal G\widehat{\otimes}_A B\lra \mathcal F\widehat{\otimes}_A B\lra (\mathcal F/\mathcal G)\widehat{\otimes}_AB\lra 0,
\]
from which we get a commutative diagram with exact rows
\[
\xymatrix{\cdots \ar[r]  & H^i(\cX, \mathcal G)\otimes B\ar[r]\ar[d]^{\simeq} & H^i(\cX, \mathcal F)\otimes B\ar[r]\ar[d]^{\rm can} & H^{i}(\cX, \mathcal F/\mathcal G)\otimes B\ar[r]\ar[d]^{\rm can} & \cdots \\ \cdots \ar[r]  & H^i(\cX, \mathcal G\widehat{\otimes} B)\ar[r] & H^i(\cX, \mathcal F\widehat{\otimes} B)\ar[r] & H^{i}(\cX, (\mathcal F/\mathcal G)\widehat{\otimes} B)\ar[r] & \cdots.}
\]
Consequently, to prove our proposition for $\mathcal F$, it suffices to show it for $\mathcal F/\mathcal G$. Therefore,  replacing $\cF$ by $\cF/\cG$ if needed, we assume that $\mathcal F$ is $p$-torsion free.

Let $\mathcal F_n=\mathcal F/p^{n}$.  By flat base change, for all $n\geq 0$, the natural morphisms
\[
H^i(\cX,\mathcal F_n)\otimes_A B\longrightarrow H^i(\cX,\mathcal F_n\otimes_A B), \quad i\in \mathbb Z,
\]
are isomorphisms. Passing to projective limits, we obtain isomorphisms:
\[
\alpha\colon \varprojlim_n\left(H^i(\cX,\mathcal F_n)\otimes_A B\right) \longrightarrow \varprojlim_n H^i(\cX,\mathcal F_n\otimes_A B), \quad i\in \mathbb Z.
\]
As $A$ is noetherian, the projective system $\left(H^{i-1}(\cX,\mathcal F_n)\right)_{n\geq 0}$ satisfies the (ML)-condition (\cite[Corollaire 3.4.4]{Gro61}), hence so does $\left(H^{i-1}(\cX,\mathcal F_n)\otimes_A B\right)_{n\geq 0}$. Thus
\[
R^1\varprojlim_n (H^{i-1}(\cX,\mathcal F_n)\otimes_A B)=R^1\varprojlim_n H^{i-1}(\cX,\mathcal F_n\otimes_A B)=0.
\]
Using the set of affine open formal subschemes of $\cX$ and \cite[Lemma 3.18]{Sch}, one checks that $R^j\varprojlim_n (\mathcal F_n\otimes_A B)=0$ whenever $j>0$. So $H^i(\cX,\cF\widehat{\otimes}_AB)\simeq H^i(\cX,R\varprojlim (\cF_n\otimes_A B))$, and the natural morphism
\[
\beta \colon H^i(\cX,\mathcal F\widehat{\otimes}_A B)\longrightarrow \varprojlim H^i(\cX,\mathcal F_n\otimes_A B).
\]
is surjective with kernel isomorphic to $R^1\varprojlim_n H^{i-1}(\cX,\mathcal F_n\otimes_A B)=0$ (Lemma \ref{jannsen}). In other words, $\beta$ is an isomorphism.

Next, consider the tautological exact sequence (recall that $\mathcal F$ has no $p$-torsion)
\[
0\longrightarrow \mathcal F\stackrel{p^{n}}{\longrightarrow} \mathcal F\longrightarrow \mathcal F_n\longrightarrow 0.
\]
We obtain a short exact sequence
\[
0\longrightarrow H^i(\cX,\mathcal F)/p^{n}\longrightarrow H^i(\cX,\mathcal F_n)\longrightarrow H^{i+1}(\cX,\mathcal F)[p^{n}]\longrightarrow 0,
\]
and thus the one below as $B$ is flat over $A$:
\begin{eqnarray*}
0\longrightarrow H^i(\cX,\mathcal F)\otimes_A B/p^{n} \stackrel{\gamma_n}{\longrightarrow} H^i(\cX,\mathcal F_n)\otimes_A B\longrightarrow H^{i+1}(\cX,\mathcal F)[p^{n}]\otimes_AB\longrightarrow 0.
\end{eqnarray*}
Because $H^{i+1}(\cX,\mathcal F)$ is an $A$-module of finite type and $A$ is noetherian, the $A$-submodule $H^{i+1}(\cX,\mathcal F)_{p\textrm{-tor}}\subset H^{i+1}(\cX,\cF)$ of elements killed by some power of $p$ is finitely generated over $A$. In particular, there exists $a\in \mathbb N$ such that $p^a$ kills $H^{i+1}(\cX,\mathcal F)_{p\textrm{-tor}}$ and thus $H^{i+1}(\cX,\mathcal F)[p^{n}]$ for all $n\in \mathbb N$. It follows that the transition map below is trivial for every $n$:
\[
H^{i+1}(\cX,\mathcal F)[p^{n+a}]\lra H^{i+1}(\cX,\mathcal F)[p^{n}], \quad x\mapsto p^a x.
\]
Thus the projective systems $(H^{i+1}(\cX,\mathcal F)[p^{n}])_n$ and $(H^{i+1}(\cX,\mathcal F)[p^{n}]\otimes_AB)_n$ satisfy the (ML)-condition. So one deduces an isomorphism $\gamma:=\varprojlim \gamma_n$
\[
\gamma \colon H^i(\cX,\mathcal F)\otimes_A B=\varprojlim \left(H^i(\cX,\mathcal F)\otimes_A B/p^{n}\right) \stackrel{\sim}{\longrightarrow} \varprojlim \left(H^i(\cX,\mathcal F_n)\otimes_A B\right).
\]
Here we have the first equality because $H^{i+1}(\cX,\cF)$ is of finite type over $A$, and $B$ is flat and $p$-adically complete (Lemma \ref{lem:elementary-commutative-algebra}).

Finally from the commutative diagram below
\[
\xymatrix{H^i(\cX,\mathcal F)\otimes_A B \ar[rr]^{\rm can}\ar[d]_{\gamma}^{\simeq} & & H^i(\cX,\mathcal F\widehat{\otimes}_A B) \ar[d]^{\beta}_{\simeq}\\ \varprojlim \left(H^i(\cX,\mathcal F_n)\otimes_A B\right) \ar[rr]^{\alpha}_{\simeq} & & \varprojlim H^i(\cX,\mathcal F_n\otimes_A B),}
\]
we obtain that the upper horizontal morphism is an isomorphism, as required.
\end{proof}

\begin{prop}\label{prop:technical-base-change} Let $A$ be a $p$-adically complete notherian ring. Let $B$ be an $A$-module that is $p$-adically complete. Assume that $A$ and $B$ are flat over $\mathbb Z_p$. Let $\cX\to \Spf(A)$ be a proper flat morphism between $p$-adic formal schemes. Let $\mathcal F$ be a bounded complex of coherent sheaves on $\cX$, such that for every term $\cF^i$ of $\cF$, $\cF^i[1/p]$ is locally a direct factor of a finite free $\cO_{\cX}[1/p]$-module.
\begin{enumerate}
\item For every $i\in \mathbb Z$, there is a natural map
\[
 H^i\left( R\Gamma(\cX, \mathcal F)\otimes_A^L B\right)\longrightarrow H^i(\cX, \mathcal F\widehat{\otimes}_A B)
\]
whose kernel and cokernel are killed by some power of $p$.
\item If moreover the finite $A[1/p]$-modules $H^j(\cX,\cF)[1/p]$, $j\in \mathbb Z$, are flat over $A[1/p]$, the kernel and the cokernel of the natural map
\[
H^i(\cX,\cF)\otimes B\lra H^i(\cX,\cF\widehat{\otimes}B)
\]
are annihilated by some power of $p$. In particular, we have isomorphisms
\[
H^i(\cX,\cF)\otimes B[1/p]\stackrel{\sim}{\lra} H^i(\cX,\cF\widehat{\otimes}B[1/p]), \quad \forall \ i\in \mathbb Z.
\]
\end{enumerate}
\end{prop}

\begin{proof} (1) We claim first that $B$ has a resolution $B^{\bullet}\ra B$ by $p$-adically complete and flat $A$-modules. Indeed, let $F^{\bullet}\ra B$ be a resolution of $B$ by free $A$-modules. As $A$ is flat over $\mathbb Z_p$, each $F^{i}$ is $p$-torsion free. Since $B$ is flat over $\Zp$, it is also $p$-torsion free. Therefore, the induced complex
\begin{equation}\label{eq:complex-B}
\cdots \lra F^{-1}/p\lra F^0/p\lra B/p\lra 0 \lra \cdots,
\end{equation}
and thus the complex
\[
\cdots \lra \hat{F}^{-1}\lra \hat{F}^0\lra B\lra 0\lra \cdots
\]
are exact. Therefore, we get a resolution $B^{\bullet}:=\hat{F}^{\bullet}$ of $B$ by flat $A$-modules (\cite[\href{https://stacks.math.columbia.edu/tag/06LE}{Lemma 06LE}]{Stacks}) that are $p$-adically complete (\cite[\href{https://stacks.math.columbia.edu/tag/05GG}{Lemma 05GG}]{Stacks}). In particular, by Lemma \ref{lem:flat-base-change}, we obtain an isomorphism in the derived category
\begin{equation}\label{eq:base-change-for-complete-flat-A-algebra}
R\Gamma(\cX, \cF)\otimes_A B^{\bullet}\stackrel{\sim}{\lra}R\Gamma(\cX,\cF\widehat{\otimes}B^{\bullet}).
\end{equation}

Consider the morphism $R\Gamma(\cX,\cX)\otimes^L B\ra R\Gamma(\cX, \cF\widehat{\otimes}B)$ in the derived category of abelian sheaves on $\cX$ defined by the the commutative diagram below
\begin{equation}\label{eq.diagram-mapping-cone}
\xymatrix{R\Gamma(\cX,\cF)\otimes^L B\ar[r]  & R\Gamma(\cX,\cF\widehat{\otimes}B) \\ R\Gamma(\cX,\cF)\otimes B^{\bullet}\ar[r]^{\simeq}_{\eqref{eq:base-change-for-complete-flat-A-algebra}}\ar[u]^{\simeq} & R\Gamma(\cX,\cF\widehat{\otimes}B^{\bullet}).\ar[u] }
\end{equation}
Here the left vertical map is an isomorphism as $B^{\bullet}\ra B$ is a flat resolution of $B$. To complete the proof of (1), it remains to show that the upper horizontal map of \eqref{eq.diagram-mapping-cone} induces a morphism
\begin{equation}\label{eq.diagram-upper-horizontal-cohomology}
 H^i\left( R\Gamma(\cX, \mathcal F)\otimes_A^L B\right)\longrightarrow H^i(\cX, \mathcal F\widehat{\otimes}_A B)
\end{equation}
whose kernel and cokernel are annihilated by some power of $p$. Using the naive truncations of $\cF$ and by induction on the length of the bounded complex $\cF$, we reduce to the case where $\cF$ is a complex concentrated in degree $0$, i.e., a coherent sheaf on $\cX$, such that $\cF[1/p]$ is locally a direct factor of a finite free $\cO_{\cX}[1/p]$-module.

Consider the complex
\begin{equation}\label{eq:cF-tensor-B}
\cdots \lra \cF\widehat{\otimes}_AB^{-1}\lra \cF\widehat{\otimes}_A B^0\lra \cF\widehat{\otimes}_A B\lra 0 \lra \cdots.
\end{equation}
We claim that there exists $N\in \mathbb N$ such that $p^N$ kills all its cohomologies. This is a local question on $\cX$, so assume that $\cF[1/p]$ is a direct factor of the finite free $\cO_{\cX}[1/p]$-module $\cO_{\cX}^d[1/p]$. Because $\cF$ is coherent, there exist morphisms $f:\cF\ra \cO_{\cX}^d$ and $g:\cO_{\cX}^d\ra \cF$ with $g\circ f=p^N\cdot \mathrm{id}_{\cF}$ for some $N\in \mathbb N$. By the functoriality of the complex \eqref{eq:cF-tensor-B} relative to $\cF$, we have the following commutative diagram
\[
\xymatrix{\cdots \ar[r] &  \cF\widehat{\otimes}_AB^{-1}\ar[r]\ar[d]^{f\widehat{\otimes}\mathrm{id}_{B^{-1}}} &  \cF\widehat{\otimes}_A B^0\ar[r] \ar[d]^{f\widehat{\otimes}\mathrm{id}_{B^{0}}} &  \cF\widehat{\otimes}_A B\ar[r] \ar[d]^{f\widehat{\otimes}\mathrm{id}_{B}} & 0 \ar[r] \ar[d]^0& \cdots \\ \cdots \ar[r] &  \cO_{\cX}^d\widehat{\otimes}_AB^{-1}\ar[r]  \ar[d]^{g\widehat{\otimes}\mathrm{id}_{B^{-1}}}&  \cO_{\cX}^d\widehat{\otimes}_A B^0\ar[r] \ar[d]^{g\widehat{\otimes}\mathrm{id}_{B^{0}}}&  \cO_{\cX}^d\widehat{\otimes}_A B\ar[r]\ar[d]^{g\widehat{\otimes}\mathrm{id}_{B}}& 0 \ar[r]\ar[d]^0 & \cdots \\ \cdots \ar[r] &  \cF\widehat{\otimes}_AB^{-1}\ar[r] &  \cF\widehat{\otimes}_A B^0\ar[r] &  \cF\widehat{\otimes}_A B\ar[r] & 0 \ar[r] & \cdots}.
\]
Because $(g\widehat{\otimes}\mathrm{id}_{B})\circ (f\widehat{\otimes}\mathrm{id}_{B})=p^N$ and $(g\widehat{\otimes}\mathrm{id}_{B^i})\circ (f\widehat{\otimes}\mathrm{id}_{B^i})=p^N$ for every $i$, to prove our claim, it suffices to show that the complex in the second row
\begin{equation}\label{eq:cO-tensor-B}
\cdots \lra \cO_{\cX}^d\widehat{\otimes}_AB^{-1}\lra \cO_{\cX}^d\widehat{\otimes}_A B^0\lra \cO_{\cX}^d\widehat{\otimes}_A B\lra 0 \lra \cdots
\end{equation}
is exact. Since $\cX$ is flat over $A$, we obtain from \eqref{eq:complex-B} a similar exact sequence
\begin{equation*}%\label{eq:complex-B}
\cdots \lra \cO_{\cX}^d\otimes_AB^{-1}/p\lra \cO_{\cX}^d\otimes_AB^0/p\lra \cO_{\cX}^d\otimes_AB/p\lra 0\lra \cdots.
\end{equation*}
Because the sheaves $\cO_{\cX}^d\otimes_A B^i$'s and $\cO_{\cX}^d\otimes_A B$ are $p$-torsion free, we deduce as above the exactness of \eqref{eq:cO-tensor-B}, completing the proof of our claim.

Let $C$ be the complex of abelian sheaves concentrated in degrees $\leq 0$ given by the following distinguished triangle
\[
\cF\widehat{\otimes}B^{\bullet} \lra \cF\widehat{\otimes}B\lra C\stackrel{+1}{\lra}.
\]
By what we have shown above, there exists some $N\in \mathbb N$ such that $p^{N}\cdot \mathcal H^i(C)=0$ for every $i\in \mathbb Z$. In particular, the cohomology groups of $R\Gamma(\cX, C)$, which is also the mapping cone of the right vertical map of \eqref{eq.diagram-mapping-cone}, are annihilated by some power of $p$. Consequently, the kernel and the cokernel of the map \eqref{eq.diagram-upper-horizontal-cohomology} are killed by some power of $p$, as required by (1).

(2) Consider the spectral sequence
\[
E_2^{a,b}=\mathrm{Tor}_A^{-a}(H^b(\cX,\cF),B)\Longrightarrow  H^{a+b}\left(R\Gamma(\cX,\cF)\otimes_A^L B\right).
\]
Observe that, for $M$ a finite $A$-module such that $M[1/p]$ is flat, thus locally free, over $A[1/p]$, $\mathrm{Tor}_{A}^{-a}(M, B)$ is killed by some power of $p$ whenever $a<0$. In particular, $E_{2}^{a,b}$ is killed by some power of $p$ for $a<0$. Thus, the kernel and the cokernel of
\[
E_{2}^{0,i}=H^{i}(\cX,\cF)\otimes B\lra H^{i}\left(R\Gamma(\cX,\cF)\otimes^L B\right)
\]
are annihilated by some power of $p$. Combining the first statement of our proposition, we obtain that the natural map
\[
H^i(\cX,\cF)\otimes B\lra H^i(\cX,\cF\widehat{\otimes}B)
\]
has its kernel and cokernel killed by some power of $p$. Inverting $p$. we deduce
\[
H^i(\cX,\cF)\otimes B[1/p]\stackrel{\sim}{\lra} H^i(\cX,\cF\widehat{\otimes}B)[1/p]\simeq   H^i(\cX,\cF\widehat{\otimes}B[1/p]),
\]
as desired in (2).
\end{proof}

%%%%%%%%%%%%%%%%%%%%%%%%%%%%%%%%%%%%

%%%%%%%%%%%%%%%%%%%%%%%%%%%%%%%%%%%%%%%%%%%%%
\vspace{\baselineskip}

Fucheng Tan

Research Institute for Mathematical Sciences

Kyoto University

Kitashirakawa-Oiwakecho, Sakyo-ku

Kyoto 606-8502, Japan

Email: ftan@kurims.kyoto-u.ac.jp

\vspace{\baselineskip}
 \vspace{\baselineskip}

Jilong Tong

School of Mathematical Sciences

Capital Normal University

105 Xi San Huan Bei Lu

Beijing 100048, China

Email: jilong.tong@cnu.edu.cn

\end{document}